\documentclass[oneside]{amsart}
\newcommand{\formatswitch}{preprint}

\usepackage{amssymb}
\usepackage{ifthen}
\usepackage{verbatim}
\usepackage{epsfig}
\usepackage[all]{xy}

\newcommand{\tref}[1]{(\ref{#1})}

\ifthenelse{\equal{\formatswitch}{big}}{
\setlength{\textwidth}{432pt}
\setlength{\oddsidemargin}{18.8775pt}

\setcounter{tocdepth}{1}
}{
\ifthenelse{\equal{\formatswitch}{paper}}{

\setcounter{tocdepth}{1}
}{
\setcounter{tocdepth}{2}
}
}
\DeclareMathAlphabet\EuScript{U}{eus}{m}{n}
\DeclareMathAlphabet\EuScriptb{U}{eus}{b}{n}

\newcommand{\sscr}[1]{\EuScript{#1}}

\newcommand{\claimenum}{\renewcommand{\theenumi}{\alph{enumi}}
 \renewcommand{\labelenumi}{\textit{(\theenumi)}}
 \renewcommand{\theenumii}{\roman{enumii}}
 \renewcommand{\labelenumii}{\textit{(\theenumii)}}
 \begin{enumerate}}
\newcommand{\claimenumend}{\end{enumerate}}

\newcommand{\romanenum}{\renewcommand{\theenumi}{\roman{enumi}}
 \renewcommand{\labelenumi}{\textit{(\theenumi)}}
 \renewcommand{\theenumii}{\alph{enumii}}
 \renewcommand{\labelenumii}{\textit{(\theenumii)}}
 \begin{enumerate}}
\newcommand{\romanenumend}{\end{enumerate}}

\newtheorem{dummy}{realdumb}[section]
\newtheorem{thm}{Theorem}

\newtheorem{lemma}[dummy]{Lemma}

{\theoremstyle{definition} }
\newtheorem{rmk}[dummy]{Remark}
{\theoremstyle{definition} }
\newtheorem{cor}{Corollary}[dummy]

\renewcommand{\text}{\mathrm}

\newcommand{\strutdepth}{\dp\strutbox}
\newcommand{\marginalnote}[1]
   {\strut\vadjust{\kern-\strutdepth\domarginalnote{#1}}}
\newcommand{\domarginalnote}[1]{\vtop to \strutdepth{
  \baselineskip\strutdepth
   \vss\llap{ #1\ \ }\null}}  
\newcounter{showlabelflag}
\newcounter{makelabelflag}
\newcommand{\showlabels}{\setcounter{showlabelflag}{1}}

\newcommand{\makelabels}{\setcounter{makelabelflag}{1}}
\newcommand{\hidelabels}{\setcounter{showlabelflag}{2}}
\newcommand{\mylabel}[1]{
  \ifthenelse{\value{makelabelflag}=1}
    {\label{#1}}{}
  \ifthenelse{\value{showlabelflag}=1}
    {\marginpar{#1}}{}\relax}

\newcommand{\R}{{\mathbf R}}
\newcommand{\Q}{{\mathbf Q}}
\newcommand{\Z}{{\mathbf Z}}

\ifthenelse{\equal{\formatswitch}{draft}}{
\showlabels
}{\relax}

\usepackage{pmat}

\newcommand{\scr}{\sscr}

\newcommand{\mymargin}[1]{
  \ifthenelse{\value{showlabelflag}=1}
    {\marginpar{#1}}{}\relax}

\newcounter{enumo}

\newcommand{\RRsh}{\kern -1 pt \Rsh}

\newcounter{keepitemnum}

\newcounter{keepitemnumm}

\begin{document}

\bibliographystyle{amsplain}

\makelabels
\hidelabels

\renewcommand{\theenumi}{\roman{enumi}}

\begin{center}On Shavgulidze's Proof of the Amenability \\
of some Discrete Groups of Homeomorphisms \\
of the Unit Interval \end{center}

\begin{center}by Various\end{center}

\vspace{5pt}

\begin{center}\today\end{center}

\vspace{5pt}

Primarily, these notes have been created by the participants of a
seminar formed to go through the English language version, available
on the arXiv, of the paper \cite{shavgulidze} whose main result
implies the amenability of Thompson's group \(F\).  The seminar has
been running sporadically since July 9, 2009.  I (Matt Brin) have
been acting as recorder for the seminar.  A short paper
\cite{MR2486813} announcing and summarizing the technieques of
\cite{shavgulidze} is available, but will not be covered in these
notes.

Questions have arisen during our readings that have been answered
via email by several people from outside the seminar.  At least one
of our outside consultants is in touch with Shavgulidze, and so we
have gotten, indirectly, some of Shavgulidze's elaborations on some
of the points in his paper.  What follows is an alphabetical list of
all that are in the seminar as well as those outside that we have
been in touch with.  As time goes on and contributers are added, the
list will surely grow longer.
Azer Akhmedov,
Vadim Alekseev,
Matt Brin,
Ross Geoghegan,
Victor Guba,
Fernando Guzm\'an,
Marcin Mazur,
Tairi Roque,
Lucas Sabalka,
Mark Sapir,
Candace Schenk,
Anton Schick,
Matt Short,
Marco Varisco,
Xiangjin Xu.

It is the intention to update the notes as more of the paper is
digested.  Contributions from others is encouraged, but with some
conditions.  First, I (Matt Brin) need to understand the
contribution.  This is a heavy condition since I am unfamiliar with
most of these techniques.  The level of detail in what follows gives
a hint as to the level of detail that I need before I can claim to
understand anything.  Second, all that a contribution will get you
is that your name will be added to the list in the previous
paragraph.  If you have something truly original that you want your
name attached to, then you had best find your own public venue for
it.

I have been sending these notes out periodically to a short mailing
list.  I will stop doing that and just send out brief notifications
when this posting is updated.

\section{Amenability}

A group \(G\) is {\itshape amenable} if there is a measure
consisting of a function \[\mu:P(G)\rightarrow [0,1]\] where
\begin{enumerate} 
\item \(P(G)\) is the set of all subsets of \(G\),
\item \(\mu(G)=1\),
\item if \(A_1, A_2, \ldots, A_n\) are pairwise disjoint, then 
\[\mu\left (\bigcup_{i=1}^n A_i\right) = \sum_{i=1}^n \mu(A_i),\]
and
\item for all \(A\subseteq G\) and all \(g\in G\) we have
\[\mu(Ag) = \mu(A)\] where \(Ag = \{ag\mid a\in A\}\).
\end{enumerate}

Significances of the above are (1) the measure is defined on
{\itshape all} subsets of \(G\), (2) it is non-trivial and bounded,
(3) it is finitely additive, and (4) it is translation invariant.

All finite groups are obviously measurable.  If \(|G|=n\), then let
every singleton have measure \(1/n\) and extend by (3).

The definition above does not explain the name.  If \(G\) is a
group, let \(B(G)\) be the set of all functions \(f:G\rightarrow
\R\) so that each function is bounded (each \(f\in B(G)\) has a
compact interval \(I_f\subseteq \R\) with \(f(G)\subseteq I_f\)).
The group \(G\) acts on \(B(G)\) by \((gf)(h) = f(hg^{-1})\).  It is
a straightforward exercise that \(G\) is amenable if and only if
there is a function \(\mu':B(G)\rightarrow \R\) satisfying the
following.
\begin{enumerate}
\item For \(f\in B(G)\) if \(f(G)\subseteq I_f\) for a compact
interval \(I_f\subseteq \R\), then \(\mu'(f)\in I_f\).
\item The function \(\mu'\) is linear in that for all \(f_1,f_2\in
B(G)\) and \(r,s\in \R\), we have \[\mu'(rf_1+sf_2) = r\mu'(f_1) +
s\mu'(f_2).\]
\item The function \(\mu'\) is translation invariant in that for all
\(f\in B(G)\) and \(g\in G\) we have \(\mu'(gf) = \mu'(f)\).
\end{enumerate} 

Item (1) says that \(\mu'(f)\) must lie between the inf and sup of
\(f\).  In particular, \(\mu'(f)=C\) when \(f\) is the constant
function to \(C\).  One refers to \(\mu'\) as a {\itshape mean}
(i.e., average) of the bounded functions on \(G\).  Thus the
amenability of a group is equivalent to the exitence of a mean on
its bounded real functions.  The word ``amenable'' was attached to
the definition as a pun by Mahlon M.~Day \cite{MR19:1067c}.
Amenable groups lead to nice Hilbert spaces and so the pun was
chosen to express the niceness of the property.

A celebrated combinatorial condition on a group, known as the
F\o{}lner criterion \cite{folner}, is equivalent to amenability.
However, this criterion is not used by Shavgulidze.  His proof
proceeds by constructing the required mean.  Other than a brief
mention in the next few paragraphs, the F\o{}lner criterion will not
be discussed here.

It has been mentioned that finite groups are amenable.  Infinite
amenable groups exist.  The first known such was \(\R/\Z\)
\cite[Ch. II, \S 3(1)]{MR880204}.  The proof (due to Banach) was the
first application of what came to be known as the Hahn-Banach
theorem \cite[Theorem 1, P. 18]{MR880204}.  Thus the axiom of choice
was involved.

It was shortly noticed that Banach's proof extended to all abelian
groups (was this noticed by von Neumann?) and then it was observed
by von Neumann \cite{vonneumann} that the class of amenable groups
was closed under the operations of (1) taking subgroups, (2) taking
quotients, (3) taking extensions, and (4) taking direct limits.  The
smallest class of groups containing all finite and all amenable
groups and closed under (1--4) was called (by Day?) the class of
{\itshape elementary amenable} groups.

In spite of the large class of groups that were demonstrably
amenable, all proofs (other than for finite groups) up to the
appearance of the F\o{}lner criterion were based on the power of the
Hahn-Banach theorem, and thus the axiom of choice, even for as nice
a group as the integers.  The proof using the F\o{}lner criterion
that \(\Z\) is amenable takes about one line.

It was also observed by von Neumann \cite{vonneumann} that \(F_2\),
the free group on two generators, is not amenable.\footnote
{von Neumann was looking at the Banach-Tarski paradox.  He
observed that a ``paradox'' of the Banach-Tarski type was a property
of a group action and he proved that a certain group property, later
called amenability, was equivalent to the inability of a group to
participate in a paradoxical action.  He pointed out that the
existence of the paradox in 3 dimensions comes from the fact that
the isometry group of \(E^3\) is not amenable because it contains a
subgroup isomorphic to the free group on two generators.}
If we let \(F_2\) be freely generated by \(x\) and \(y\) and the
elements of \(F_2\) be represented by reduced words in \(x, y,
x^{-1}\) and \(y^{-1}\), we can define four sets as follows.  The
set \(X\) consists of all reduced words that end in \(x\),
\(X^{-1}\) is the set of all reduced words that end in \(x^{-1}\)
and similarly for \(Y\) and \(Y^{-1}\).  These four sets and
\(\{1\}\) disjointly cover all of \(F_2\).

We observe \[\begin{split} (X\cup Y\cup Y^{-1} \cup\{1\})x
&\subseteq X, \\ (X^{-1}\cup Y\cup Y^{-1} \cup\{1\})x^{-1}
&\subseteq X^{-1}, \\ (X\cup X^{-1}\cup Y \cup\{1\})y &\subseteq Y,
\\ (X\cup X^{-1}\cup Y^{-1} \cup\{1\})y^{-1} &\subseteq Y^{-1}.
\end{split}\] It is immediate that a singleton in an infinite group
has measure zero, and it is just as immediate from the facts above
that each of the four infinite sets discussed has measure zero.
Thus the entire group has measure zero contradicting one of the
requrements.

From von Neumann's observations, any group containing a subgroup
isomorphic to \(F_2\) cannot be amenable.  It is known that
Thompson's group \(F\) cannot contain a subgroup isomorphic to
\(F_2\) \cite{brin+squier, CFP}.  It has been a well known open
question for a few decades as to whether \(F\) is amenable.

It is elementary that \(F\) is not elementary amenable.  Results of
Chou \cite{chou} say the following.  Let \(EG_0\) be the class of
groups that are either finite or abelian, and define inductively for
an ordinal \(\alpha\) the class \(EG_\alpha\) to be the class of
groups obtained from groups in classes \(EG_\beta\) with
\(\beta<\alpha\) using the operations (3) extension, and (4) direct
limits mentioned above.  Note that taking subgroups or quotients is
not to be used.  Then each \(EG_\alpha\) is closed under (1) taking
subgroups and (2) taking quotients, and further the class of
elementary amenable groups is the union of the \(EG_\alpha\).  To
rule out an appearance of \(F\) in one of the \(EG_\alpha\), we need
three facts.  First, \(F\) is finitely generated which implies that
if \(F\) is a direct limit of groups, then one of the groups in the
limit will have \(F\) as a quotient.  Second, any non-trivial normal
subgroup of \(F\) contains subgroups that are isomorphic to \(F\).
This shows that if \(F\) is in some \(EG_\alpha\) with \(\alpha>0\),
then it must already be in some \(EG_\beta\) for some
\(\beta<\alpha\).  The third fact (or pair of facts) is that \(F\)
is neither finite nor abelian and is thus not in \(EG_0\).

\section{Thompson's group \protect\(F\protect\)}

There are several ways to define Thompson's group \(F\).  The one
that is closest to what is needed for this discussion is the easiest
and least revealing algebraically.  We define \(F\) to be the group
(with group operation composition) of those homeomorphisms
\(h:[0,1]\rightarrow [0,1]\) satisfying the following.
\begin{enumerate}
\item \(h\) is piecewise linear (PL) in that its graph consists of a
finite number of straight line segments.
\item The slopes of \(h\), where defined, are of the form \(2^n\),
\(n\in \Z\). 
\item The points in \([0,1]\) where the slope of \(h\) is not
defined are confined to the dyadic rationals (those points of the
form \(m/2^n\) for \(m,n\in \Z\)).
\end{enumerate}

We usually like to have elements of \(F\) act on the right, but to
agree with the papers we will be quoting, we reluctantly adopt the
convention that \(F\) acts on the left and composes from right to
left.

Note that (2) implies that all \(h\in F\) are increasing and so
preserve orientation.

The operation of differentiation is not defined for all \(t\in
[0,1]\) for non-identity elements of this definition of \(F\).
However, it is defined on all but finitely many points and given an
\(f\in F\) we can integrate \(f'\) quite successfully to reconstruct
\(f\) from \(f'\).  It follows that if \(f,g\in F\) are not equal,
then they have derivatives that are somewhere not equal.  Since the
values taken on by the derivatives are all integral powers of 2, it
follows that this version of \(F\) satisfies the following.
\mymargin{TheHypTwo}\begin{equation}\label{TheHypTwo}
\forall f\ne g\in F,\,\,\exists t\in [0,1]\,\, 
\bigg( |\log(f'(t))-\log(g'(t))|\ge\log(2)\bigg).
\end{equation}

This will match with one of the key hypotheses in the proof that
\(F\) is amenable.  However, another hypothesis will require that
all of the elements of \(F\) be at least three times continuously
differentiable.  Thus the version of \(F\) above will not do.

The following is a combination and slight extension (extracted from
the proofs) of two results, Theorems 1.13 and 2.3, from
\cite{ghys+serg}.

\begin{thm}\mylabel{GhysSer} For each integer \(r\) with \(1\le r\le
\infty\) and each \(C>0\) there is a monomorphism \(\theta\) of
\(F\) into \(\text{Diff}^r([0,1])\) that satisfies
\[
\forall f\ne g\in F,\,\,\exists t\in [0,1]\,\, 
\bigg( |\log((\theta f)'(t))-\log((\theta g)'(t))|\ge C\bigg).
\]
\end{thm}

What follows is a slight rewording of the proof from
\cite{ghys+serg}.  There are a series of definitions and lemmas to
do first.

We first simplify the conclusion.  We write \(f\) and \(g\) rather
than \(\theta f\) and \(\theta g\) to keep the notation simple.  We
have \[
\begin{split}
|\log(f'(t))-\log(g'(t))| 
&= |\log((f'(t))(g'(t))^{-1})| \\
&= |\log((f'(t))((g^{-1})'(g(t))))| \\
&= |\log((fg^{-1})'(g(t)))|
\end{split}
\]
Thus the conclusion of Theorem \ref{GhysSer} holds if and only if
the following holds.
\mymargin{CondARev}\begin{equation}\tag{b}\label{CondARev}
\forall f\ne 1\in F,\,\, \exists t\in [0,1]\,\, \bigg( |\log((\theta
f)'(t))| \ge C \bigg).
\end{equation}

The proof is based on the fact that the straight line pieces of the
graphs of elements of \(F\) come from a rather nice group.  The
re-embedding of \(F\) comes from a re-embedding of the group of
straight line pieces.  Let \(\Q_2\) denote the group of dyadic
rationals---rational numbers of the form \(p/2^q\) with both \(p\)
and \(q\) from \(\Z\).

Now let \(GA(\Q_2)\) be the group of affine transformations of
\(\Q_2\) of the form 
\mymargin{GAElt}\begin{equation}\label{GAElt}
x\mapsto 2^n x+p/2^q.
\end{equation}
The map in \tref{GAElt}
will be denoted by the pair \((2^n, p/2^q)\).

We let \(PL_2(\R)\) denote the self homeomorphisms \(f\) of \(R\)
that are piecewise linear (which implies that every point has a
neighborhood which has only finitely many points of discontinuity of
\(f'\)), and for which every point of continuity of \(f'\) has a
neighborhood on which \(f\) agrees with an element of \(GA(\Q_2)\).
Thus \(PL_2(\R)\) is the group of transformations of \(R\) that are
``piecewise \(GA(\Q_2)\).''  We have a homomorphic inclusion of
\(GA(\Q_2)\) into \(PL_2(\R)\).

We will need to refer to the structure of the group \(GA(\Q_2)\), so
we describe it in detail.  

If \(r\) is a dyadic rational, we use \(T_r\in GA(\Q_2)\) to denote
the translation by \(r\), and sending \(r\) to \(T_r\) is a
homomorphic embedding of \(\Q_2\) in \(GA(\Q_2)\).  We use \(D\) to
denote the doubling map \(x\mapsto 2x\).

For any \(r\in \Q_2\), we have \(DT_r=T_{2r}D\).  From this we have
\(T_{2r} = DT_rD^{-1}\) and from this \[T_{2^{q}} =
D^{q}T_1D^{-q}\] holds for all integral values of \(q\) or,
equivalently, \[T_{2^{-q}} =
D^{-q}T_1D^{q}.\]

If \(r=p/2^q\), then \[T_r =
T^p_{2^{-q}} = D^{-q}T^p_1D^{q} = D^{-q}T_pD^{q}.\] Since \(r\mapsto
T_r\) is a homomorphic embedding of \(\Q_2\) in \(GA(Q_2)\), we have
that this homomorphism can also be expressed by \[\frac{p}{2^q}
\mapsto D^{-q}T_1^pD^q = D^{-q}T_pD^q.\]
From \tref{GAElt} we see that the element \((2^n, p/2^q)\) of
\(GA(\Q_2)\) is given by 
\mymargin{AffineForm}\begin{equation}\label{AffineForm}
(2^n, p/2^q) = \left(D^{-q}T^p_1
D^q\right)D^n.
\end{equation}

We are now ready to show that \(GA(Q_2)\) is isomorphic to the
Baumslag-Solitar group \(B(1,2) = \langle t,d\mid
dtd^{-1}=t^2\rangle\).  We start with a general lemma since we will
need it again later.

\begin{lemma}\mylabel{BSChar} If a group \(G\) is generated by two
elements \(D\) and \(T\) that satisfy \(DTD^{-1}=T^2\) then every
element in \(G\) is represented by a word in the form
\((D^iT^pD^{-i})D^n\) where \(p\) is odd.  Further, if the words
\((D^iT^pD^{-i})D^n\) represent different elements when the triples
\((n,p,i)\) of integers with \(p\) odd are different, then sending
\(d\) to \(D\) and \(t\) to \(T\) extends to an isomorphism from
\(B(1,2)\) to \(G\).  \end{lemma}

\begin{proof} We work first in \(B(1,2)\) since its only defining
relation is \(dtd^{-1}=t^2\).

In \(B(1,2)\), define \(t_i=d^itd^{-i}\) for \(i\in
\Z\).  Since \(t_i=t^{2^i}\) for \(i\ge0\), we know that the
\(t_i\), \(i\ge0\), commute pairwise, and from that it follows that
all the \(t_i\) commute.  

It is standard (in any group) that any word in \(\{t,d\}\) and their
inverses is a product of conjugates of \(t\) by powers of \(d\)
followed by a power of \(d\).  Thus, in \(G\), any word is a product
of various \(t_i\) followed by a power of \(d\).  If \(i\) is the
smallest subscript in the product, then every other conjugate in that
product will be a power of \(t_i\).  Thus an arbitrary word is
equivalent to one of the form
\mymargin{BSForm}\begin{equation}\label{BSForm}
\left(d^it^pd^{-i}\right)d^n.
\end{equation}

If \(p\) is even and
\(p=2k\), then the expression can be altered by
\[\begin{split}
\left(d^it^{2k}d^{-i}\right)d^n &= 
\left(d^i(t^2)^kd^{-i}\right)d^n \\ &= 
\left(d^it^2d^{-i}\right)^kd^n \\ &= 
\left(d^{i-1}td^{-i+1}\right)^kd^n \\ &= 
\left(d^{i-1}t^kd^{-i+1}\right)d^n.
\end{split}
\]

Thus every word can be reduced to one of the form \tref{BSForm}
where \(p\) is odd.  This applies to any group satisfying the
defining relation of \(B(1,2)\) and so applies to \(G\).  This
verifies the first claim.

Since \(DTD^{-1}=T^2\), the assignment \(t\mapsto T\) and \(d\mapsto
D\) extends to an epimorphism \(\psi:B(1,2)\rightarrow G\).  But
given what we have proven, the hypotheses of the second claim imply
that this is a monomorphism.  \end{proof}

\begin{cor} Taking \(t\) to \(T_1\) and \(d\) to \(D\) extends to an
isomorphism \(\psi\) from \(B(1,2)\) to \(GA(Q_2)\).  \end{cor}

\begin{proof} From \tref{AffineForm} we know that \(\psi\) is an
epimorphism, and we know that a word in the form \tref{BSForm} is
taken by \(\psi\) to \((2^n, p2^{i})\).  However differening triples
\((n,p,i)\) with \(p\) odd give different elements of \(GA(Q_2)\)
since different values of \(n\) give different slopes and different
pairs \((p,i)\) with \(p\) odd give different values of \(p2^i\),
the \(y\)-intercept.  \end{proof}

We now re-embed \(PL_2(\R)\) in \(\text{Homeo}_+(\R)\), the group of
increasing self homeomorphisms of \(\R\), by first re-embedding
\(GA(\Q_2)\) in \(\text{Homeo}_+(\R)\).  The re-embedding of
\(GA(\Q_2)\) will be done by replacing \(D\) by another function
\(f\) so that \(T_1\) and \(f\) generate a copy of \(B(1,2)\) in a
manner identical to \(T_1\) and \(D\).  There is a small set of
properties that \(f\) will have to satisfy in order to do this, and
the flexibility in choosing this \(f\) will allow us to get extra
properties of the embedding by adding extra conditions to \(f\).  In
particular we will get that the image of \(PL_2(\R)\) in
\(\text{Homeo}_+(\R)\) can be made arbitrarily smooth and that given
any \(C>0\), condition \tref{CondARev} can be satisfied.

Let \(f\) be an element of \(\text{Homeo}_+(\R)\) that satisfies (I)
and (II) below.  
{\renewcommand{\theenumi}{\Roman{enumi}}
\begin{enumerate}
\item  For every real \(x\), we have \(f(x+1)=f(x)+2\).
\item \(f(0)=0\).
\end{enumerate}
}

In the following, we will always assume that (I) and (II) are
satsified.

We exploit the fact that for \(r\in \Q_2\), we have \(r=T_r(0)\).
For \(r=p/2^q\in \Q_2\), we note \[r=T_r(0) =
D^{-q}T_pD^q(0).\] With \(r\) as just given define
\mymargin{RBarDef}\begin{equation}\label{RBarDef} \overline{r} =
f^{-q}T_pf^q(0).  \end{equation}

(We will ignore the fact that \(D(0)=f(0)=0\) unless it becomes
convient to notice it.  When we do notice it, we will see that
\(\overline{r} = f^{-q}T_p(0) = f^{-q}(p)\).)

\begin{lemma} The map \(r\mapsto \overline{r}\) from \(\Q_2\) to
\(\R\) is well defined, strictly increasing, fixes the integers
pointwise, and commutes with \(T_1\).  \end{lemma}

\begin{proof} For well definedness, it suffices to show that
\(\overline{p/2^q} = \overline{2p/2^{q+1}}\).  This asks that
\[f^{-q}T_pf^q = f^{-q-1}T_{2p}f^{q+1}\] or \[T_p=f^{-1}T_{2p}f.\]
This becomes \(fT_p=T_{2p}f\) which is just \(f(x+p)=f(x)+2p\) which
follows from (I).

For the last claim, we note that if
\(r=D^{-q}T_pD^r(0)\), then \[\begin{split}
r+1 &= T_1D^{-q}T_pD^q(0) \\ 
&= D^{-q}T_{2^q}T_pD^q(0) \\ 
&= D^{-q}T_{2^q+p}D^q(0)
\end{split}\]
so that 
\[
\begin{split}
\overline{r+1}
&= f^{-q}T_{2^q+p}f^q(0) \\
&= f^{-q}T_{2^q}T_pf^q(0) \\
&= T_1f^{-q}T_pf^q(0) \\
&= \overline{r} + 1.
\end{split}\]

Using well definedness, we can represent two given elements in
\(\Q_2\) with the same denominator.  It is now convenient to notice
that \(\overline{p/2^q} =f^{-q}(p)\).
That \(\overline{p/2^q}< \overline{p'/2^q}\) when \(p<p'\) follows
from the fact that \(f\) is an increasing self homeomorphism of
\(\R\).

Lastly, when \(r=p\), an integer, then \(q=0\) in
\(\overline{p}=f^{-q}T_pf^q(0)\) and we get \(\overline{p}=p\).
\end{proof}

\begin{lemma} Sending \(T_1\) to itself and \(D\) to \(f\) induces a
homomorphic embedding \(\theta_f:GA(Q_2)\rightarrow
\text{Homeo}_+(\R)\).  \end{lemma}

\begin{proof}  It suffices to show that sending \(t\) to \(T_1\) and
\(d\) to \(f\) extends to an isomorphism from \(B(1,2)\) to the
group \(G\) generated by \(T_1\) and \(f\).

Item (I) implies that \(T_1^2 = T_2 = f T_1 f^{-1}\).  Thus what we
have to show is that different words of the form
\(W=(f^{-q}T_1^pf^q)f^n\) with \(p\) odd correspond to different
elements of \(G\).

We have \(W(0)=f^{-q}(p) = \overline{p/2^q}\) and we know that these
differ as long as the values of \(p/2^q\) differ.

We have \[
\begin{split}
W(1) 
&= f^{-q}T_pf^{q+n}(1) \\
&= f^{-q}T_p(2^{q+n}) \\
&=f^{-q}(p+2^q2^n) \\
&=f^{-q}(p)+2^n
\end{split}\]
where the second and last equalities follow from the fact
\(f^q(x+m)=f^q(x)+2^qm\) that is easily derived from (I).

This is sufficient information to give the conclusion.  \end{proof}

Recall that sending \(r\in \Q_2\) to \(T_r\) homomorphically embeds
\(Q_2\) in \(GA(Q_2)\).  We regard \(Q_2\) as a subgroup of
\(GA(Q_2)\) for the next statement.

\begin{cor} The restriction of \(\theta_f\) to \(\Q_2\) takes
\(p/2^q\) to \(f^{-q}T_pf^q\) and is a homomorphic
embedding of \(\Q_2\) into \(\text{Homeo}_+(\R)\).  \end{cor}

We can gather some notational trivialities.

\begin{rmk}\mylabel{QTwoEmb} For \(r\in Q_2\), we have \(\theta_f(r)
= \theta_f(T_r)\).  In addition \(\overline{r} = \theta_f(r)(0) =
\theta_f(T_r)(0)\).  For \(p\in \Z\), we have \(\theta_f(p) =
\theta_f(T_p)=T_p\) and \(\overline{p} = \theta_f(p)(0) =
\theta_f(T_p)(0) = T_p(0)=p\).  \end{rmk}

The next is almost as trivial.

\begin{lemma}\mylabel{PreCont} If \(h\in GA(Q_2)\) takes \(x\in
Q_2\) to \(y\), then \(\theta_f(h)\) takes \(\overline{x}\) to
\(\overline{y}\).  \end{lemma}

\begin{proof} We have that \(h\) is some \((2^n, p/2^q)\) or
\(h=(D^{-q}T_1^p D^q)D^n\) and \(x\) is some \(i/2^j\) or
\(x=D^{-j}T_1^i D^j(0)\).  Thus \[y=h(x) = (D^{-q}T_1^p D^q)D^n
D^{-j}T_1^i D^j(0).\] If we denote the word in \(D\) and \(T_1\) on
the right by \(W(D, T_1)\), then we have \(y=W(D,T_1)(0)\).  From
Lemma \ref{BSChar}, we know that in \(GA(Q_2)\) the word
\(W(D,T_1)\) reduces to a word in the form \((D^{-k}T_1^mD^k)D^u\)
so \[y= (D^{-k}T_1^mD^k)D^u(0) = (D^{-k}T_1^mD^k)(0).\]

If we let \(\overline{W}(f,T_1)\) be obtained from \(W(D,T_1)\) by
replacing every appearance of \(D\) by \(f\), then we know first
that \(\overline{W}(f,T_1)(0)\) gives \(\theta_f(\,\overline{x}\,)\)
by definition, and second we know that \(\overline{W}(f,T_1)(0)\)
reduces to \[\overline{y}= (f^{-k}T_1^mf^k)f^u(0) =
(f^{-k}T_1^mf^k)(0)\] because taking \(D\) to \(f\) and \(T_1\) to
itself is an isomorphism from \(GA(Q_2)\) to its image under
\(\theta_f\).  \end{proof}

\begin{cor} If \(h\in GA(Q_2)\) fixes an integer \(p\), then
\(\theta_f(h)\) fixes \(p\).  \end{cor}

\subsection{Extending \protect\(\theta_f\protect\) to 
\protect\(PL_2(\R)\protect\)}

Just as \(PL_2(\R)\) consists of functions made from pieces of
functions from \(GA(\Q_2)\), we extend \(\theta_f\) to embed all of
\(PL_2(\R)\) into \(\text{Homeo}_+(\R)\) by building the functions
in the image \(\theta_f(PL_2(\R))\) from pieces of functions from
\(\theta_f(GA(\Q_2))\).

Let \(h\) be in \(PL_2(\R)\).  There is a sequence
\((x_n)_{n\in\Z}\) in \(\Q_2\) with no accumulation point in \(R\)
and a sequence of functions \(\gamma_n\in GA(\Q_2)\) so that for
each \(n\) we have \[h|_{[x_n, x_{n+1}]} = \gamma_n|_{[x_n,
x_{n+1}]}.\]
The sequence \((x_n)\) is not unique for a given \(h\) since we can
always add more points.  We could ask for a smallest such sequence,
but that will not be necessary.

For this \(h\in PL_2(\R)\), we define \(\theta_f(h)\) in pieces.  It
will then have to be shown that the result is continuous.

Define \(\theta_f(h)\) so that \[
\theta_f(h)|_{[\,\overline{x}_n, \overline{x}_{n+1})} = 
\theta_f(\gamma_n)|_{[\,\overline{x}_n, \overline{x}_{n+1})}.\]
It is clear that this is well defined for a given sequence \((x_n)\)
on whose complement \(h'\) is defined.  Given two such sequences, we
can get a common ``refinement'' by taking their union, so we get
that the definition is independent of the choice of sequence
\((x_n)\) if it is shown to be invariant under the addition of a
finite number of points in a given neighborhood.  But if \(h\)
agrees with a given \(\gamma_n\) on two intervals, then the same
\(\theta_f(\gamma_n)\) is used on both intervals.  If the intervals
abut, then the result is \(\theta_f(\gamma_n)\) on the union of the
two intervals.  Thus \(\theta_f(h)\) is independent of the choice of
the sequence \((x_n)\).

It is also clear that the restriction of this \(\theta_f\) to
\(GA(\Q_2)\) agrees with the previous defintion of \(\theta_f\).

\subsection{Properties of the extension}

We first deal with continuity.

\begin{lemma}\mylabel{ThetaToHomeo} If \(h\in PL_2(\R)\), then
\(\theta_f(h)\) is a self homeomorphism of \(\R\).  \end{lemma}

\begin{proof} Since \(x\mapsto \overline{x}\) is order preserving
and commutes with adding 1, we know that since the \(x_i\) go to
\(\pm\infty\) when \(i\) goes to \(\pm\infty\), so do the
\(\overline{x}_i\).  Thus \(\theta_f(h)\) is unbounded.  We know
that each piece is increasing, so we only need to concentrate on
continuity.

We only need worry about the points \(\overline{x}_n\), and what we
must verify is that \[\theta_f(\gamma_n)(\,\overline{x}_n) =
\theta_f(\gamma_{n-1})(\,\overline{x}_n).\]

But we know \(\gamma_n({x}_n) =
\gamma_{n-1}({x}_n)\) from the continuity of the original \(h\) and
what we want follows from Lemma \ref{PreCont}. \end{proof}

\begin{lemma} \(\theta_f:PL_2(\R) \rightarrow \text{Homeo}_+(\R)\)
is a homomorphism of groups.  \end{lemma}

\begin{proof} To discuss \(\theta_f(h_1\circ h_2)\), one takes a
sequence of ``break points'' for \(h_2\) and \(h_2^{-1}\) of a
sequence of ``break points'' for \(h_1\) and merges them into a
sequence \((x_n)_{n\in\Z}\) so that \(h_2\) is affine on each
\([x_n, x_{n+1}]\) and \(h_1\) is affine on each \([h_2(x_n),
h_2(x_{n+1})]\).  Now on each affine piece, \(\theta_f(h_i)\) is just
\(\theta_f\) of the corresponding affine function and we know that
\(\theta_f\) is a homomorphism on \(GA(\Q_2)\).  \end{proof}

\begin{lemma} If \(h\in PL_2(\R)\) is the identity on an interval
\([x,y]\) with \(x,y\in \Q_2\), then \(\theta_f(h)\) is the identity
on \([\,\overline{x}, \overline{y}\,]\).  In particular, if the
support of \(h\in PL_2(\R)\) is in \([0,1]\), then the support of
\(\theta_f(h)\) is in \([0,1]\).  \end{lemma}

\begin{proof} 
The first sentence follows from the fact that \(\theta_f\) takes
the identity in \(GA(\Q_2)\) which is denoted \((0,0)\) in our
notation to \(T_0f^0\) which is the identity.

The second sentence follows from the first and the fact that
\(\overline{p}=p\) for any \(p\in\Z\).  \end{proof}

We now add another assumption about \(f\).  In the following \(r\)
is an integer with \(1\le r\le \infty\).

(III\({}_r\)) \(f\) is of class \(C^r\), \(f'(0)=1\) and
\(f^{(k)}(0)=0\) for \(2\le k\le r\).

\begin{lemma}  If \(f\) also satisfies (III\({}_r\)), then the image
of \(\theta_f\) consists of diffeomorphisms of class \(C^r\).
\end{lemma}

\begin{proof} This short proof uses more background facts about
Thosmpons's group \(F\) than the even shorter proof in
\cite{ghys+serg}.  However, I do not understand the terminology in
the proof of \cite{ghys+serg}.

We introduce the function \(D_0\) defined by \[D_0(x) =
\begin{cases} x, &x<0, \\
2x, &0\le x \le 1, \\
x+1, &1\le x, \end{cases}\] 
and the corresponding function \(f_0\)
defined by 
\[f_0(x) =
\begin{cases} x, &x<0, \\
f, &0\le x \le 1, \\
x+1, &1\le x. \end{cases}\] 
Because of our hypotheses, \(f_0\) is of class \(C^r\) on all of
\(\R\).

We know that \(\theta_f(T_1)=T_1\) from Remark \ref{QTwoEmb}.  It
follows from this, the definition of \(\theta_f\) and the facts
\(\overline{0}=0\) and \(\overline{1}=1\) that \(\theta_f(D_0)=f_0\).

It is well known that \(T_1\) and \(D_0\) generate the model of
\(F\) that is defined on all of \(\R\).  It is also well known that
every function in \(PL_2(\R)\) can be matched on any compact subset
of \(R\) by a function from this model of \(F\).

Let \(h\) be from \(PL_2(\R)\).  Let \(A\) be a compact interval in
\(\R\) with endpoints in \(\Q_2\).  There is a word \(W\) in
\(\{T_1, D_0\}\) and their inverses so that \(W\) and \(h\) agree on
\(A\).  It follows that \(\theta_f(h)\) and \(\theta_f(W)\) agree on
\(A\).  But \(\theta_f(W)\) is a composition of functions of class
\(C^r\) so \(\theta_f(h)|_A\) is of class \(C^r\).  Since \(A\) can
be taken to be arbitrarily large, we have the desired result.
\end{proof}

The following alternative proof sketch is probably closer to the
meaning of the proof in \cite{ghys+serg}.

\begin{proof} Let \((h_i), i\in\Z\) be a family of affine functions
in \(GA(\Q_2)\) all of which share the point \((p,q)\) in their
graphs with \(p\) and \(q\) in \(\Q_2\) so that the slope at \(p\)
of \(h_i\) is \(2^i\).  The behavior of all the \(h_0^{-1}h_i\) near
\(p\) is the behavior of \(T_pD^iT_{-p}\).

It is then desired to show that under the assumption (III\({}_r\))
we have that the first \(r\) derivatives of all the \(h_i\) agree at
\(p\).  That is, we want to calculate the derivatives of \[h_i =
h_0T_pD^iT_{-p}\] at \(p\).  When the point \(p\) is passed from
right to left through the composition on the right, it is seen that
it is treated by the factor \(D^i\) as its fixed point 0.  When
\(\theta_f\) is applied the composition on the right becomes
\(\theta_f(h_0)\theta_f(T_p)f^i\theta_f(T_{-p})\) and it is
evaluated at \(\overline{p}\).  Again the factor \(f^i\) is to be
evaluated at its fixed point 0.

One can then calculate the first \(r\) derivatives of this
composition taking into account that 0 is a fixed point of \(f\) and
that the first \(r\) derivatives of \(f\) at 0 are as dicated by
(III\({}_r\)).  It is not too hard to get an expression inductively
on the depth of the derivation that carries all the needed
information.  Alternatively, one writes out the terms of the Taylor
expansion up to the term involving the \(r\)-th derivative.  Either
technique will show that the first \(r\) derivatives of all the
\(\theta_f(h_i)\) at \(\overline{p}\) will agree.  In
\cite{ghys+serg} this discussion is covered by mention of the jet at
0 of \(f\).
\end{proof}

We now turn to condition (b).  We assume that \(f\) satisfies
(III\({}_\infty\)) and has a graph as shown below.

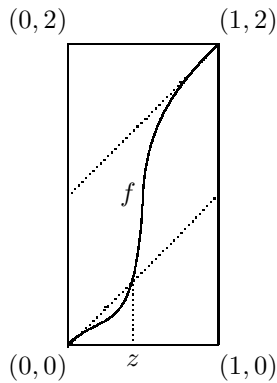
\[
\xy
(0,0); (20,0)**@{-}; (20,40)**@{-}; (0,40)**@{-}; (0,0)**@{-};
(0,0); (20,20)**@{.};
(0,20); (20,40)**@{.};
(0,0); (10,20)**\crv~Lc{~**\dir{}~*{}(5,5)&(9,0)};
(10,20); (20,40)**\crv~Lc{~**\dir{}~*{}(10.5,30)&(15,35)};
(-4,-3)*{(0,0)};
(-4,43)*{(0,2)};
(24,-3)*{(1,0)};
(24,43)*{(1,2)};
(8.65,8.65); (8.65,0)**@{.};
(8.65,-2)*{z};
(8,20)*{f};
\endxy
\]

The important points about this \(f\) are that \(f(z)=z\), that
\(z\in (0,1)\) is the largest value in \([0,1]\) for which
\(f(z)=z\), and that \(f'(z)>1\).  We let \(C= \log(f'(z))\).

We recall condition (b).

\begin{equation}\tag{b}
\forall h\ne 1\in F,\,\, \exists t\in [0,1]\,\, \bigg( |\log((\theta
h)'(t))| \ge C \bigg).
\end{equation}

In the following, we regard \(F\) as a subgroup of \(PL_2(\R)\) by
declaring that every element of \(F\) act as the identity outside of
\([0,1]\).  The theorem implies Theorem 1.

\begin{thm}\mylabel{FSatB} If \(f\) and \(C\) are as given above,
then the restriction of \(\theta_f(F)\) to \([0,1]\) has its image
in \(\text{Diff}^\infty([0,1])\), satisfies (b), and for every \(g\)
in the image \(g'(0)=g'(1)=1\) holds.  \end{thm}

\begin{proof}  All but condition (b) are covered by previous lemmas.

Let \(h\ne1\) be in \(F\).  Let \(x\) be the largest value in
\([0,1]\) for which \(h\) is the identity on \([0,x]\).  We know
that \(x\in \Q_2\) and \(x<1\).

For some \(k>0\) we know that \(h\) is affine and not the identity
on \(J=[x,\,x+2^{-k}]\).  By inverting if necessary, we can assume
that the slope of \(h\) on \(J\) is some \(2^n\) for \(n>0\).  Since
\(x\) is a fixed point of \(h\), we know that \(h\) on \(J\) is just
the conjugate \(T_xD^nT_{-x}\) of \(D^n\) on \([0,\,2^{-k}]\).

Therefore \(\theta_f(h)\) on \(\overline{J}=[\overline{x},
\overline{x+2^{-k}}]\)  is the conjugate \[
\theta_f(T_x) \theta_f(D^n) \theta_f(T_{-x})=
\theta_f(T_x) f^n \theta_f(T_{-x})
\]
of \(\theta_f(D^n)=f^n\) on \([0,\overline{2^{-k}}]\).  Thus we
should understand \(\overline{2^{-k}}\) and the behavior of \(f^n\)
on \([0,\overline{2^{-k}}]\).

We have that \(\overline{2^{-k}} = f^{-k}(1)\).  Since \([0,z]\) is
taken by \(f\) to itself, we know inductively that for all \(k>0\)
we have \(f^{-k}(1) \notin[0,z]\) or \(f^{-k}(1) > z\).  Thus for
all \(k>0\) we have \([0,z]\subseteq [0,\overline{2^{-k}}]\).  In
particular the behavior of \(f^n\) on \([0,\overline{2^{-k}}]\)
includes the behavior of \(f^n\) on its fixed point \(z\).

The derivative of \(f^n\) at \(z\) is \(C^n\).  It follows from the
chain rule that if all the ingredients of \(\psi\phi\psi^{-1}\) are
differentiable and if \(\zeta\) is a fixed point of \(\phi\), then
\((\psi\phi\psi^{-1})'(\psi(\zeta)) = \phi'(\zeta)\) and
\((\phi^n)'(\zeta)=(\phi'(\zeta))^n\).  Thus the function
\(\theta_f(h)\) as a conjugate of \(f^n\) has a point in
\(\overline{J}\) on which the derivative is \((f'z)^n\).
\end{proof}

\section{Statements of the main results in \cite{shavgulidze}}

In what follows, a theorem number followed by (S-n) will refer to
Theorem ``n'' in \cite{shavgulidze}.

Let \(\text{Diff}^3_0([0,1])\) be the set of all thrice continuously
differentiable self homeomorphisms \(f\) of \([0,1]\) that preserve
the endpoints and that additionally satisfy \(f'(0)=f'(1)=1\).  We
will be interested in subgroups \(G\) of \(\text{Diff}^3_0([0,1])\)
that satisfy the following.
\mymargin{TheHypC}\begin{equation}\tag{a}\label{TheHypC}
\exists C>0,\,\, \forall f\ne g\in G, {\underset{t\in[0,1]}{\sup}}
( |\log(f'(t))-\log(g'(t))|)\ge C.
\end{equation}

The main result in \cite{shavgulidze} is the following.

\begin{thm}[S-2]\mylabel{ShavThm} If a discrete subgroup \(G\) of
\(\text{Diff}^3_0([0,1])\) satisfies condition \tref{TheHypC}, then
the subgroup \(G\) is amenable.  \end{thm}

The bulk of the work will be to prove a theorem about the existence
of certain functionals on certain function spaces.  We will make the
appropriate defintions to give the statement.

We will work with several spaces of functions of which
\(\text{Diff}^3_0([0,1])\) will be among the smallest.  We give a
diagram of inclusions to help keep the definitions straight.  The
unit interval \([0,1]\) will be denoted \(I\).

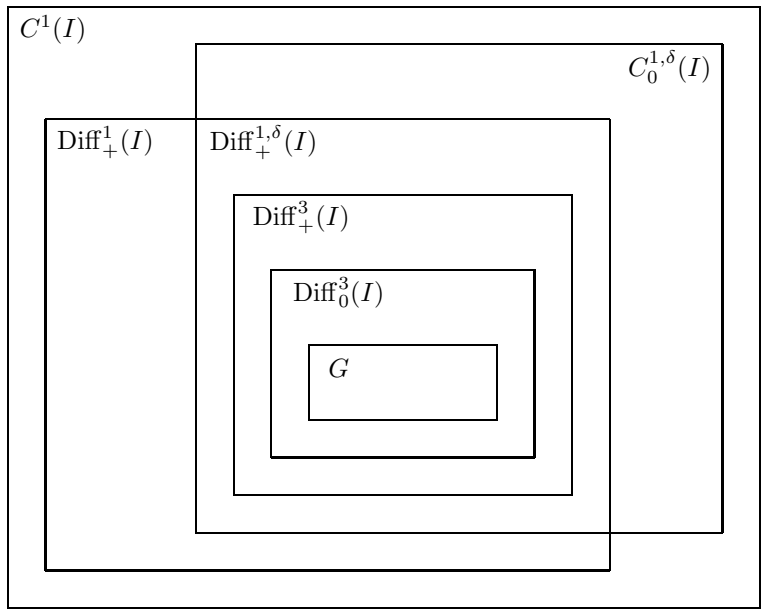
\[
\xy
(0,0); (100,0)**@{-}; (100,80)**@{-}; (0,80)**@{-}; (0,0)**@{-};
(6,77)*{C^1(I)};
(5,5); (5,65)**@{-}; (80,65)**@{-}; (80,5)**@{-}; (5,5)**@{-};
(13,62)*{\text{Diff}_+^1(I)};
(25,10); (25,75)**@{-}; (95,75)**@{-}; (95,10)**@{-}; (25,10)**@{-};
(34,62)*{\text{Diff}_+^{1,\delta}(I)};
(88,72)*{C_0^{1,\delta}(I)};
(30,15); (30,55)**@{-}; (75,55)**@{-}; (75,15)**@{-}; (30,15)**@{-};
(39,52)*{\text{Diff}_+^3(I)};
(35,20); (35,45)**@{-}; (70,45)**@{-}; (70,20)**@{-}; (35,20)**@{-};
(44,42)*{\text{Diff}_0^3(I)};
(40,25); (40,35)**@{-}; (65,35)**@{-}; (65,25)**@{-}; (40,25)**@{-};
(44,32)*{G};
\endxy
\]

We define the objects above.  One has already been defined, but we
will repeat the definition.

\begin{enumerate}

\item \(C^1(I)\) is the space of all continuously differentiable,
real valued functions on \(I\) with topology given by the norm
\[\|f\|_{C^1} = \max\left\{ {\underset{t\in[0,1]}{\sup}}|f(t)|,
{\underset{t\in[0,1]}{\sup}}|f'(t)|\right\} \]

\item \(\text{Diff}_+^1(I)\) is the group of all diffeomorphisms of
class \(C^1\) of \(I\) that are fixed on the endpoints.  The
topology on \(\text{Diff}_+^1(I)\) is the one inherited from \(C^1(I)\).

\item For \(0< \delta<1\), \(C_0^{1, \delta}(I)\) is the set of all
functions \(f\in C^1(I)\) so that \(f(0)=0\) and so that there is
\(C>0\) so that for all \(t_1, t_2\in I\), we have \[|f'(t_2) -
f'(t_1)| < C|t_2-t_1|^\delta.\] The constant \(C\) will be called a
H\"older constant for \(f'\) and we will say that \(f'\) is H\"older
with constant \(C\) and exponent \(\delta\).  The topology is given
by the following.  \[\|f\|_{1, \delta} = |f'(0)| +
{\underset{t_1,t_2 \in[0,1]}{\sup}} \frac{|f'(t_2) - f'(t_1)|}
{|t_1-t_2|^\delta}.\]

\item \(\text{Diff}_+^{1, \delta}(I) = \text{Diff}_+^1(I) \cap
C_0^{1, \delta}(I)\).  There are two topologies to choose from given
that there are two topological spaces that are being intersected,
and the choice is that the topology is inherited from that of
\(C_0^{1, \delta}(I)\).

\item \(\text{Diff}_+^3(I)\) is the subgroup of
\(\text{Diff}_+^1(I)\) that are of class \(C^3\).

\item \(\text{Diff}_0^3(I)\) is the set of elements \(f\) from
\(\text{Diff}_+^3(I)\) for which \(f'(0)=f'(1)=1\).

\end{enumerate}

In the following \(\|f\|_\infty\) denotes the sup norm of \(f\) over
the interval \([0,1]\).

\begin{lemma}\mylabel{PreCompCont} If \(f\) is in
\(\text{Diff}_+^{1, \delta}(I)\), then \(\|f\|_\infty \le \|f\|_{1,
\delta}\) and \(\|f'\|_\infty\le \|f\|_{1, \delta}\).  \end{lemma}

\begin{proof} We have for \(t\in [0,1]\), 
\[
\begin{split}
|f'(t)| 
&\le |f'(0)| + |f'(t)-f'(0)| \\
&\le |f'(0)| + \frac{|f'(t)-f'(0)|}{t^\delta}t^\delta \\
&\le |f'(0)| + \frac{|f'(t)-f'(0)|}{t^\delta} \\
&\le \|f\|_{1, \delta}.
\end{split}
\]

Now the mean value theorem and the fact that \(f(0)=0\) says that
\(\|f\|_\infty\le \|f'\|_\infty\le \|f\|_{1, \infty}\).
\end{proof}

It is easy to show that \(\|f\|_{1, \delta}\) is a norm.  If it is
zero on \(f\), then \(f'(0)=0\) and the second part forces \(f'\) to
be constant and thus zero.  But \(f(0)=0\) in \(C_0^{1, \delta}(I)\)
so \(f\) is idencially zero.  The linearity with respect to
multiplication by constants is immediate and the triangle inequality
is very straightforward.

The location of \(\text{Diff}_+^3(I)\) in \(\text{Diff}_+^{1,
\delta}(I)\) comes because the existence of a second derivative
implies a H\"older constant for the first derivative, and the other
parts of the definition of \(C_0^{1, \delta}(I)\) are met.

\begin{lemma} If \(f\) is in \(\text{Diff}_+^{1, \delta}(I)\) than
so is \(f^{-1}\).  Further, if \(C\) is the H\"older constant for
\(f'\) and \(m\) is the minimum of \(f'\) on \(I\), then the
H\"older constant for \((f^{-1})'\) is \(C/m^{2+\delta}\).
\end{lemma}

\begin{proof} We have \(f^{-1}\) in \(\text{Diff}_+^1(I)\) by
definition, and \(f(0)=0\) implies \(f^{-1}(0)=0\), so we must show
that there is a H\"older constant for \((f^{-1})'\).  We know that
the minimum for \(f'\) exists and is strictly greater than zero
because of the continuity of \(f'\), because \(f^{-1}\) is
differentiable by definition of \(\text{Diff}_+^1(I)\), and because
\(f\) must be increasing on \(I\) to be in \(\text{Diff}_+^1(I)\).
From the chain rule we know that \(1/m\) is the maximum of
\((f^{-1})'\) on \(I\).

We have 
\[
\begin{split}
|(f^{-1})'(t_2) - (f^{-1})'(t_1)|
&= \left|\frac{1}{f'(f^{-1}(t_2))} - \frac{1}{f'(f^{-1}(t_1))} \right|
\\
&=
\left| \frac{f'(f^{-1}(t_1)) - f'(f^{-1}(t_2))}
{f'(f^{-1}(t_2)) f'(f^{-1}(t_1))}\right|
\\
&\le
\frac{1}{m^2}C |f^{-1}t_1)-f^{-1}(t_2)|^\delta
\\
&\le
\left(\frac{1}{m^2} C \frac{1}{m^\delta} \right) |t_2-t_1|^\delta
\\
&= \frac{C}{m^{2+\delta}} |t_2-t_1|^\delta.
\end{split}
\]

\end{proof}

\begin{lemma} If \(f\) and \(g\) are in \(\text{Diff}_+^{1,
\delta}(I)\) than so is \(f\circ g\).  Further, if \(C_f\) is the
H\"older constant for \(f'\), \(C_g\) is the H\"older constant for
\(g'\), \(M_f\) is the maximum of \(f'\) on \(I\), and \(M_g\) is
the maximum of \(g'\) on \(I\), then the H\"older constant for
\((f\circ g)'\) is \(C_gM_f + C_fM_g^{1+\delta}\).  \end{lemma}

\begin{proof} As before, we need only compute the H\"older
constant.  
\[
\begin{split}
|(fg)'(t_2) - (fg)'(t_1)|
&=
|f'(g(t_2))g'(t_2) - f'(g(t_1))g'(t_1)| 
\\
&\le |f'(g(t_2))g'(t_2) - f'(g(t_2))g'(t_1)| 
+
\\
& \qquad\qquad
|f'(g(t_2))g'(t_1) - f'(g(t_1))g'(t_1)|
\\
&\le M_f|g'(t_2)-g'(t_1)|
+
M_g|f'(g(t_2)) - f'(g(t_1))|
\\
& \le
M_f C_g|t_2-t_1|^\delta + M_gC_f|g(t_2)-g(t_1)|^\delta
\\
&\le M_fC_g|t_2-t_1|^\delta + M_gC_fM_g^\delta|t_1-t_1|^\delta
\\
&=
(C_gM_f + C_fM_g^{1+\delta})|t_1-t_1|^\delta.
\end{split}
\]

\end{proof}

\begin{cor} \(\text{Diff}_+^{1, \delta}(I)\) is a group.  \end{cor}

In spite of the corollary, \(\text{Diff}_+^{1, \delta}(I)\) is not a
topological group with its given topology.

\begin{lemma} There is a \(g\in \text{Diff}_+^{1, \delta}(I)\) so
that the map \(f\mapsto g \circ f\) is not continuous on
\(\text{Diff}_+^{1, \delta}(I)\).  \end{lemma}

For convenience, the calculations in the proof will use \([-1,1]\)
as the interval \(I\).

\begin{proof} Let \(g(x) = (x+x^{5/3})/2\), let \(f(x) = x\) and let
\(f_\epsilon(x) = x-\epsilon+\epsilon x^2\) for some \(\epsilon\)
with \(0 < \epsilon < 1/2\).  Now all functions fix both \(-1\) and
\(1\).  All have derivatives that are continuous and positive on
\(I\).  The functions \(f\) and \(f_\epsilon\) have second
derivatives and so their derivatives satisfy the H\"older condition
with exponent \(\delta = 2/3\).

We consider \(g\).  We have 
\[ \
\frac{|g'(t_2)-g'(t_1)|} {|t_1-t_1|^{2/3}} 
= 
\frac56 \frac{|t_2^{2/3} - t_1^{2/3}|} {|t_2-t_1|^{2/3}}.
\] 
Since we can assume \(t_2\ne t_1\), we can also
assume that \(t_1\ne0\).  Let \(m=t_2/t_1\).  The fraction above
becomes 
\[
\frac56 \frac{|m^{2/3}-1|} {|m-1|^{2/3}}.
\]
This is continuous away from 1 and has limit \(5/6\) as
\(m\rightarrow \pm\infty\) and limit 0 as \(m\rightarrow 1\).  Thus
it is bounded and \(g'\) is H\"older with exponent \(2/3\).

In the following \(\delta=2/3\).

Using \[\|h\|_{1, \delta} = |h'(-1)| + \underset{t_1, t_2\in
[-1,1]}{\sup} \frac{|h'(t_2)-h'(t_1)|} {|t_2-t_1|^\delta}\] we have
\[\|f_\epsilon -f\|_{1, \delta} = 2\epsilon + \underset{t_1, t_2\in
[-1,1]}{\sup} \frac{2\epsilon|t_2-t_1|} {|t_2-t_1|^{2/3}} =
(2+2^{4/3})\epsilon.\]  This implies that \(f_\epsilon\rightarrow
f\) as \(\epsilon\rightarrow 0\) in \(\text{Diff}_+^{1, \delta}(I)\).

We now work on \(\|gf_\epsilon-gf\|_{1, \delta}\).

We have 
\[
\begin{split}
(gf)'(x) &= \frac12 + \frac56 x^{2/3}, \\
(gf_\epsilon)'(x) &= \left( \frac12 + \frac56 
[x-\epsilon + \epsilon x^2]^{2/3} \right) (1+2\epsilon x).
\end{split}\]
Now if we set \(\phi = (gf_\epsilon)'-(gf)'\), then we have
\[
\begin{split}
\phi(0) &= \frac12 + \frac56 [-\epsilon]^{2/3} - \frac12 \\
&= \frac56 \epsilon^{2/3}, \\
\phi(\epsilon) &= 
\left(\frac12 + \frac56 [\epsilon \cdot \epsilon^2]^{2/3}\right)
(1+2\epsilon\cdot \epsilon) -
\left(\frac12 + \frac56 \epsilon^{2/3}\right) \\
&= 
\frac12 + \epsilon^2 + \frac56 \epsilon^2 + \frac53 \epsilon^4 
- \frac12 - \frac56 \epsilon^{2/3} \\
&=
-\frac56 \epsilon^{2/3} + \frac{11}6 \epsilon^2 + \frac56
\epsilon^4, \\
\phi(\epsilon) - \phi(0) 
&=
-\frac56 \epsilon^{2/3} + \frac{11}6 \epsilon^2 + \frac56
\epsilon^4 - \frac56 \epsilon^{2/3} \\
&=
-\frac53 \epsilon^{2/3} + \frac{11}6 \epsilon^2 + \frac56
\epsilon^4 \\
&=
\left(-\frac53 + \frac{11}6 \epsilon^{4/3} + \frac56
\epsilon^{10/3}\right)\epsilon^{2/3}.
\end{split}
\]
Hence 
\[
\frac{|\phi(\epsilon)-\phi(0)|} {|\epsilon-0|^{2/3}}
=
\left|-\frac53 + \frac{11}6 \epsilon^{4/3} + \frac56
\epsilon^{10/3}\right|
\]
which has limit \(5/3\) as \(\epsilon\rightarrow 0\).  Since
\[
\|gf_\epsilon - gf\|_{1, \delta} \ge 
\frac{|\phi(\epsilon)-\phi(0)|} {|\epsilon-0|^{2/3}}
\]
we have that \(\|gf_\epsilon-gf\|_{1, \delta}\) does not converge to
0 as \(\epsilon\rightarrow 0\).
\end{proof}

In spite of this example, we do get continuity if there are
restrictions on \(g\).  The next lemma gives this.

\begin{lemma}\mylabel{CompCont} If \(g\in \text{Diff}^2(I)\), then
\(f\mapsto g\circ f\) is continuous on \(\text{Diff}_+^{1,
\delta}(I)\).  \end{lemma}

\begin{proof} Given \(g\in \text{Diff}^2(I)\), given \(f_0\in
\text{Diff}_+^{1, \delta}(I)\), and given \(\epsilon>0\), we will
find a \(K>0\) that depends only on \(g\) and \(f_0\), and
we will find an \(\epsilon_1\) that depends only on \(\epsilon\) and
\(g\) so that \(\|gf-gf_0\|_{1, \delta} < K\epsilon\) when
\(\|f-f_0\|_{1, \delta}< \epsilon_1\).

We start with \(\epsilon_1\).  With \(g\in \text{Diff}^2(I)\), we
know that \(g''\) is continous on the compact interval \(I\) and is
thus uinformly continuous.  Choose \(\epsilon_1\) so that whenever
\(|x-y|<\epsilon_1\) we have \(|g''(x)-g''(y)|<\epsilon\).  We also
require that \(\epsilon_1\le \epsilon\).  This does not
overdetermine \(\epsilon_1\).  In what follows, we will be assuming
\[\|f-f_0\|_{1, \delta} < \epsilon_1 \le \epsilon,\] so we can
safely use \(\epsilon\) in many places where we would have been
allowed to use \(\epsilon_1\).

Recall that \(f(0)=f_0(0)=0\).

What follows is a minor calculational stew.

The first part of \(\|gf-gf_0\|_{1, \delta}\) involves
\[
\begin{split}
|(gf)'(0)- (gf_0)'(0)|
&=
|g'(f(0))f'(0)- g'(f_0(0))f'_0(0)| \\
&=
|g'(0)| \, |f'(0)- f'_0(0)| \\
&\le
\|g'\|_\infty \, \|f- f_0\|_{1, \delta} \\
&\le
\epsilon \|g'\|_\infty  .
\end{split}
\]

The second part of \(\|gf-gf_0\|_{1, \delta}\) involves
\[
\begin{split}
&|((gf)'(t_2)-(gf)'(t_1)) - ((gf_0)'(t_2) - (gf_0)'(t_1))| \\
= &|g'(f(t_2))f'(t_2)
    - g'(f(t_1))f'(t_1)
    - g'(f_0(t_2))f'_0(t_2)
    + g'(f_0(t_1))f'_0(t_1)| \\
\le &|f'_0(t_1)|\,
     |g'(f(t_2))
    - g'(f(t_1))
    - g'(f_0(t_2))
    + g'(f_0(t_1))| \\
&+ |f'_0(t_1) - f'(t_1)|\,
  | g'(f(t_1)) -
    g'(f(t_2))| \\
&+ |f'_0(t_1) - f'_0(t_2)|\,
  | g'(f_0(t_2)) -
    g'(f(t_2))| \\
&+ |f'_0(t_1) - f'_0(t_2) - f'(t_1) + f'(t_2)|\,
  | g'(f(t_2))|
\end{split}
\]
What is needed now is an analysis of the four summands in the
expression that follows the inequality.

A factor of the first summand is 
\(|g'(f(t_2))
    - g'(f(t_1))
    - g'(f_0(t_2))
    + g'(f_0(t_1))|\) which is 
\(|(g'\circ f
    - g'\circ f_0)(t_2)
    - (g'\circ f
    - g'\circ f_0)(t_1)|\).  
This is the difference of the function \(g'f-g'f_0\) evaluated at
two places.  We will estimate the difference by estimating the
derivative \((g'f-g'f_0)'\).  Its absolute value is
bounded by
\[
\begin{split}
&|g''(f(x))f'(x)-g''(f(x))f'_0(x)| + 
|g''(f(x))f'_0(x)- g''(f_0(x))f'_0(x)| \\
=&|g''(f(x))|\,|f'(x)-f'_0(x)| + 
|g''(f(x))- g''(f_0(x))|\,|f'_0(x)|
\end{split}
\]
which can be made smaller than \(\|g''\|_\infty \epsilon +
\epsilon\|f'_0\|_\infty\) where the second \(\epsilon\) is derived from
our choice of \(\epsilon_1\) based on the uniform continuity of
\(g''\).  Now the first summand is bounded by 
\[
\begin{split}
&\|f'_0\|_\infty \Big(\epsilon\|g''\|_\infty +
\epsilon\|f'_0\|_\infty \Big) |t_2-t_1| \\
\le &\|f'_0\|_\infty \Big(\epsilon\|g''\|_\infty +
\epsilon\|f'_0\|_\infty \Big) |t_2-t_1|^\delta
\end{split}\]
since \(\delta<1\).

Using Lemma \ref{PreCompCont}, the second summand is bounded by
\[2\|g'\|_\infty \|f-f_0\|_{1, \delta}|t_2-t_1|^\delta
\le
2\|g'\|_\infty \epsilon|t_2-t_1|^\delta
.\]

Using the mean value thoerem and Lemma \ref{PreCompCont}, the third
summand is bounded by 
\[
\|f_0\|_{1, \delta}|t_2-t_1|^\delta \|g''\|_\infty \|f-f_0\|_{1,
\delta}
\le
\|f_0\|_{1, \delta} \|g''\|_\infty \epsilon
|t_2-t_1|^\delta
.\] 

The fourth summand equals 
\[|(f-f_0)'(t_2) - (f-f_0)'(t_1)|\,|g'(f(t_2))|\] and so is bounded
by 
\[
\|g'\|_\infty \|f-f_0\|_{1, \delta}|t_2-t_1|^\delta
\le
\|g'\|_\infty \epsilon|t_2-t_1|^\delta
.\]

Dividing the bounds on the four summands by \(|t_2-t_1|^\delta\) and
summing shows that the second part of \(\|gf-gf_0\|_{1, \delta}\)
is no larger than
\[
\Big(
\|f'_0\|_\infty(\|g''\|_\infty + \|f'_0\|_\infty) + 2\|g'\|_\infty + 
\|f_0\|_{1, \delta} \|g''\|_\infty + \|g'\|_\infty
\Big)\epsilon.
\]

Combining this with our estimate of the first part of
\(\|gf-gf_0\|_{1, \delta}\) and using Lemma \ref{PreCompCont} to
replace both \(\|f_0\|_\infty\) and \(\|f'_0\|_\infty\) by
\(\|f_0\|_{1, \delta}\) we have the following.
\[
\|gf-gf_0\|_{1, \delta} 
\le 
\Big(
4\|g'\|_\infty +
( \|f_0\|_{1, \delta})^2 + 
2\|f_0\|_{1, \delta} \|g''\|_\infty 
\Big) \epsilon
\]
Thus defining \(K\) to be equal to the expression in the large parentheses
gives a constant that depends only on \(g\) and \(f_0\).  This
proves the claimed continuity.
\end{proof}

To state the main theorem on which Theorem (S-2) is based, we need a
few more definitions.

For a space \(X\), let \(C_b(X)\) be the linear space of all
bounded, continuous, real valued functions on \(X\).  Now for \(F\in
C_b(\text{Diff}_+^{1, \delta}(I))\), for \(f\in \text{Diff}_+^{1,
\delta}(I)\),  and for \(g\in \text{Diff}_0^3(I)\), we
define \(F_g(f) = F(g^{-1}\circ f)\).

\begin{lemma} With \(F\) and \(g\) as above, \(F_g\) is in
\(C_b(\text{Diff}_+^{1, \delta}(I))\).  \end{lemma}

\begin{proof}  This follows from Lemma \ref{CompCont}
\end{proof}

We can now state the following.

\begin{thm}[S-1]\mylabel{ShavMain} For any positive
\(\delta<\frac12\), there exists a linear functional \[L_\delta:
C_b(\text{Diff}_+^{1, \delta}(I)) \rightarrow \R\] so that
\begin{enumerate} \item \(L_\delta(F) = 1\) if \(F\) is the constant
function to 1, \item \(|L_\delta(F)| \le {\underset{f\in
\text{Diff}_+^{1, \delta}(I)}{\sup}} |F(f)|\), \item
\(L_\delta(F)\ge0\) for any non-negative \(F\in
C_b(\text{Diff}_+^{1, \delta}(I))\), and \item \(L_\delta(F_g) =
L_\delta(F)\) for any \(g\in \text{Diff}_0^3(I)\) and \(F\in
C_b(\text{Diff}_+^{1, \delta}(I))\).  \end{enumerate} \end{thm}

The proof of Theorem (S-1) occupies the bulk of \cite{shavgulidze}.

\noindent {\bfseries Remarks.}We note that the linearity of
\(L_\delta\) and (i) implies that that \(L_\delta(C_K)=K\) where
\(C_K\) represents the constant function to \(K\).  Since
\(F-\inf(F)\) is non-negative, we get \(L_\delta(F-\inf(F))\ge0\)
from (iii), and then linearity implies that
\(L_\delta(F)\ge\inf(F)\).  Similarly, \(L_\delta(\sup(F)-F)\ge0\)
implies \(L_\delta(F)\le \sup(F)\).

\section{Reducing Theorem (S-2) to Theorem (S-1)}

Theorem (S-1) says that a certain space of functions is ``amenable
with respect to the action of a certain subgroup.''  In this case
the space of functions is \(\text{Diff}_+^{1, \delta}(I)\) and the
subgroup is \(\text{Diff}_0^3(I)\).  To apply this to a group that
is contained in \(\text{Diff}_0^3(I)\), such as a \(G\) that
satisfies (a), one is presented with the problem of saying something
about \(C_b(G)\) based on knowledge of \(C_b(\text{Diff}_+^{1,
\delta}(I))\).

This is done by finding a way to extend an arbitrary element
\(F:G\rightarrow \R\) of \(C_b(G)\) to all of \(\text{Diff}_+^{1,
\delta}(I)\) in such a way that various properties of \(F\) are
preserved.  

We introduce some necessary tools.

Pick a positive \(\delta< \frac12\).  For an \(f\in
\text{Diff}_+^{1, \delta}(I)\), define \[p_\delta(f) = |\log(f'(0))|
+ \underset{t_1, t_2 \in I}{\sup} \frac{|\log(f'(t_2)) -
\log(f'(t_1))|} {|t_2-t_1|^\delta}.\]  

\begin{lemma}  If \(f\) is in \(\text{Diff}_+^{1, \delta}(I)\) with
\(0<\delta<1\), then \(p_\delta(f)\) is finite.  \end{lemma}

\begin{proof}  We only have to worry about the second summand.  We
need to control \[|\log(f'(t_2))-\log(f'(t_1))|\] in comparison with
\(|t_2-t_1|^\delta\).  If \(m\) is the minimum of \(f'\) and \(M\)
is the maximum of \(f'\) on \(I\), we have \(0<m\le M\) because of
the restrictions on \(\text{Diff}_+^{1, \delta}(I)\).  On \([m,M]\)
the log function is differentiable with maximum derivative \(L\).
Thus we have 
\[
\begin{split}
|\log(f'(t_2))-\log(f'(t_1))| 
&\le 
L|f'(t_2)-f'(t_1| \\
&\le
L\|f\|_{1, \delta}|t_2-t_1|^\delta.
\end{split}
\]
This is all that is needed to show the finiteness of
\(p_\delta(f)\).
\end{proof}

In the following, note that if \(m\) is the minimum of \(f'\) over
\(I\) for an \(f\in \text{Diff}_+^1(I)\), then \(1/m\) is the maximum
of \((f^{-1})'\) over \(I\).  The lemma is stated with too strong a
hypothesis on \(g\), but it is what gets used later.

\begin{lemma}\mylabel{PrePreCompact} Let \(g\in \text{Diff}_0^3(I)\)
and \(f\in \text{Diff}_+^{1, \delta}(I)\) be such that
\(p_\delta(g\circ f) \le C\) for some \(C>0\).  Let \(m\) be the
minimum of \(f'\) on \(I\).  Then \(\psi=\log(g')\) is H\"older with
exponent \(\delta\) and H\"older constant \(C_g = (C +
p_\delta(f))/m^\delta\).  \end{lemma}

\begin{proof} Fix \(s,t\) with \(s<t\) in \(I\) and set
\(y=f^{-1}(t)\) and \(x=f^{-1}(s)\).  Then
\[
\begin{split}
\psi(t) - \psi(s) 
&=
\log(g'(f(y))) - \log(g'(f(x))) \\
&=
\log(g'(f(y))f'(y)) - \log(f'(y)) - [\log(g'(f(x))f'(x)) -
\log(f'(x))]
\\
&=
\log((g\circ f)'(y)) - \log((g\circ f)'(x)) - [\log(f'(y)) -
\log(f'(x))].
\end{split}
\]
This shows that 
\[
\begin{split}
|\psi(t) - \psi(s)| 
&\le p_\delta(g\circ f)|y-x|^\delta + p_\delta(f) |y-x|^\delta
\\
&\le (C+p_\delta(f))|f^{-}(t)- f^{-1}(s)|^\delta
\\
&\le (C+p_\delta(f))\frac{1}{m^\delta}|t-s|^\delta.
\end{split}\]
This verifies the claimed constant.
\end{proof}

\subsection{The Arzela-Ascoli Theorem}  A collection of theorems
about the compactness of certain spaces of functions is known by
various names.  We will make no attempt to be accurate about the
names.  We take our information from Munkres \cite{MR0464128},
Section 7-3.  A generalization that we do not need is in
\cite{MR0464128} Section 7-6.

Let \((Y,d)\) be a metric space, \(X\) a topological space and
\(C(X,Y)\) the set of continuous functions from \(X\) to \(Y\).  A
set \(S\) of functions in \(C(X,Y)\) is {\itshape equicontinuous} at
\(x_0\) if for ever \(\epsilon>0\) there is an open \(U\) containing
\(x_0\) so that for all \(f\in S\) and \(x\in U\) we have \(d(f(x),
f(x_0))< \epsilon\).  If \(S\) is equicontinuous at all \(x_0\in
X\), then \(S\) is {\itshape equicontinuous}.

The following is Theorem 3.3 of Chapter 7 of \cite{MR0464128}.

\begin{thm} Let \(X\) be a compact topological space and consider
\(C(X, \R^n)\) with the sup (uniform) metric.  A subset of \(C(X,
\R^n)\) is compact if and only if it is closed, bounded, and
equicontinuous.  \end{thm}

It is an elementary exercise to show that the theorem can be
restated to read that a subset \(S\) of \(C(X, \R^n)\) with \(X\)
compact has compact closure if it is equicontinuous and there is one
point \(x\in X\) (equivalently, for every point \(x\in X\)) so that
the set \(\{f(x)\mid f\in X\}\) is bounded.

The point of all this is the following.

\begin{lemma}\mylabel{PreFinite} Let \(f\) be in \(\text{Diff}_+^{1,
\delta}(I)\) and let \(G\subseteq \text{Diff}_0^3(I)\).  Then for
\(C>0\) the set \[A_C = \{\psi=\log(g') \mid g\in
G,\,p_\delta(g\circ f)\le C\}\] has compact closure in \(C(I)\) with
the sup metric.  \end{lemma}

\begin{proof} A summand of \(p_\delta(g\circ f)\) is \(|\log((g\circ
f)'(0))|\).  For \(f\) to be in \(\text{Diff}_+^{1, \delta}(I)\), we
must have \(f(0)=0\).  So \(p_\delta(g\circ f)\le C\) implies that
\[|\log(g'(f(0))) + \log(f'(0))| = |\log(g'(0)) + \log(f'(0))| \le C\]
giving that \(|\log(g'(0))|\le C+|\log(f'(0))|\) and 
\(\{\psi(0)\mid \psi \in A_C\}\) is bounded.

By Lemma \ref{PrePreCompact} any \(\psi\in A_C\) realized as
\(\psi=\log(g')\) satisfies \[|\psi(t_2)-\psi(t_1)| \le
C_g|t_2-t_1|^\delta\] where \(C_g = (C + p_\delta(f))/m^\delta\)
with \(m\) the minimum of \(f'\) on \(I\).  Thus \(C_g\) depends
only on \(C\), \(f\) and \(\delta\) and not on \(g\).  Thus \(A_C\)
is equicontinuous.  \end{proof}

\begin{cor}\mylabel{IsFinite} Let \(f\) be in \(\text{Diff}_+^{1,
\delta}(I)\) and let \(G\subseteq \text{Diff}_0^3(I)\) satsify
condition (a).  Then for
\(C>0\) the set \[A_C = \{g\in
G\mid \,p_\delta(g\circ f)\le C\}\] is finite.  \end{cor}

\begin{proof}  The set \(\{\log(g')\mid g\in A_C\}\) lies in a
compact subset of \(C(I)\) under the sup metric.  However condition
(a) is exactly the statement that there is a \(C>0\) so that the
balls of radius \(C/2\) in the sup metric on \(C(I)\) centered at
the elements of \(LG=\{\log(g')\mid g\in G\}\) are pairwise disjoint.
Thus only finitely many elements of \(LG\) can lie in a compact set.
\end{proof}

We now apply the corollary to proof Theorem \ref{ShavThm} from
Theorem \ref{ShavMain}.  To do this we need to establish the
continuity of the function \(p_\delta\) and we need to define
another function and establish its continuity.

Recall that we work with a positive \(\delta<1/2\) and recall the
definition
\[p_\delta(f) = |\log(f'(0))|
+ \underset{t_1, t_2 \in I}{\sup} \frac{|\log(f'(t_2)) -
\log(f'(t_1))|} {|t_2-t_1|^\delta}.\]  
Define \[r_\delta(f) = \underset{h\in G}{\inf}\left(p_\delta(h^{-1}
\circ f)\right)\] where \(G\) is a subgroup of
\(\text{Diff}_0^3(I)\) that satisfies condition (a).

\begin{lemma} The functions \(p_\delta\) and \(r_\delta\) are
continuous from \(\text{Diff}_+^{1, \delta}(I)\) to \(\R\).
\end{lemma}

\begin{proof}  For \(p_\delta\) we must show that we can control
\(|p_\delta(f)-p_\delta(f_0)|\) by keeping \(\|f-f_0\|_{1, \delta}\)
small.  If \(\|f-f_0\|_{1, \delta}< \epsilon\) then from Lemma
\ref{PreCompCont} we have \(\|f-f_0\|_\infty< \epsilon\) and 
\(\|f'-f'_0\|_\infty< \epsilon\).  We will also use the fact that
elements of \(\text{Diff}_+^{1, \delta}(I)\) have continuous
positive first derivatives that are bounded away from 0.  For the
following, we will let \(m_f\) and \(M_f\) be the min and max of
\(f'\) on \(I\) and similarly for \(m_{f_0}\) and \(M_{f_0}\).

For the first part of \(p_\delta\) we have 
\[
\begin{split}
|\log(f'(0))|-|\log(f'_0(0))|
&\le
|\log(f'(0))-\log(f'_0(0))| \\
&\le 
L_1|f'(0)-f'_0(0)| \\
&\le
L_1\epsilon
\end{split}
\]
where \(L_1\) is the maximum of \(\log'\) on the union of
\([m_f,M_f]\) and \([m_{f_0}, M_{f_0}]\).  Since \(log'\) is
decreasing, we know that \(L_1\) is the value of \(\log'\) at the
smaller of \(m_f\) and \(m_{f_0}\).

Since we also have \(\|f'-f'_0\|_\infty< \epsilon\), we can insist
that \(\epsilon< m_{f_0}/2\) from which we will get \(0 < m_{f_0}/2
< m_f\) and we can simply take \(L_1 = \log'(m_{f_0}/2) = 2/m_{f_0}\).

For the second part, we have to study
\mymargin{LogSup}\begin{equation}\label{LogSup}
\underset{t_1, t_2 \in I}{\sup} \frac{|\log(f'(t_2)) -
\log(f'(t_1))|} {|t_2-t_1|^\delta}
\end{equation}
and how it changes when \(f\) changes.  In the expression
\tref{LogSup}, we can assume \(t_1<t_2\).  The expression
\mymargin{LogQuot}\begin{equation}\label{LogQuot}
Q_f^\delta(t_1, t_2) =\frac{|\log(f'(t_2)) -
\log(f'(t_1))|} {|t_2-t_1|^\delta}
\end{equation}
defines a function \(Q_f^\delta\) that is defined
on the partly open triangle \(\Delta\) defined by \(0\le t_1<t_2\le
1\) in the unit 
square.  We thus want to compare \(\sup(Q_f^\delta)\) with 
\(\sup(Q_{f_0}^\delta)\).

If we show that for every \((t_2, t_2)\) in \(\Delta\), that
\[|Q_f^\delta(t_1, t_2) - Q_{f_0}^\delta(t_1, t_2)|\le \eta\] for some
\(\eta>0\), then we will have
\[|\sup(Q_f^\delta(t_1, t_2)) - \sup(Q_{f_0}^\delta(t_1, t_2))|\le
\eta.\] 
Thus we study \(|Q_f^\delta(t_1, t_2) - Q_{f_0}^\delta(t_1, t_2)|\).

We look at 
\[
\begin{split}
|Q_f^\delta(t_1, t_2) - Q_{f_0}^\delta(t_1, t_2)|
&=
\frac
{|
[\log(f'(t_2)) - \log(f'(t_1))] -
[\log(f'_0(t_2)) - \log(f'_0(t_1))] 
|}
{|t_2-t_1|^\delta} \\
&=
\frac
{|
[\log(f'(t_2)) - \log(f'_0(t_2))] -
[\log(f'(t_1)) - \log(f'_0(t_1))] 
|}
{|t_2-t_1|^\delta} \\
&=\frac
{\left|\log\left(\frac{f'(t_2)}{f'_0(t_2)}\right) 
  - \log\left(\frac{f'(t_1)}{f'_0(t_1)}\right)\right|}
{|t_2-t_1|^\delta} \\
&\le
L_2\frac
{\left|\frac{f'(t_2)}{f'_0(t_2)}
  - \frac{f'(t_1)}{f'_0(t_1)}\right|}
{|t_2-t_1|^\delta}
\end{split}
\]
where \(L_2\) is the maximum of \(\log'\) on the values achievable
by \(f'/f'_0\) on \(I\).  This is achieved on the smallest possible
value of \(f'/f'_0\) on \(I\) which is at least \(m_f/M_{f_0}\).
Since we are assuming \(\epsilon< m_{f_0}/2\), we can declare 
\mymargin{LogControl}\begin{equation}\label{LogControl}
L_2 = \log'\left(\frac{m_{f_0}}{2M_{f_0}}\right)
=\frac{2M_{f_0}}{m_{f_0}}.
\end{equation}

Now 
\[
\begin{split}
\left|\frac{f'(t_2)}{f'_0(t_2)}
  - \frac{f'(t_1)}{f'_0(t_1)}\right|
=&
\left|\frac{f'(t_2)}{f'_0(t_2)} -1
  - \frac{f'(t_1)}{f'_0(t_1)}+1\right| \\
=&
\left|\frac{f'(t_2)-f'_0(t_2)}{f'_0(t_2)} 
  - \frac{f'(t_1)-f'_0(t_1)}{f'_0(t_1)}\right| \\
\le&
\left|\frac{f'(t_2)-f'_0(t_2)}{f'_0(t_2)} -
      \frac{f'(t_1)-f'_0(t_1)}{f'_0(t_2)} 
            \right| \\
  &+ \left|\frac{f'(t_1)-f'_0(t_1)}{f'_0(t_2)} -
           \frac{f'(t_1)-f'_0(t_1)}{f'_0(t_1)}
            \right| \\
\le&
\left|\frac{1}{f'_0(t_2)}\right|\,
  \left|({f'(t_2)-f'_0(t_2)}) -
      (f'(t_1)-f'_0(t_1))
            \right| \\
  &+ |f'(t_1) - f'_0(t_1)|\,
    \left|\frac{1}{f'_0(t_2)} -
           \frac{1}{f'_0(t_1)}
            \right| \\
\le&
\frac{1}{m_{f_0}}\|f-f_0\|_{1, \delta}|t_2-t_1|^\delta \\
   &+\|f'-f'_0\|_\infty 
   \left|\frac{f'_0(t_1)-f'_0(t_2)}{f'_0(t_2)f'_0(t_1)}
    \right| \\
\le&
\frac{1}{m_{f_0}}\epsilon |t_2-t_1|^\delta \\
   &+\epsilon \frac{1}{(m_{f_0})^2}
    \|f_0\|_{1, \delta}|t_2-t_1|^\delta \\
=&
\frac
{m_{f_0}
   +
    \|f_0\|_{1, \delta}}
{(m_{f_0})^2}
\epsilon
|t_2-t_1|^\delta.
\end{split}
\]
Thus 
\[
|Q_f^\delta(t_1, t_2) - Q_{f_0}^\delta(t_1, t_2)|
\le L_2 \frac
{m_{f_0}
   +
    \|f_0\|_{1, \delta}}
{(m_{f_0})^2}
\epsilon
\]
with \(L_2\) as defined in
\tref{LogControl}.

Combining all this gives
\[
|p_\delta(f) - p_\delta(f_0)|
\le
\left(
\frac{2}{m_{f_0}} +
\frac
{2 M_{f_0} (m_{f_0} + \|f_0\|_{1, \delta})}
{(m_{f_0})^3}
\right) \epsilon
\]
when \(\|f-f_0\|_{1, \delta} < \epsilon < m_{f_0}/2\).  This proves
the continuity of \(p_\delta\).

We now turn to \(r_\delta\).  The proof of continuity will use all
available facts, including the fact that \(G\) satisfies condition
(a).

Pick \(f\) in \(\text{Diff}_+^{1, \delta}(I)\).  We will show that
\(r_\delta\) is continuous at \(f\) by showing that it is continuous
on some open set about \(f\).

Pick some \(C>r_\delta(f)\).

From Lemma \ref{PrePreCompact}, we know that for any \(g\in G\) with
\(p_\delta(g^{-1}\circ f)<C\), we have that \(g^{-1}\) is H\"older
with exponent \(\delta\) and H\"older constant no more than
\mymargin{UniformHolder}\begin{equation}\label{UniformHolder}
K_C(f) = (C+p_\delta(f)) \left\|\frac{1}{f'}\right\|_\infty.
\end{equation}
From Corollary \ref{IsFinite}, the set \[G(f, C) = \{g\in G\mid
p_\delta(g^{-1}\circ f) < C\}\] is finite.  Since for each \(g\in
G(f,C)\) the function \(f\mapsto g^{-1}\circ f\) is continuous,
there is an open \(U\) about \(f\) so that for every \(\tilde{f}\in
U\) we have \(p_\delta(g^{-1}\tilde{f})<C\).  In particular, we have
for every \(\tilde{f}\in U\) that \(r_\delta(\tilde{f})<C\).

The expression \(K_C(f)\) defined in \tref{UniformHolder} is
continuous in \(f\).

Pick a real \(D\) that is greater than \(K_C(f)\) for our chosen
\(f\) and \(C\).

Make the open \(U\) about \(f\) that was chosen above smaller so
that for all \(\tilde f\) in \(U\), we now also have \(K_C(\tilde{f})<
D\).

For this \(U\), define 
\[
\begin{split}
N_G(U) 
&= 
\{g\in G\mid \exists \tilde{f} \in U
\,\,\mathrm{with}\,\,\, p_\delta(g^{-1}\tilde{f}) < C\} \\
&=
\bigcup_{\tilde{f}\in U} \{g\in G\mid p_\delta(g^{-1}\tilde{f})< C\}.
\end{split}
\]
Thus for every \(\tilde{f}\in U\), the elements of \(G\) relevant to
the computation of \(r_\delta(\tilde{f})\) must be in \(N_G(U)\).

However for every \(g\in N_G(U)\), the H\"older constant is no more
than \(D\).  Thus as argued in Lemma \ref{PreFinite} and its
corollary, the set \(N_G(U)\) is finite.

Thus the function \(r_\delta\) restricted to \(U\) is the minimum of
a finite set of continuous functions (the functions
\(\tilde{f}\mapsto g^{-1}\circ \tilde{f}\) for \(g\in N_G(U)\)) and
is thus continuous on \(U\).  \end{proof}

We repeat the statement of Theorem \ref{ShavThm}.

\vspace{5pt}

\noindent {\bfseries Theorem \ref{ShavThm}} (S-2){\bfseries .}
{\itshape If a discrete subgroup \(G\) of \(\text{Diff}^3_0([0,1])\)
satisfies condition \tref{TheHypC}, then the subgroup \(G\) is
amenable.  }

\vspace{5pt}

\begin{proof}[Proof assuming Theorem \ref{ShavMain}]  We need one
more function which is obviously continuous.

Define \[\theta(t) = \begin{cases} 1-t, & 0\le
t\le 1, \\ 0, &t>1. \end{cases}\]

We now define a mapping 
\[
\pi_\delta:B(G) \rightarrow C_b(\text{Diff}_+^{1, \delta}(I))\] by
setting 
\mymargin{PiDef}\begin{equation}\label{PiDef}
\pi_\delta F(f) = \frac
{\sum_{h\in G}\theta(p_\delta(h^{-1}\circ f) - r_\delta(f)))F(h)}
{\sum_{h\in G}\theta(p_\delta(h^{-1}\circ f) - r_\delta(f)))}
.
\end{equation}
Note that \(\theta(p_\delta(h^{-1}\circ f) - r_\delta(f))\) is non-zero
only when \[r_\delta(f) \le p_\delta(h^{-1}\circ f) \le r_\delta(f) + 1.\]  By
Corollary \ref{IsFinite}, this only occurs for finitely many \(h\in
G\).  Thus the sums in \tref{PiDef} are finite sums and \(\pi_\delta
F\) is defined on all \(f\in \text{Diff}_+^{1, \delta}(I)\).

We now let \(L_\delta\) be as given by Theorem \ref{ShavMain} and
define a linear functional \[l:B(G) \rightarrow \R\] by setting
\(l(F) = L_\delta(\pi_\delta F)\).

The function \(\pi_\delta F\) on a given \(f\) is a weighted average
of values of \(F\) on \(G\) where the sum of the weights is 1 and
where the weights do not depend on \(F\).  From this and the remarks
after the statement of Theorem \ref{ShavMain}, we know \[\inf(F)\le
\inf(\pi_\delta F)\le l(F) \le \sup(\pi_\delta F) \le \sup(F).\]

For \(F\in B(G)\), let \(F_g\in B(G)\) be defined by \(F_g(h) =
F(g^{-1}h)\).  Letting \(j=g^{-1}h\) gives \(h=gj\) and we can write
\[
\begin{split}
\pi_\delta F_g(f) 
&=
\frac
{\sum_{h\in G}\theta(p_\delta(h^{-1}\circ f) - r_\delta(f))F(g^{-1}h)}
{\sum_{h\in G}\theta(p_\delta(h^{-1}\circ f) - r_\delta(f))} \\
&=
\frac
{\sum_{j\in G}\theta(p_\delta(j^{-1} \circ g^{-1} \circ f) - r_\delta(f))F(j)}
{\sum_{j\in G}\theta(p_\delta(j^{-1} \circ g^{-1} \circ f) - r_\delta(f))}
\\
&=
\pi_\delta F(g^{-1}\circ f) \\
&=
(\pi_\delta F)_g(f).
\end{split}
\]
Now from Theorem \ref{ShavMain}(iv) we have 
\[l(F_g) = L_\delta(\pi_\delta F_g) = L_\delta((\pi_\delta F)_g)
= L_\delta(\pi_\delta F) = l(F).\]

Thus \(l:B(G)\rightarrow \R\) satisfies all the requirements of a mean.
\end{proof}

\section{Six lemmas}  This section covers Lemmas 1--6 in
\cite{shavgulidze}.  The notation [S-Ln] refers to Lemma n in
\cite{shavgulidze}. It is hoped that motivation for these lemmas
will appear here in the fullness of time.

\subsection{Fourier transforms on \protect\(L^1(\R)\protect\) and 
\protect\(L^2(\R)\protect\)}

The proof of the first lemma will use Fourier transforms
extensively.  We will refer to \cite{MR0210528} and \cite{MR1657104}
where the definitions differ trivially.  (Compare \cite[\S
9.1]{MR0210528} with \cite[\S 17.1.1]{MR1657104}.)  We need the
following facts.

\newcommand{\IN}{\int\limits_{-\infty}^{\infty}}

For any element $f\in {L}^1(\mathbb R)$ and any $x\in R$ the
integral \[ \hat{f}(x) =\frac{1}{\sqrt{2\pi}}\IN f(t)e^{-ixt}dt \]
is well defined and defines a function $\hat{f}$ which is continuous
and vanishes at $\pm \infty$ \cite[Theorem 9.6]{MR0210528}.

The next three paragraphs summarize pieces of \cite[Theorem
9.13]{MR0210528} and the discussion preceding it, as well as \cite[\S
22.1]{MR1657104}.

If $f$ belongs to both ${L}^2(\mathbb R)$ and
${L}^1(\mathbb R)$ then $\hat{f}$ belongs to
${L}^2(\mathbb R)$ and $\| f \|_2=\|\hat{f}\|_2$.

For any function $f\in {L}^2(\mathbb R)$ and any $A>0$ let
$f_A$ be the product of $f$ and the characteristic function of the
interval $[-A,A]$. Each $f_A$ is a function in ${L}^2(\mathbb
R)\cap {L}^1(\mathbb R)$ and $\lim_{A\to \infty}f_A=f$ in the
${L}^2$ topology.  It follows that the family $\hat{f_A}$ is
Cauchy (i.e. for any $\epsilon>0$ there is $t$ such that if $A,B >t$
then $\|\hat{f_A}-\hat{f_B}\|_2<\epsilon$). Since
${L}^2(\mathbb R)$ is complete, there is
$\hat{f}\in{L}^2(\mathbb R)$ such that $\lim_{A\to
\infty}\hat{f_A}=\hat{f}$.  This defines the Fourier transform on
${L}^2(\mathbb R)$. The Fourier transform is an isometry of
${L}^2(\mathbb R)$ onto itself. In particular, it preserves the
inner product on ${L}^2(\mathbb R)$ given by \[
<f,g>=\frac{1}{\sqrt{2\pi}}\IN f(x)\overline{g}(x)dx.\] Furthermore,
the Fourier transform of $\hat{f}$ coincides with (the class of) the
function $x\mapsto f(-x)$

For any $f\in {L}^2(\mathbb R)$ there exists a sequence $A_n$
of real numbers approaching $\infty$ such that \[
\hat{f}(x)=\lim_{n\to\infty}\hat{f_{A_n}}(x)=
\lim_{n\to\infty}\frac{1}{\sqrt{2\pi}}\int_{-A_n}^{A_n}f(x)e^{-ixt}dt
\] for almost all $x$. In particular, if the limit
\[\lim_{A\to\infty}\frac{1}{\sqrt{2\pi}}\int_{-A}^{A}f(x)e^{-ixt}dt
\] exists for almost all $x$, then it computes the Fourier transform
of $f$.

For two functions $f$, $g$ and $x\in \mathbb R$ one defines \[
(f*g)(x)=\IN f(t)g(x-t)dt .\] If the integral exists for (almost)
all $x$ then $f*g$ is a new function, called the {\itshape
convolution} of $f$ and $g$.  

Assume that $f$ and $g$ are in ${L}^1(\mathbb R)$.  Then the
convolution $f*g$ is again in ${L}^1(\mathbb R)$ \cite[Theorem
7.14]{MR0210528}.  The convolution is a commutative, associative
operation on ${L}^1(\mathbb R)$ \cite[\S9.19(d)]{MR0210528}.
Moreover, $\widehat{f*g}=\hat{f}\hat{g}$ \cite[Theorem
9.2(c)]{MR0210528}, \cite[Proposition 23.1.2]{MR1657104}.

Now assume that $f,g$ are in ${L}^2(\mathbb R)$.  Then
$(f*g)(x)$ is well defined for any $x$ \cite[Proposition
23.2.1]{MR1657104}.  The function $f*g$ is continuous
\cite[Proposition 20.3.1]{MR1657104} and vanishes at infinity
\cite[Exercise 23.6]{MR1657104} but it is not necessarily in
${L}^2(\mathbb R)$ or in ${L}^1(\mathbb R)$.  However,
$\hat{f}\hat{g}\in {L}^1(\mathbb R)$ (proof of
\cite[Proposition 23.2.1(i)]{MR1657104}), so one can apply the
Fourier transform (or the inverse Fourier transform) to
$\hat{f}\hat{g}$.  It turns out that \[
\widehat{\hat{f}\hat{g}}(x)=(f*g)(-x)\] for any $x\in \mathbb R$
(ibid).  In particular, if $f*g\in {L}^1(\mathbb R)$
then \[\widehat{f*g}=\hat{f}\hat{g}.\]  It follows that if
$\hat{f}\hat{g}\in {L}^2(\mathbb R)$ then $f*g\in
{L}^2(\mathbb R)$ and $\displaystyle
\widehat{f*g}=\hat{f}\hat{g}$.

We now give a preliminary lemma.

\begin{lemma}\mylabel{PrelimSLOne}
The integral $H(y)=\int_{0}^{\infty}\frac{\cos(xy)dx}{\sqrt{1+x^2}}$ converges for any $y\neq 0$.
It defines a continuous function on $(0,\infty)$ with the following properties:
\begin{enumerate}
\item there is $\epsilon >0$ such that $\displaystyle -\log\frac{y}{\epsilon}\leq H(y)\leq -\log\frac{y}{4}$
for all $y\in~(0,\epsilon)$;

\item $|H(y)|\leq \frac{A}{y}$ for some $A>0$.
\end{enumerate}

\end{lemma}

\begin{proof}
We may assume that $y> 0$.
For any integer $n$ we define
\[ 
\begin{split}
H_n(y)
&=
\int_{-\pi/2}^{\pi/2}\frac{\cos(x)dx}{\sqrt{y^2+(x+n\pi)^2}}=
\int_{n\pi-\pi/2}^{n\pi+\pi/2}\frac{\cos(x-n\pi)dx}{\sqrt{y^2+x^2}} \\
&=
(-1)^n\int_{n\pi-\pi/2}^{n\pi+\pi/2}\frac{\cos(x)dx}{\sqrt{y^2+x^2}}=
(-1)^n\int_{(n\pi-\pi/2)/y}^{(n\pi+\pi/2)/y}\frac{\cos(xy)dx}{\sqrt{1+x^2}}
\end{split}
\]

Clearly $H_1(y)>H_2(y)>...>0$. Furthermore, since $\displaystyle
a^2+b^2\geq \frac{(a+b)^2}{2}$, we have for $n>0$ 
\mymargin{HnUpBnd}\begin{equation}\label{HnUpBnd}
\begin{split}
H_n(y) 
&=
\int_{-\pi/2}^{\pi/2}\frac{\cos(x)dx}{\sqrt{y^2+(x+n\pi)^2}} \\
&\leq
\int_{-\pi/2}^{\pi/2}\frac{\sqrt{2}dx}{y+x+n\pi} \\
&=
\sqrt{2}\log\left(1+\frac{\pi}{y+n\pi-\pi/2}\right),
\end{split}
\end{equation}
so $\lim_{n\to\infty}H_n(y)=0$.
It follows by the alternating series test that
\[\frac{1}{2}H_0(y)+\sum_{n=1}^{\infty}(-1)^nH_n(y)\]
converges. Furthermore, if $A>0$ and $k$ is an integer such that
$k\pi-\pi/2\leq Ay< (k+1)\pi-\pi/2$ then
\[
\int_{0}^{A}\frac{\cos(xy)dx}{\sqrt{1+x^2}}
=
\frac{1}{2}H_0(y)+\sum_{n=1}^{k-1}(-1)^nH_n(y)+(-1)^{k}s(A)
\]
for some $s(A)$ which satisfies $0\leq s(A)\leq H_k(y)$. It follows
that the integral defining $H(y)$ converges and 
\[
H(y)=\frac{1}{2}H_0(y)+\sum_{n=1}^{\infty}(-1)^nH_n(y). 
\]

In particular,
\mymargin{HnULBnds}\begin{equation}\label{HnULBnds}
\frac{1}{2}H_0(y)-H_1(y)\leq H(y)\leq \frac{1}{2}H_0(y).
\end{equation}
Note that
\[\frac{1}{2}H_0(y)
=
\int_{0}^{\pi/2}\frac{\cos(x)dx}{\sqrt{y^2+x^2}}
\leq
\int_{0}^{\pi/2}\frac{dx}{\sqrt{y^2+x^2}}
=
\log\left(\frac{\pi}{2}+\sqrt{\left(\frac{\pi}{2}\right)^2+y^2
}\right)-\log y.
\] Since $a^2+b^2\leq (a+b)^2$ for non-negative \(a\)
and \(b\), and $\log x$ is increasing, we conclude that
\mymargin{HNullUpBnd}\begin{equation}\label{HNullUpBnd}
\frac{1}{2}H_0(y)\leq \log(\pi+y)-\log(y)
= -\log\left(\frac{y}{\pi+y}\right)
\leq 
-\log\frac{y}{4}
\end{equation}
for all $y\in(0,4-\pi)$. On the other hand, using the inequality
$\cos x\geq 1-x^2/2$ we get
\[\frac{1}{2}H_0(y)
=
\int_{0}^{\pi/2}\frac{\cos(x)dx}{\sqrt{y^2+x^2}}
\geq
\int_{0}^{\pi/2}\frac{(1-\frac{x^2}{2})dx}{y+x}
\]
Now 
\[(y+x)\left(\frac{y-x}{2}\right) + 1-\frac{y^2}{2} =
1-\frac{x^2}{2}
\]
so 
\[
\frac{(1-\frac{x^2}{2})}{y+x}
=
\frac{(1-\frac{y^2}{2})}{y+x}
+
P(x,y)
\]
where \(P(x,y)\) is a polynomial in \(x\) and \(y\).  So there is a
constant \(D_1\) so that for all \(y\in (0,4-\pi)\subseteq(0,1)\), we have
\[
\begin{split}
\frac{1}{2}H_0(y)
&\ge
\left(1-\frac{y^2}{2}
\right)
\int_{0}^{\pi/2}\frac{dx}{y+x}
+
D_1 \\
&\ge
\int_{0}^{\pi/2}\frac{dx}{y+x}
+
D_1 \\
&=
\int_{y}^{y+\pi/2}\frac{du}{u}
+
D_1 \\
&=
\log\left(\frac{y+\pi/2}{y}\right)
+
D_1 \\
&=
-\log\left(\frac{y}{y+\pi/2}\right)
+
D_1 \\
&\geq
 -\log y +D_1
\end{split}
\]
Since $H_1(y)$ is
a bounded function of $y$, we have $H_1(y)\leq D_2$ for some
constant $D_2$. If $\epsilon \in (0,4-\pi)$ is such that $\log
\epsilon \leq D_1-D_2$, then we get the estimate
\[ -\log\frac{y}{\epsilon}
\leq 
-\log y +D_1-D_2\leq \frac{1}{2}H_0(y)-H_1(y)\]
for all $y\in (0,4-\pi)$. It follows that
\[ -\log\frac{y}{\epsilon}\leq H(y)\leq -\log\frac{y}{4}\]
for all $y\in(0,\epsilon)$.

For the estimate at infinity, note that by \tref{HNullUpBnd} we have
\[ 0\leq \frac{1}{2}H_0(y)\leq \log(\pi+y)-\log(y)=\log(1+\frac{\pi}{y})\leq \frac{\pi}{y},\]
and \tref{HnUpBnd} implies that
\[0\leq H_1(y)\leq \frac{\sqrt{2}\pi}{y+\pi/2}\leq \frac{\sqrt{2}\pi}{y}. \]
By \tref{HnULBnds} we get that
\[|H(y)|\leq \frac{\sqrt{2}\pi}{y}. \]
Finally, to see that $H$ is continuous note that
\[ 
\begin{split}
H_n(a)-H_n(b)
&=
\int_{-\pi/2}^{\pi/2}
\left(\frac{1}{\sqrt{a^2+(x+n\pi)^2}}-\frac{1}{\sqrt{b^2+(x+n\pi)^2}}
\right)
\cos(x)dx \\
&=
\int_{-\pi/2}^{\pi/2}\frac{\cos(x)(b^2-a^2)dx}
{\sqrt{Q(a,x,n)}\sqrt{Q(b,x,n)}
(\sqrt{Q(a,x,n)}+\sqrt{Q(b,x,n)})}.
\end{split}
\]
where \(Q(z,x,n) = z^2+(x+n\pi)^2\).
It follows that
\[ |H_n(a)-H_n(b)|\leq
|a^2-b^2|\int_{-\pi/2}^{\pi/2}\frac{dx}{(x+n\pi)^3}\]
A calculation shows that there is a \(C>0\) independent of \(n\) so
that 
\[
 |H_n(a)-H_n(b)|
\leq
C|a^2-b^2|n^{-3}\] 
for any $n\geq 1$. Thus
\[|H(a)-H(b)|\leq \frac{1}{2}|H_0(a)-H_0(b)|+C|a^2-b^2|\sum_{n=1}^{\infty}\frac{1}{n^3},\]
which immediately implies continuity of $H$.
\end{proof}

\subsection{Remark}

The function $H(y)$ has been studied extensively in the theory of
Bessel functions, where it is denoted by $K_0(y)$. It is a solution
to the differential equation \[ xf''(x)+f'(x)-xf(x)=0.\] Using
techniques from complex analysis one proves the following equalities
for $y>0$: \[ H(y) = \int_{1}^{\infty}\frac{e^{-yt}dt}
{\sqrt{t^2-1}}=\int_{0}^{\infty}e^{-y\cosh t}dt.  \] (See page 185
of \cite{MR1349110}.) The last integral easily shows that $H$
decreases exponentially at infinity. It also follows that $H$ is
nonnegative. The following expansion describes the asymptotic
behavior of $H$ around $0$ (combine (14) on \cite[Page
80]{MR1349110} with (2) of \cite[Page 77]{MR1349110} and
separate out the first term):
\mymargin{WatsonForm}\begin{equation}\label{WatsonForm}
H(y)=-\log\frac{y}{2}-\gamma+\sum_{m=1}^{\infty}
\frac{(\frac{y}{2})^{2m}}{(m!)^2}(\psi(m+1)-\log\frac{y}{2}),
\end{equation}
where $\gamma$ is the Euler constant and, from
\cite[Page 60]{MR1349110}, $\psi(m+1)=\sum_{k=1}^{m}\frac{1}{k}-\gamma$.

\subsection{Setting up the first lemma}

Let $f(x)=(1+x^2)^{-1/2}$. Clearly $f\in {L}^2(\mathbb
R)$. Note that \[\int_{-A}^{A}
e^{-ixy}f(x)dx=2\int_{0}^{A}\frac{\cos(xy)dx}{\sqrt{1+x^2}}.\] By
Lemma \ref{PrelimSLOne} and the main properties of the Fourier transform
discussed above we have $\hat{f}(y)=\displaystyle
\frac{2}{\sqrt{2\pi}}H(y)$ (i.e. the right hand side represents
$\hat{f}$).  Let \(I_1=f\) and for \(n\ge2\) define \(I_n\) by
\[I_n(x)=\frac{1}{(\sqrt{2\pi})^{n-1}}\IN\cdots\IN \frac{dx_1\ldots
dx_{n-1}}{\sqrt{(1+x_1^2)(1+(x_2-x_1)^2) \dots(1+(x-x_{n-1})^2)}}\]
By definition, we have $I_{n+1}(x)=I_n(x)*f$. We use
induction on $n$ to prove that $I_n\in {L}^2(\mathbb R)$ and
$\widehat{I_n}=(\hat{f})^n$. For $n=1$ this is clear. Assuming the
claim for $n$ we see that $\widehat{I_n}\hat{f}=(\hat{f})^{n+1}$. By
Lemma~\ref{PrelimSLOne}, we have $(\hat{f})^{n+1}\in {L}^2(\mathbb
R)$. It follows that $I_{n+1}=I_n*f \in {L}^2(\mathbb R)$ and
$\widehat{I_{n+1}}=\widehat{I_n}\hat{f}=(\hat{f})^{n+1}$.

Now we can prove the following

\begin{lemma}[S-L1] 
There exist positive constants $c_1,c_2$ such that the integrals
\[
T_n
=
\int\limits_{-\infty}^{+\infty} \cdots \int\limits_{-\infty}^{+\infty}
\frac{dx_1 \cdots dx_n}
{
\sqrt{
      (1+x_1^2)
      (1+(x_2-x_1)^2)
      \cdots
      (1+(x_n-x_{n-1})^2)
      (1+x_n^2)
     }
}
\] 
satisfy $c_12^{n+1}(n+1)!\leq T_n\leq c_22^{n+1}(n+1)!$ for every integer $n>0$.
\end{lemma}

\begin{proof}

Clearly \[\frac{T_n}{(\sqrt{2\pi})^{n}}=\frac{1}{\sqrt{2\pi}}\IN
I_n(x_n)f(x_n)dx_n=\frac{1}{\sqrt{2\pi}} \IN
I_n(x_n)\overline{f}(x_n)dx_n=<I_n,f>\] (the inner product in
$\text{L}^2(\mathbb R)$). Since the Fourier transform is an
isometry, we have
$<I_n,f>=<\widehat{I_n},\hat{f}>=<(\hat{f})^n,\hat{f}>$. It follows
that 
\[
\begin{split}
T_n
&=
(\sqrt{2\pi})^{n-1}\IN(\hat{f}(x))^n\overline{\hat{f}}(x)dx \\
&=
(\sqrt{2\pi})^{n-1}\IN(\hat{f}(x))^{n+1}dx \\
&=
\frac{2^n}{\pi}\IN
H(x)^{n+1}dx
=
\frac{2^{n+1}}{\pi}\int\limits_{0}^{\infty} H(x)^{n+1}dx.
\end{split}
\]
Recall now that $\int_{0}^{1}(-\log x)^ndx=n!$. It follows from Lemma
\ref{PrelimSLOne} that 
\[ 
\begin{split}
\int\limits_{0}^{\infty} H(x)^{n+1}dx 
&\leq
\int\limits_0^{\epsilon}(-\log\frac{y}{4})^{n+1}dy
+
A^{n+1} \int\limits_{\epsilon}^{\infty}\frac{dy}{y^{n+1}} \\
&\leq
\int\limits_0^{4}(-\log\frac{y}{4})^{n+1}dy
+
A^{n+1}\frac{1}{n\epsilon^n} \\
&=
4(n+1)!+A^{n+1}\frac{1}{n\epsilon^n}
\end{split}
\]
and 
\[
\begin{split}
\int\limits_{0}^{\infty} H(x)^{n+1}dx
&\geq
\int\limits_0^{\epsilon}(-\log\frac{y}{\epsilon})^{n+1}dy
-
A^{n+1} \int\limits_{\epsilon}^{\infty}\frac{dy}{y^{n+1}} \\
&=
\epsilon(n+1)!
-
A^{n+1}\frac{1}{n\epsilon^n}
\end{split}
\]
The results follows now easily by the fact that $\displaystyle
\lim_{n\to\infty}\frac{A^{n+1}}{n\epsilon^n(n+1)!}  =0$.
\end{proof}

\subsection{Exercise}

Using \tref{WatsonForm}, show that
\[\lim_{n\to \infty}\frac{T_n}{2^{n+1}(n+1)!}=G/\pi,\]
where $\log G=\log 2 -\gamma$

\subsection{A definition}

Let 
\[
v_1(\tau) = \int\limits_{-\infty}^{+\infty} 
\frac{d\tau_1}
{\sqrt{(1+\tau_1^2)(1+(\tau-\tau_1)^2)}}
\]
for any \(\tau\in \R\).  The function \(v_1\) is the convolution of
two functions in \(L^2(\R)\) and so by remarks above, it vanishes at
\(\pm\infty\).  We have \(v_1(0)=\pi\) and we show below that this
is the maximum value of \(v_1\) on \(\R\).

Note that 
\[
\begin{split}
v_1(-\tau) 
=& 
\int\limits_{-\infty}^{+\infty} 
\frac{d\tau_1}
{\sqrt{(1+\tau_1^2)(1+(\tau+\tau_1)^2)}} \\
=& 
\int\limits_{-\infty}^{+\infty} 
\frac{d\tau_2}
{\sqrt{(1+(\tau_2-\tau)^2)(1+\tau_2^2)}} 
\qquad \mathrm{letting}\,\,\tau+\tau_1=\tau_2,
\\
=& 
\int\limits_{-\infty}^{+\infty} 
\frac{d\tau_2}
{\sqrt{(1+\tau_2^2)(1+(\tau-\tau_2)^2)}} = v_1(\tau).
\end{split}
\]

Since \(v_1(\tau)=v_1(-\tau)\), we have 
\(v'_1(\tau) = (v_1(-\tau))' = -v'_1(-\tau)\), or \[v'_1(-\tau) =
-v'_1(\tau).\] 

\begin{lemma}[S-L2] The derivative \(v'_1(t)\) is negative for any
\(t>0\), and \(|v'_1(t)| \le \frac{4}{|t|}v_1(t)\) for any \(t\ne
0\).  \end{lemma}

\begin{proof}  Replacing some variables so that a substitution works
out nicely lets us write
\[
v_1(t) = \int\limits_{-\infty}^{+\infty} 
\frac{d\tau}
{\sqrt{(1+\tau^2)(1+(t-\tau)^2)}}.
\]
Differentiating inside the integral gives
\[
v'_1(t) = \int\limits_{-\infty}^{+\infty} 
\frac{-(t-\tau)d\tau}
{\sqrt{(1+\tau^2)(1+(t-\tau)^2)^3}}.
\]
Setting \(\tau_1=t-\tau\) gives \(\tau=t-\tau_1\) and 
\[
\begin{split}
v'_1(t) &= \int\limits_{-\infty}^{+\infty} 
\frac{-\tau_1d\tau_1}
{\sqrt{(1+\tau_1^2)^3(1+(\tau_1-t)^2)}} \\
&=
-\int\limits_0^{+\infty} \left(
\frac1 {\sqrt{1+(\tau_1-t)^2}}
-
\frac1 {\sqrt{1+(\tau_1+t)^2}}
\right)
\frac{\tau_1d\tau_1}
{\sqrt{(1+\tau_1^2)^3}}
\end{split}
\]
by replacing \(\tau_1\) by \(-\tau_1\) on \((-\infty, 0]\).

Combining fractions and rationalizing the numerator gives
\[
\begin{split}
v'_1(t)
&=
-\int\limits_0^{+\infty} \left(
\frac{4t\tau_1}
{\sqrt{PQ}
(\sqrt{P}+\sqrt{Q})}
\right)
\frac{\tau_1d\tau_1}
{\sqrt{(1+\tau_1^2)^3}}
\\
&=
-\int\limits_0^{+\infty}
\left(
\frac{4t}             
{\sqrt{Q}
(\sqrt{P}+\sqrt{Q})}
\right)
\left(
\frac{\tau_1^2}
{1+\tau_1^2}
\right)
\frac{ d\tau_1}
{\sqrt P\sqrt{1+\tau_1^2}}
\end{split}
\]
where 
\(P={1+(\tau_1-t)^2}\) and 
\(Q={1+(\tau_1+t)^2}\).  This shows \(v'(t)<0\) when \(t>0\).

Note that 
\[
\begin{split}
{\sqrt{Q}
(\sqrt{P}+\sqrt{Q})}
=&
{\sqrt{1+(\tau_1+t)^2}
(\sqrt{1+(\tau_1-t)^2}+\sqrt{1+(\tau_1+t)^2})} \\
=&
\sqrt{1+(\tau_1^2-t^2)^2}+({1+(\tau_1+t)^2}) \\
\ge&
2+(\tau_1+t)^2
\end{split}
\]
where we know that \(\tau_1\ge0\).  If \(t>0\), then 
\[
{\sqrt{Q}
(\sqrt{P}+\sqrt{Q})} \ge t^2.
\]
Thus for \(t>0\), we have 
\[
\begin{split}
|v'_1(t)| 
&\le
\frac 4 t 
\int\limits_0^{+\infty}
\frac{ d\tau_1}
{\sqrt P\sqrt{1+\tau_1^2}} \\
&=
\frac 4 t 
\int\limits_0^{+\infty}
\frac{ d\tau_1}
{\sqrt{1+(\tau_1-t)^2}\sqrt{1+\tau_1^2}} \\
&\le
\frac 4 t 
\int\limits_{-\infty}^{+\infty}
\frac{ d\tau_1}
{\sqrt{1+(\tau_1-t)^2}\sqrt{1+\tau_1^2}} = 
\frac 4 t v_1(t).
\end{split}
\]
Since \(v'_1(-t) = -v'_1(t)\), we have 
\[
|v'_1(t)| \le \frac{4}{|t|}v_1(t)
\]
for all \(t\ne0\).
\end{proof}

\begin{cor}\mylabel{LimVOne} For any \(r\in \R\) we have
\[\lim_{t\rightarrow +\infty} \frac{v_1(t-r)}{v_1(t)} = 1.\]
\end{cor}

\begin{proof}
Since \[\frac{v_1(t-r)}{v_1(t)} =
\frac{v_1(t-r)-v_1(t)}{v_1(t)} +1,\] we need only show that 
\[\lim_{t\rightarrow +\infty} 
\frac{v_1(t-r)-v_1(t)}{v_1(t)} = 0.\]

 For \(t>0\), \(v_1(t)\) is positive and decreasing and by taking
 \(t\) large enough, we can assume both \(t\) and \(t-r\) are
 positive. 

We start with negative \(r\) so that \(t<t-r\) and \(v_1(c) <
v_1(t)\) for \(c\in (t,t-r)\).

Now
\[
\left|
\frac{v_1(t-r)-v_1(t)}{v_1(t)}
\right| = |r|\frac{|v'_1(c)|}{v_1(t)} 
\le
|r|\frac{4}{|c|}\frac{v_1(c)}{v_1(t)} 
\le
|r|\frac{4}{|c|}
\]
for some \(c\) between \(t-r\) and \(t\). 
But this goes to zero as \(t\rightarrow +\infty\).

If \(r>0\), then \(t-r<t\) and 
\[
\left|\frac{v_1(t-r)-v_1(t)}{v_1(t)}\right|
\le 
\left|\frac{v_1(t-r)-v_1(t)}{v_1(t-r)}\right|
\]
which can be made arbitrarily small by the first calculation.
\end{proof}

\subsection{Definitions}

Let 
\[v(t) = v_1(\log(t+\sqrt{t^2-1}))\qquad\mathrm{for}\,\,\,t\ge1.\]

Let
\[D_n = \{(x_1, \ldots, x_{n-1})\mid 0<x_1< \cdots < x_{n-1}<1\}
\]
and define \(x_{-1} = x_{n-1}-1\), \(x_0=0\), and \(x_n=1\).

Let 
\[
u_{1,n}(x_1, \ldots, x_{n-1}) = 
\prod_{k=1}^n 
\frac
{1}
{x_k-x_{k-1}}
v\left(
\frac
{x_k-x_{k-2}}
{2\sqrt{(x_k-x_{k-1}) (x_{k-1}-x_{k-2})}}
\right),
\]
\[J_n = 
\int\limits_0^1
\int\limits_{x_1}^1
\cdots
\int\limits_{x_{n-2}}^1
u_{1,n}(x_1, \ldots, x_{n-1})
dx_1\cdots dx_{n-1},
\]
\[
u_n(x_1, \ldots, x_{n-1}) = 
\frac{u_{1,n}(x_1, \ldots, x_{n-1})}{J_n}.
\]

Define transformations 
\[
\begin{split}
A(x_1, \ldots, x_{n-1}) &= (l_1, \ldots, l_{n-1}),
\\
B(l_1, \ldots, l_{n-1}) &= (y_1, \ldots, y_{n-1}),
\qquad\mathrm{and} \\
C(y_1, \ldots, y_{n-1}) &= (z_1, \ldots, z_{n-1})
\end{split}
\]
using
\[
\begin{split}
l_k &= {x_k-x_{k-1}}, \\
y_k &= \frac{l_k}{1-x_{n-1}} = \frac{l_k}{l_n}, \\
z_k &= \frac12\log(y_k),
\end{split}
\]
for \(0\le k\le n\).

Note that with the conventions about \(x_{-1}\), \(x_0\) and
\(x_n\), we have \[
\begin{split}
l_0 &=x_0-x_{-1} = 0-(x_{n-1}-1) = l_n, \\
y_n &= \frac{l_n}{l_n} = 1, 
\qquad\mathrm{and} \\
y_0 &= \frac{l_0}{l_n} =
\frac{l_n}{l_n} 
 = 1,
\end{split}
\]
and we get \(z_0=z_n=0\).

\subsection{Jacobians}

Let \(U\) be the interior of the
region of integration in the definition of \(J_n\).  Then we have
the following.

\begin{lemma}\mylabel{JacCalc} The following hold.
\begin{enumerate}
\item The transformations \(A\), \(B\) and \(C\) are all
invertible.
\item  The composition \(BA\) is a bijection from \(U\) to 
\((0,\infty)^{n-1}\).
\item The transformation \(C\) is a bijection from
\((0,\infty)^{n-1}\) to \(\R^{n-1}\).
\item The Jacobians of \(A\), \(B\) and \(C^{-1}\) are, respectively,
\(1\), \((1-x_{n-1})^{-n}\) and \[2^{n-1}\prod_{k=1}^{n-1}y_k.\]
\end{enumerate}
\end{lemma}

\begin{proof}
The trasformation \(A\) is invertible since \(x_k=\sum_{j=1}^k
l_j\).
  The Jacobian of \(A\) is 1 since the matrix \(\partial
l/\partial x\) is triangular with ones on the diagonal.

The transformation \(B\) is invertible since we first recover
\(l_n\) from \[S=\sum_{k=1}^{n-1}y_k = \frac{1}{l_n}
\sum_{k=1}^{n-1}l_k = \frac{x_n}{l_n} = \frac{1-l_n}{l_n}\] as \[l_n
= \frac{1}{S+1}.\] Then \(l_k=y_k l_n\).  
The composition \(BA\) takes \(U\) into \((0, \infty)^{n-1}\) and
the inverse computes as 
\[x_k = \sum_{j=1}^k l_j = \frac{\sum_{j=1}^k y_j}
{1 + \sum_{j=1}^{n-1} y_j}\] which takes any tuple \((y_1, \ldots,
y_{n-1})\) in \((0, \infty)^{n-1}\) to a tuple  \((x_1, \ldots,
x_{n-1})\) in \(U\).

To compute the Jacobian of
\(B\), we note that \(l_n=1-\sum_{k-1}^{n-1} l_k\) giving \(\partial
l_n/\partial l_k=-1\) for \(1\le k\le n-1\).  So 
\[ 
\frac{\partial y_k}{\partial l_j} = 
\begin{cases} 
\displaystyle{\frac{l_n+l_k}{l_n^2}}, &j=k, \\
\displaystyle{\frac{l_k}{l_n^2}}, &j\ne k. \end{cases}
\]
Thus the Jacobian of \(B\) is 
\[
l_n^{-2(n-1)}
\pmatc{|}{.}{|}
l_n+l_1 & l_1 & l_1 & \cdots & l_1 \cr
l_2 & l_n+l_2 & l_2 & \cdots & l_2 \cr
l_3 & l_3 & l_n+l_3 & \cdots & l_3 \cr
\vdots& \vdots& \vdots& \ddots& \vdots \cr
l_{n-1} & l_{n-1} & l_{n-1} & \cdots & l_n+l_{n-1} \cr
\endpmat
\]
If \(\mathbf c_j\) is the \(j\)-th column, then for
\(1\le j\le n-2\) we replace simultaneously \(\mathbf c_j\) by 
\(\mathbf c_j - \mathbf c_{j+1}\) and get that the Jacobian of
\(B\) is  
\[
l_n^{-2(n-1)}
\pmatc{|}{.}{|}
l_n & 0 & 0 & \cdots & 0 & l_1 \cr
-l_n & l_n & 0 & \cdots & 0 & l_2 \cr
0 & -l_n & l_n & \cdots & 0 & l_3 \cr
\vdots& \vdots& \vdots& \ddots& \vdots & \vdots \cr
0 & 0 & 0 & \cdots & -l_n & l_n+l_{n-1} \cr
\endpmat
\]
If \(\mathbf r_j\) is the \(j\)-th row, then for
\(2\le j\le n-1\) we replace, in succession from \(j=2\), \(\mathbf r_j\) by 
\(\mathbf r_j + \mathbf r_{j-1}\) and get that the Jacobian of
\(B\) is
\[
l_n^{-2(n-1)}
\pmatc{|}{.}{|}
l_n & 0 & 0 & \cdots & 0 & x_1 \cr
0 & l_n & 0 & \cdots & 0 & x_2 \cr
0 & 0 & l_n & \cdots & 0 & x_3 \cr
\vdots& \vdots& \vdots& \ddots& \vdots & \vdots \cr
0 & 0 & 0 & \cdots & 0 & 1 \cr
\endpmat
= l_n^{-2(n-1)+(n-2)} = l_n^{-n} = (1-x_{n-1})^{-n}
\]
since \(x_k=\sum_{j=1}^k
l_j\) and \(\sum_{j=1}^n l_j=1\).

The claims about the transformation \(C\) are straightforward.
\end{proof}

\subsection{Calculations}\mylabel{CalcSubSec}

These are here mostly to help me keep my sanity.

An element of \(D_n\) is basically a coordinate in the interior of
an \((n-1)\)-simplex.  The element \((x_1, \ldots, x_{n-1})\) in
\(D_n\) gives \(n\) lengths \((x_1-x_0, \ldots, x_{n}-x_{n-1})\)
following the convention that \(x_0=0\) and \(x_n=1\).  These are
all strictly positive and sum to 1, so the \(n\) lengths give a
point in the \((n-1)\)-simplex.

We can refer to the lengths as \(l_k=x_k-x_{k-1}\).  The lengths
are not independent since they must sum to 1.  The \(y_k\) dilate
the \(l_k\) by \(1/l_n\) and rescale the coordinates so that they
occupy all of \((0,\infty)\).  Moving from the \(l_k\) to \(y_k\)
preserves ratios of the lengths for \(1\le k\le n-1\) and commutes
with summing.  Specifically \(l_k+l_{k-1}\) is taken to
\(y_k+y_{k-1}\) by the dilation \(1/l_n\).

We have the equalities of ratios
\mymargin{SameRatios}\begin{equation}\label{SameRatios}
\left(
\frac
{x_k-x_{k-2}}
{2\sqrt{(x_k-x_{k-1}) (x_{k-1}-x_{k-2})}}
\right)
=
\left(
\frac
{l_k+l_{k-1}}
{2\sqrt{l_k l_{k-1}}}
\right)
=
\left(
\frac
{y_k+y_{k-1}}
{2\sqrt{y_k y_{k-1}}}
\right).
\end{equation}
Now
\[
\left(
\frac
{a+b}
{2\sqrt{a b}}
\right)
=
\frac12
\left(
\sqrt{\displaystyle{\frac{a}{b}}}
+
\sqrt{\displaystyle{\frac{b}{a}}}
\right)
\]
which has the form 
\[
\frac12
\left(
z+\frac 1 z
\right).
\]
Now 
\[
\frac12(p+q+|p-q|) = \begin{cases}
p, &p\ge q, \\
q, &p<q,
\end{cases}
\]
so we will get the larger of \(z\) or \(1/z\) if we can form 
\[
\frac12
\left(
z+\frac 1 z
+\left|
z-\frac 1 z\right|
\right).
\]
We take advantage of the fact that 
\[
\frac14\left(
z+\frac 1 z
\right)^2 
-
\frac14\left(
z-\frac 1 z
\right)^2
=1
\]
to get
\[
\sqrt{
\displaystyle
{
\left(\frac12
\left(z+\frac 1 z\right)
\right)^2 - 1
}
}
=
\sqrt{
\displaystyle
{
\left(\frac12
\left(z-\frac 1 z\right)
\right)^2
}
}
=
\frac12\left|z-\frac 1 z\right|.
\]
Combining all this we get
\[
\left(
\frac
{a+b}
{2\sqrt{a b}}
\right)
+
\sqrt{
\displaystyle{
\left(
\frac
{a+b}
{2\sqrt{a b}}
\right)
^2
}-1
}
=
\begin{cases}
\sqrt{\displaystyle{\frac a b}}, &a\ge b, \\
\sqrt{\displaystyle{\frac b a}}, &a < b.
\end{cases}
\]
Recalling
\[v(t) = v_1(\log(t+\sqrt{t^2-1}))\qquad\mathrm{for}\,\,\,t\ge1.\]
and letting \(t\) be any of the  ratios in \tref{SameRatios}, we get
\mymargin{WhatVIs}\begin{equation}\label{WhatVIs}
\begin{split}
v
\left(
\frac
{x_k-x_{k-2}}
{2\sqrt{(x_k-x_{k-1}) (x_{k-1}-x_{k-2})}}
\right)
&=
v
\left(
\frac
{l_k+l_{k-1}}
{2\sqrt{l_k l_{k-1}}}
\right) \\
&=
v_1(|\log(\sqrt{l_k}) - \log(\sqrt{l_{k-1}})|) \\
&=
v_1(\frac12|\log({l_k}) - \log({l_{k-1}})|) \\
=
v
\left(
\frac
{y_k+y_{k-1}}
{2\sqrt{y_k y_{k-1}}}
\right)
&=
v_1(\frac12|\log({y_k}) - \log({y_{k-1}})|) \\
&=
v_1(|z_k-z_{k-1}|).
\end{split}
\end{equation}

The function \(v_1\) has a maximum at \(0\) with value \(\pi\) and
decreases to \(0\) as its argument goes to \(\pm\infty\).  Thus the
values in \tref{WhatVIs} measure the equality of of two
consecutive intervals.  We call the value in \tref{WhatVIs} the
{\itshape equality} of the lengths of the intervals.  The
equality is \(\pi\) if the lengths are the same, and the
equality decreases to 0 as the ratio of the lengths gets farther
from 1.

\begin{lemma}[S-L3]
The following holds
\[
J_n
=
\int\limits_{-\infty}^{+\infty} \cdots \int\limits_{-\infty}^{+\infty}
\frac
{dt_1 \cdots dt_{2n-1}}
{\sqrt
{
(1+t_1^2)
(1+(t_2-t_1)^2)
\cdots
(1+(t_{2n-1}-t_{2n-2})^2)
(1+t_{2n-1}^2)
}}
\]
for any natural \(n\) and 
\[c_1 2^{3n-1}(2n)! \le J_n \le c_2 2^{3n-1}(2n)!.
\]
\end{lemma}

\begin{proof}

Remembering that \(x_n=1\) and using Lemma \ref{JacCalc}, we have
\[
\begin{split}
&\left(
\prod_{k=1}^n 
\frac
{1}
{x_k-x_{k-1}}
\right)
dx_1\cdots dx_{n-1} \\
=&
\left(
\prod_{k=1}^{n-1} 
\frac
{1}
{x_k-x_{k-1}}
\right)
\frac{1}{1-x_{n-1}}
dx_1\cdots dx_{n-1} \\
=&
\left(
\prod_{k=1}^{n-1} 
\frac
{1}
{x_k-x_{k-1}}
\right)
\frac{(1-x_{n-1})^n}{1-x_{n-1}}
dy_1\cdots dy_{n-1} \\
=&
\left(
\prod_{k=1}^{n-1} 
\frac
{1-x_{n-1}}
{x_k-x_{k-1}}
\right)
\frac{1-x_{n-1}}{1-x_{n-1}}
dy_1\cdots dy_{n-1} \\
=&
\frac
{dy_1\cdots dy_{n-1}}
{y_1 \cdots y_{n-1}}.
\end{split}
\]

Recall that
\[
\frac{y_k+y_{k-1}}{2\sqrt{y_ky_{k-1}}}
=
\frac{x_k-x_{k-2}}
{2\sqrt{(x_k-x_{k-1}) (x_{k-1} - x_{k-2})}}.
\]

We now have
\mymargin{FirstJn}\begin{equation}\label{FirstJn}
J_n
=
\int\limits_0^{+\infty} \cdots \int\limits_0^{+\infty} 
\prod_{k=1}^n v
\left(
\frac{y_k+y_{k-1}}{2\sqrt{y_ky_{k-1}}}
\right)
\frac
{dy_1\cdots dy_{n-1}}
{y_1 \cdots y_{n-1}}.
\end{equation}
Taking into account \(y_0=y_n=1\), we get
\[
J_n
=
\int\limits_0^{+\infty} \cdots \int\limits_0^{+\infty} 
v
\left(
\frac{y_1+1}{2\sqrt{y_1}}
\right)
v
\left(
\frac{1+y_{n-1}}{2\sqrt{y_{n-1}}}
\right)
\prod_{k=2}^{n-1} v
\left(
\frac{y_k+y_{k-1}}{2\sqrt{y_ky_{k-1}}}
\right)
\frac
{dy_1\cdots dy_{n-1}}
{y_1 \cdots y_{n-1}}.
\]
This verifies the first line of the proof of Lemma 3 in
\cite{shavgulidze}. 

We have \[
\begin{split}
v_1(a-b)
=&
\int\limits_{-\infty}^{+\infty}
\frac{dz}
{\sqrt{
(1+z^2)(1+(a-b-z)^2)
}
} \\
=&
\int\limits_{-\infty}^{+\infty}
\frac{dw}
{\sqrt{
(1+(w-b)^2)(1+(a-w)^2)
}
}\qquad\mathrm{letting}\quad w=z+b. \\
\end{split}
\]
Since we know \(v_1(-\tau) = v_1(\tau)\), the above is also the
formula for \(v_1(b-a)\).  

We now define \(t_{2k} = \frac{1}{2}\log(y_k) = z_k\).  We pick up
the odd subscripts by 
letting our variable of integration for
\(v_1(|t_{2k} - t_{2k-2}|)\) be \(t_{2k-1}\), so that we get 
\mymargin{TheOddSub}\begin{equation}\label{TheOddSub}
v_1(|t_{2k}-t_{2k-2}|) =
\int\limits_{-\infty}^{+\infty}
\frac{dt_{2k-1}}
{\sqrt
{
(1+(t_{2k-1} - t_{2k-2})^2)
(1+(t_{2k} - t_{2k-1})^2)
}
}.
\end{equation}
This disagrees with the content of the proof of Lemma 3 in
\cite{shavgulidze}, but that seems to be a misprint.  The above
agrees with the top of Page 8 of \cite{shavgulidze}.

Using \tref{TheOddSub} and
Lemma \ref{JacCalc}, we can replace
\tref{FirstJn}  by
\[
J_n =
2^{n-1}
\int\limits_{-\infty}^{+\infty} \cdots \int\limits_{-\infty}^{+\infty}
\frac{dt_1\,dt_2\, \cdots dt_{2n-1}}
{
\sqrt{
\textstyle{
\prod_{k=1}^{2n} (1+(t_k-t_{k-1})^2)
}
}
}.
\]
With \(t_0=t_{2n}=0\), this agrees with the statement of the lemma
we are proving.

The last provision of the lemma follows directly from Lemma (S-L1)..
\end{proof}

\begin{lemma}[S-L4] For each \(\epsilon>0\) with \(\epsilon<1\),
there exists \(c_3>0\) 
so that 
\[
v
\left( \frac{y_1+a}{2\sqrt{y_1a}}
\right)
v
\left( \frac{a+y_2}{2\sqrt{ay_2}}
\right)
\le
c_3
v
\left( \frac{y_1+y_2}{2\sqrt{y_1 y_2}}
\right)
\]
for all \(a\), \(y_1\), \(y_2\) satisfying \(\epsilon \le a<1\),
\(y_1>0\), \(y_2>0\), \(y_1+y_2\le 1\).  \end{lemma}

The lemma is to be interpreted while remembering that \(v\) measures
the equality of the lengths two intervals where the value
decreases as the ratio of the lengths varies farther from 1.  The
lemma relates the equalities of the three pairs in a triple of
intervals if the length of the middle interval is at least
\(\epsilon\).

\begin{proof} Let \(r=-\frac{1}{2}\log(\epsilon)\).

We know that \(v_1\) is positive, even, continuous and is decreasing
on \([0,\infty)\).  Further, its maximum is at 0 where it has the
value \(\pi\).

From Corollary \ref{LimVOne}, there is an \(R>0\) so that
\(v_1(t-r)\le 2v_1(t)\) on all of \([R, \infty)\).  We can choose
\(R>r\).  Since \(v_1\) is decreasing on \([0,\infty)\) and
increasing on \((-\infty, 0]\), we have \(v_1(t-\tau)\le 2v_1(t)\)
for all \(t\) with \(|t|\ge R\) and all \(\tau\in [0,r]\).

Since \(v_1(R)\) is the minimum of \(v_1\) on \([-R,R]\), we can set
\(c^*\) to be the larger of \(2\) and \(\pi/v_1(R)\) and will have
that \(v_1(t-\tau)\le c^*v_1(t)\) for all \(t\in \R\) and \(\tau\in
[0,r]\).  Since \(v_1\) is even, we have \(v_1(|t-\tau|) \le
c^*v_1(t)\) for all \(t\in \R\) and \(\tau\in [0,r]\).

We let \(c_3=\pi(c^*)^2\).

Let \(t_i=-\frac{1}{2}\log(y_i)\), \(i=1,2\),  and
\(\alpha=-\frac{1}{2} \log(a)\).  Since \(\epsilon\le a\le 1\), we
have \(\alpha\in [0,r]\).

From \tref{WhatVIs}, we are asked to show 
\[
v_1(|t_1-\alpha|) v_1(|t_2-\alpha|) \le c_3 v_1(|t_2-t_1|).
\]

Let \(w=\min\{t_1, t_2\}\) and \(z=\max\{t_1, t_2\}\).  We have 
\(z-w\ge0\) and
\[
v_1(|w-\alpha|)v_2(|z-\alpha|) \le (c^*)^2v_1(w)v_1(z) \le
\pi(c^*)^2v_1(z) \le c_3 v_1(z-w)
\]
which is what we need to show.
\end{proof}

\subsection{A definition}  Let \(\vartheta\) be the characteristic
function on \([0,1]\).  That is, it takes the value \(1\) on
\([0,1]\) and 0 otherwise.

\begin{lemma}[S-L5] The following holds
\[
\begin{split}
&\lim_{n\to\infty}
\int\limits_0^1 \int\limits_{x_1}^1\cdots\int\limits_{x_{n-2}}^1
(1-\vartheta(\frac 1 \epsilon \underset{1\le k\le
n}{\max}(x_k-x_{k-1})))
u_n(x_1, \ldots, x_{n-1})
dx_1dx_2\ldots dx_{n-1} \\  &= 0
\end{split}
\]
for any positive \(\epsilon < 1\).
\end{lemma}

If \(r\) is \(\underset{1\le k\le n}{\max}(x_k-x_{k-1})\), then
\(\vartheta(r/ \epsilon)\) is 0 if and only if \(r>\epsilon\) and
thus 1 if and only if \(r>\epsilon\).  Thus the integral in the
statement is the restriction of the integral of \(u_n\) to the
partitions of \([0,1]\) in \(D_n\) that have at least one of the
lengths greater than \(\epsilon\).

\begin{proof}  Let 
\[I=
\int\limits_0^1 \int\limits_{x_1}^1\cdots\int\limits_{x_{n-2}}^1
(1-\vartheta(\frac 1 \epsilon \underset{1\le k\le
n}{\max}(x_k-x_{k-1})))
u_n(x_1, x_2, \ldots, x_{n-1})
dx_1\,dx_2\ldots dx_{n-1}
\]
and
\[
I_k=
\int\limits_0^1 \int\limits_{x_1}^1\cdots\int\limits_{x_{n-2}}^1
(1-\vartheta(\frac 1 \epsilon (x_k-x_{k-1})))
u_n(x_1, x_2, \ldots, x_{n-1})
dx_1\,dx_2\ldots dx_{n-1}.
\]
Each \(I_k\) integrates \(u_n\)
over the partitions in which the length of the \(k\)-th interval
excedes \(\epsilon\).  We have \(I\le \sum_{k=1}^n I_k\).  We work
to estimate \(I_k\).

Let \(D_{k,\epsilon}\) be the subset of \(D_n\) for which
\(x_k-x_{k-1}> \epsilon\).  We calculate \(I_k\) by integrating
\(u_n\) over \(D_{k,\epsilon}\).  For \((x_1, \ldots, x_{n-1})\in
D_{k,\epsilon}\) we set 
\[
\begin{split}
r&=x_k-x_{k-1}, \\
y'_{-1} &=x_{-1} = x_{n-1}-1, \\
y'_0 &= x_0 = 0, \\
y'_1 &= x_1, \\
\vdots& \\
y'_{k-1} &= x_{k-1}, \\
y'_k &= x_{k+1}-r, \\
\vdots& \\
y'_{n-2} &= x_{n-1}-r, \\
y'_{n-1} &= 1-r. 
\end{split}
\]
Note that for \(j\ge k\), we have \(y'_j=x_{j+1}-x_k + x_{k-1}\).
The transformation \[(x_1,\dots, x_{n-1})\mapsto (y'_1, \ldots,
y'_{k-1}, r, y'_k, \ldots, y'_{n-2})\] is linear with triangular
matrix with ones on the diagonal.  Thus the transformation has
Jacobian one.

Now we let \(y_j = y'_j/(1-r)\) for \(j\in \{-1, 0, 1, \ldots,
n-1\}\).  The transformation
\[(r, y'_1, \ldots, y'_{n-2}) \mapsto (r, y_1, \ldots, y_{n-2})
\]
has Jacobian \((1-r)^{n-2}\).

The \(y'_j\) divide the interval \([0,1-r]\) into segments that
correspond to the segments that the \(x_j\) divide \([0,l]\) into,
but with the segment \([x_{k-1}, x_k]\) removed.  Thus the
differences 
\[
\begin{split}
y'_j-y'_{j-1} 
&=
\begin{cases} 
x_j-x_{j-1}, & j<k \\
x_{j+1}-x_j, & j\ge k,
\end{cases} \\
y'_j-y'_{j-2} 
&= 
\begin{cases} 
x_j-x_{j-2}, & j<k \\
(x_{k+1}-x_k)+(x_{k-1}-x_{k-2}), &j=k, \\
x_{j+1}-x_{j-1}, & j\ge k+1.
\end{cases}
\end{split}
\]

\[
\begin{split}
u_{1,n}(x_1, \ldots, x_{n-1}) 
=
\prod_{j=1}^n 
\frac
{1}
{x_j-x_{j-1}}
&v\left(
\frac
{x_j-x_{j-2}}
{2\sqrt{(x_j-x_{j-1}) (x_{j-1}-x_{j-2})}}
\right) \\
=
\prod_{j=1}^{k-1}
\frac
{1}
{x_j-x_{j-1}}
&v\left(
\frac
{x_j-x_{j-2}}
{2\sqrt{(x_j-x_{j-1}) (x_{j-1}-x_{j-2})}}
\right) \\
\cdot
\frac
{1}
{x_k-x_{k-1}}
&v\left(
\frac
{x_k-x_{k-2}}
{2\sqrt{(x_k-x_{k-1}) (x_{k-1}-x_{k-2})}} 
\right) \\
\cdot
\frac
{1}
{x_{k+1}-x_{k}}
&v\left(
\frac
{x_{k+1}-x_{k-1}}
{2\sqrt{(x_{k+1}-x_{k}) (x_{k}-x_{k-1})}}
\right) \\
\prod_{j=k+2}^{n}
\frac
{1}
{x_j-x_{j-1}}
&v\left(
\frac
{x_j-x_{j-2}}
{2\sqrt{(x_j-x_{j-1}) (x_{j-1}-x_{j-2})}}
\right)
\end{split}
\]

From Lemma (S-L4), we know 
\[
\begin{split}
v\left(
\frac
{x_k-x_{k-2}}
{2\sqrt{(x_k-x_{k-1}) (x_{k-1}-x_{k-2})}} 
\right)
&v\left(
\frac
{x_{k+1}-x_{k-1}}
{2\sqrt{(x_{k+1}-x_{k}) (x_{k}-x_{k-1})}}
\right) \\
\le c_3
&
v\left(
\frac
{(x_{k+1}-x_k)+(x_{k-1}-x_{k-2})}
{2\sqrt{(x_{k+1}-x_{k}) (x_{k-1}-x_{k-2})}}
\right) \\
= c_3
&v\left(
\frac
{y'_k-y'_{k-2}}
{2\sqrt{(y'_{k}-y'_{k-1}) (y'_{k-1}-y'_{k-2})}}
\right).
\end{split}
\]
Making the other substitutions we list above and being careful with
our running index \(j\), we get
\[
\begin{split}
u_{1,n}(x_1, \ldots, x_{n-1}) 
\le
\prod_{j=1}^{k-1}
\frac
{1}
{y'_j-y'_{j-1}}
v
&\left(
\frac
{y'_j-y'_{j-2}}
{2\sqrt{(y'_j-y'_{j-1}) (y'_{j-1}-y'_{j-2})}}
\right) \\
\cdot
\frac
{1}
{r}
\frac
{1}
{y'_{k}-y'_{k-1}}  c_3
v
&\left(
\frac
{y'_{k}-y'_{k-2}}
{2\sqrt{(y'_{k}-y'_{k-1}) (y'_{k-1}-y'_{k-2})}}
\right) \\
\prod_{j=k+1}^{n-1}
\frac
{1}
{y'_j-y'_{j-1}}
v
&\left(
\frac
{y'_j-y'_{j-2}}
{2\sqrt{(y'_j-y'_{j-1}) (y'_{j-1}-y'_{j-2})}}
\right) \\
=\frac{c_3}{r} 
u_{1,n-1}
&(y'_1, \ldots, y'_{n-2}).
\end{split}
\]

We have \(y'_j=(1-r)y_j\) for \(-1\le j\le n-1\) and we have
\[
\begin{split}
dy'_j &= (1-r)dy_j, \\
\frac
{1}
{y'_j-y'_{j-1}}
&=
\frac
{1}
{(y_j-y_{j-1})(1-r)}, \qquad\mathrm{and} \\
v
\left(
\frac
{y'_j-y'_{j-2}}
{2\sqrt{(y'_j-y'_{j-1}) (y'_{j-1}-y'_{j-2})}}
\right)
&=
v
\left(
\frac
{y_j-y_{j-2}}
{2\sqrt{(y_j-y_{j-1}) (y_{j-1}-y_{j-2})}}
\right)
\end{split}
\]
for every \(j\) with \(1\le j\le n-1\).

Now we note that 
\[
\begin{split}
u_{1,n-1}(y'_1, \ldots, y'_{n-2})
=
&\prod_{j=1}^{n-1}
\frac
{1}
{y'_j-y'_{j-1}}
v
\left(
\frac
{y'_j-y'_{j-2}}
{2\sqrt{(y'_j-y'_{j-1}) (y'_{j-1}-y'_{j-2})}}
\right) \\
=
\frac{1}{(1-r)^{n-1}}
&\prod_{j=1}^{n-1}
\frac
{1}
{y_j-y_{j-1}}
v
\left(
\frac
{y_j-y_{j-2}}
{2\sqrt{(y_j-y_{j-1}) (y_{j-1}-y_{j-2})}}
\right) \\
=
\frac{1}{(1-r)^{n-1}}
&u_{1,n-1}(y_1, \ldots, y_{n-2}).
\end{split}
\]

Since 
\[(x_1,\dots, x_{n-1})\mapsto (y'_1, \ldots,
y'_{k-1}, r, y'_k, \ldots, y'_{n-2})\]
has Jacobian one, we get
\[
I_k
\le
\frac{c_3}{J_n}
\int\limits_\epsilon^1 \frac 1 r 
\left[
\int\limits _0^{1-r} \int\limits_{y'_1}^{1-r} \cdots
\int\limits_{y'_{n-3}}^{1-r}
u_{1,n-1}(y'_1, \ldots, y'_{n-2})
dy'_1\,dy'_2\,\cdots dy'_{n-2}
\right]
dr.
\]

Since 
\[(r, y'_1, \ldots, y'_{n-2}) \mapsto (r, y_1, \ldots, y_{n-2})
\]
is diagonal, we can just make direct substitutions to get
get
\[
\begin{split}
I_k
&\le
\frac{c_3}{J_n}
\int\limits_\epsilon^1 \frac 1 r 
\frac{(1-r)^{n-2}}{(1-r)^{n-1}}dr
\left[
\int\limits _0^{1} \int\limits_{y_1}^{1} \cdots
\int\limits_{y_{n-3}}^{1}
u_{1,n-1}(y_1, \ldots, y_{n-2})
dy_1\,dy_2\,\cdots dy_{n-2}
\right] \\
&=
\frac{c_3}{J_n}
\int\limits_\epsilon^1 \frac 1 {r(1-r)}
\left[
\int\limits _0^{1} \int\limits_{y_1}^{1} \cdots
\int\limits_{y_{n-3}}^{1}
u_{1,n-1}(y_1, \ldots, y_{n-2})
dy_1\,dy_2\,\cdots dy_{n-2}
\right].
\end{split}
\]

Unfortunately, this differs significantly from what appears at this
point in the proof of Lemma 5 of \cite{shavgulidze}.  Any help at
this point would be appreciated.  \end{proof}

\begin{lemma}[S-L6]
For any \(r>1\), the following holds
\[
\begin{split}
&\lim_{n\to\infty}
\int\limits_0^1 \int\limits_{x_1}^1\cdots\int\limits_{x_{n-2}}^1
\vartheta\left[
\frac 1 r \underset{1\le k\le
n}{\min}
\left(
\frac{l_k+l_{k-1}}{2\sqrt{l_kl_{k-1}}}
\right)
\right]
u_n(x_1, \ldots, x_{n-1})
dx_1dx_2\ldots dx_{n-1} \\  &= 0
\end{split}
\]
where \(l_k=x_k-x_{k-1}\) for \(1\le k \le n\).
\end{lemma}

\begin{proof}
As in Lemma (S-L5), we let \(I\) be the integral in the statement of
the lemma, and let 
\[
I_k
=
\lim_{n\to\infty}
\int\limits_0^1 \int\limits_{x_1}^1\cdots\int\limits_{x_{n-2}}^1
\vartheta\left[
\frac 1 r 
\left(
\frac{l_k+l_{k-1}}{2\sqrt{l_kl_{k-1}}}
\right)
\right]
u_n(x_1, \ldots, x_{n-1})
dx_1dx_2\ldots dx_{n-1}
\]  Now \(I\le \sum_{k=1}^n I_k\).  We fix \(k\) and work on \(I_k\).

For \(1\le j \le n\), we make the same substitutions \(y_j =
l_j/(1-x_{n-1})\) and \(t_{2j} = \frac1 2 \log(y_j)\) as in Lemma
(S-L3).  First, this makes
\[
\left(
\frac{l_k+l_{k-1}}{2\sqrt{l_kl_{k-1}}}
\right)
=
\left(
\frac{y_k+y_{k-1}}{2\sqrt{y_k y_{k-1}}}
\right).
\]
But from Section \ref{CalcSubSec} we know that 
\[
\left(
\frac
{y_k+y_{k-1}}
{2\sqrt{y_k y_{k-1}}}
\right)
+
\sqrt{
\displaystyle{
\left(
\frac
{y_k+y_{k-1}}
{2\sqrt{y_k y_{k-1}}}
\right)
^2
}-1
}
=
\max\left\{
\sqrt{\displaystyle{\frac {y_k} {y_{k-1}}}}, 
\sqrt{\displaystyle{\frac {y_{k-1}} {y_k}}}
\right\}
\]
and 
\[
\log
\left[
\left(
\frac
{y_k+y_{k-1}}
{2\sqrt{y_k y_{k-1}}}
\right)
+
\sqrt{
\displaystyle{
\left(
\frac
{y_k+y_{k-1}}
{2\sqrt{y_k y_{k-1}}}
\right)
^2
}-1
}
\right]
=
\frac 1 2 |y_k - y_{k-1}| = |t_{2k}-t_{2k-2}|.
\]
Since \(\log(t+\sqrt{t^2-1})\) is increasing for \(t\ge1\), it follows that 
\[
\left(
\frac{l_k+l_{k-1}}{2\sqrt{l_kl_{k-1}}}
\right) \le r
\]
if and only if \(|t_{2k}-t_{2k-2}|\le a = \log(r+\sqrt{r^2-1})\).

Thus from the transformations used in the proof of Lemma (S-L3)
\[
I_k
=
\frac{2^{n-1}}{J_n}
\int\limits_{-\infty}^{+\infty} \cdots \int\limits_{-\infty}^{+\infty}
\frac{\vartheta
(\frac 1 r
|t_{2k}-t_{2k-2}|)
dt_1\,dt_2\, \cdots dt_{2n-1}}
{
\sqrt{
\textstyle{
\prod_{j=1}^{2n} (1+(t_j-t_{j-1})^2)
}
}
}
\]
where we remember that \(t_0=t_{2n}=0\).

Now \(|t_{2k}-t_{2k-3}| \le |t_{2k}-t_{2k-2}|+|t_{2k-2}-t_{2k-3}|
\le a+|t_{2k-2}-t_{2k-3}|\).  If \(p\), \(q\) and \(a\) are non-negative,
and \(p\le q+a\), then 
\mymargin{OddIneq}\begin{equation}\label{OddIneq}
1+p^2 \le (1+q^2)4(1+a)^2
\end{equation}
which can be
verified by noting that \(1+p^2\le 1+(q+a)^2\), replacing the left
side of \tref{OddIneq} with \(1+(q+a)^2\), multiplying out, bringing
all terms to the right and noting that the \(-2qa\) that shows up
can be combined with one \(q^2\) and one \(a^2\) that are  on the
right to give a term of the form \((a-q)^2\).  The remaining
terms are positive.  In fact \(1+p^2 \le (1+q^2)2(1+a)^2\) holds,
but we will be satisfied with \tref{OddIneq} since \(4\) has a nicer
square root than \(2\).  This gives us
\[
1+(t_{2k}-t_{2k-3})^2 \le (1+(t_{2k-2}-t_{2k-3})^2)4(1+a)^2
\]
which then gives
\[
\frac{1}
{\sqrt{1+(t_{2k-2}-t_{2k-3})^2}}
\le
\frac{2(1+a)}
{\sqrt{1+(t_{2k}-t_{2k-3})^2}}.
\]

Letting \(P_{i,j}=\sqrt{1+(t_i-t_j)^2}\), we have
\[
\int\limits_{-\infty}^{+\infty}
\frac{dt_{2k-1}}{P_{2k,2k-1}P_{2k-1,2k-2}}
=
v_1(|t_{2k}-t_{2k-1}|) \le \pi
\]
so we can get the following estimates:
\[
\begin{split}
I_k
&=
\frac{2^{n-1}}{J_n}
\int\limits_{-\infty}^{+\infty} \cdots \int\limits_{-\infty}^{+\infty}
\frac{\vartheta
(\frac 1 a
|t_{2k}-t_{2k-2}|)
dt_1\,dt_2\, \cdots dt_{2n-1}}
{
\textstyle{
\prod_{j=1}^{2n}}
P_{j,j-1}
} \\
&=
\frac{2^{n-1}}{J_n}
\int\limits_{-\infty}^{+\infty} \cdots \int\limits_{-\infty}^{+\infty}
\frac{\vartheta
(\frac 1 a
|t_{2k}-t_{2k-2}|)
dt_{2k-1}dt_1 \cdots dt_{2k-2}\,dt_{2k}\cdots dt_{2n-1}}
{
\textstyle{
P_{2k,2k-1}
P_{2k-1,2k-2}
(\prod_{j=1}^{2k-2}
P_{j,j-1})
(\prod_{j=2k+1}^{2n}}
P_{j,j-1})
} \\
&=
\frac{2^{n-1}}{J_n}
\int\limits_{-\infty}^{+\infty} \cdots \int\limits_{-\infty}^{+\infty}
\frac{\vartheta
(\frac 1 a
|t_{2k}-t_{2k-2}|)
v_1(|t_{tk}-t_{2k-2}|)
dt_1 \cdots dt_{2k-2}\,dt_{2k}\cdots dt_{2n-1}}
{
\textstyle{
(\prod_{j=1}^{2k-2}
P_{j,j-1})
(\prod_{j=2k+1}^{2n}}
P_{j,j-1})
} \\
&\le
\frac{2^{n-1}\pi}{J_n}
\int\limits_{-\infty}^{+\infty} \cdots \int\limits_{-\infty}^{+\infty}
\frac{\vartheta
(\frac 1 a
|t_{2k}-t_{2k-2}|)
dt_1 \cdots dt_{2k-2}\,dt_{2k}\cdots dt_{2n-1}}
{
\textstyle{
(\prod_{j=1}^{2k-2}
P_{j,j-1})
(\prod_{j=2k+1}^{2n}}
P_{j,j-1})
} \\
&=
\frac{2^{n-1}\pi}{J_n}
\int\limits_{-\infty}^{+\infty} \cdots \int\limits_{-\infty}^{+\infty}
\frac{\vartheta
(\frac 1 a
|t_{2k}-t_{2k-2}|)
dt_1 \cdots dt_{2k-2}\,dt_{2k}\cdots dt_{2n-1}}
{
\textstyle{
(\prod_{j=1}^{2k-3}
P_{j,j-1})
P_{2k-2,2k-3}
(\prod_{j=2k+1}^{2n}}
P_{j,j-1})
} \\
&\le
\frac{2^{n}(1+a)\pi}{J_n}
\int\limits_{-\infty}^{+\infty} \cdots \int\limits_{-\infty}^{+\infty}
\frac{\vartheta
(\frac 1 a
|t_{2k}-t_{2k-2}|)
dt_1 \cdots dt_{2k-2}\,dt_{2k}\cdots dt_{2n-1}}
{
\textstyle{
(\prod_{j=1}^{2k-3}
P_{j,j-1})
P_{2k,2k-3}
(\prod_{j=2k+1}^{2n}}
P_{j,j-1})
} \\
&=
\frac{2^{n}(1+a)\pi}{J_n}
\int\limits_{-\infty}^{+\infty} \cdots \int\limits_{-\infty}^{+\infty}
\frac{\vartheta
(\frac 1 a
|t_{2k}-t_{2k-2}|)
dt_{2k-2}dt_1 \cdots dt_{2k-3}\,dt_{2k}\cdots dt_{2n-1}}
{
\textstyle{
(\prod_{j=1}^{2k-3}
P_{j,j-1})
P_{2k,2k-3}
(\prod_{j=2k+1}^{2n}}
P_{j,j-1})
} \\
&=
\frac{2^{n}(1+a)\pi}{J_n}
\int\limits_{-\infty}^{+\infty} \cdots \int\limits_{-\infty}^{+\infty}
\int\limits_{t_{2k}-a}^{t_{2k}+a}
\frac{
dt_{2k-2}dt_1 \cdots dt_{2k-3}\,dt_{2k}\cdots dt_{2n-1}}
{
\textstyle{
(\prod_{j=1}^{2k-3}
P_{j,j-1})
P_{2k,2k-3}
(\prod_{j=2k+1}^{2n}}
P_{j,j-1})
} \\
&=
\frac{2^{n+1}a(1+a)\pi}{J_n}
\int\limits_{-\infty}^{+\infty} \cdots \int\limits_{-\infty}^{+\infty}
\frac{
dt_1 \cdots dt_{2k-3}\,dt_{2k}\cdots dt_{2n-1}}
{
\textstyle{
(\prod_{j=1}^{2k-3}
P_{j,j-1})
P_{2k,2k-3}
(\prod_{j=2k+1}^{2n}}
P_{j,j-1})
} \\
&=
\frac{2^{n+1}a(1+a)\pi}{J_n}
\frac{J_{n-1}}{2^{n-2}}
=8\pi a(1+a) \frac{J_{n-1}}{J_n}
 \\
&\le8\pi a(1+a) \frac{c_2 2^{3n-4}(2n-2)!}
{c_1 2^{3n-1}(2n)!}
=\frac{c_2 \pi a(1+a)}{c_1(2n-1)(2n)}.
\end{split}
\]
This makes 
\[
I\le nI_k \le n \frac{c_2 \pi a(1+a)}{c_1(2n-1)(2n)}
=\frac{c_2 \pi a(1+a)}{c_1(2n-1)(2)}
\]
which goes to 0 as \(n\) goes to infinity.
\end{proof}

\section{The Wiener measure}

In this section a measure will be put on the space of functions
\(\mathrm{Diff}_+^1(I)\) and therefore on \(\mathrm{Diff}_{1,
\delta}^1(I)\).  Work will then have to be done to derive something
that is invariant under the action of \(\mathrm{Diff}_+^3(I)\).

The measure that we start with is the Wiener measure.  Our main
reference for the Wiener measure will be \S 3 of Chapter I of
\cite{MR0461643}.  This seems to be self contained in spite of it
being the third section.  My impression is that the two sections in
\cite{MR0461643} that come before are used later in \cite{MR0461643}
to define a Wiener-like measure in more general settings.

\subsection{Maps betwteen function spaces}

Wiener measure is defined on a linear space of functions rather than
a space of homeomorphisms.  So we need to move from one to the
other.

\subsubsection{The mappings}

Let \(C_0(I)\) be the linear space of continuous functions from
\([0,1]\) into \(\R\) that take \(0\) to \(0\).  Let
\(A:\mathrm{Diff}_+^1(I) \rightarrow C_0(I)\) be defined by
\mymargin{Forward}\begin{equation}\label{Forward}
A(f)(t) = \log(f'(t)) - \log(f'(0)), \qquad \forall t \in [0,1].
\end{equation}

We will typically use a letter like \(x\) for an element of
\(C_0(I)\) (to agree with our main reference \cite{MR0461643}) and
the argument for such an \(x\) will typically be \(s\), \(t\) or
\(\tau\). 

Let \(B:C_0(I)\rightarrow \mathrm{Diff}_+^1(I)\) be defined by 
\mymargin{Backward}\begin{equation}\label{Backward}
B(x )(t) = 
\frac{\int_0^t e^{x (\tau)}d\tau} {\int_0^1 e^{x (\tau)}d\tau}.
\end{equation}
We look at \(g=BA(f)\) for an \(f\in
\mathrm{Diff}_+^1(I)\).  Let \(L=\int_0^1 e^{x (\tau)}d\tau\) for
\(x (\tau)=A(f)(\tau) = \log(f'(\tau)) - \log(f'(0))\).  Now 
\[
g'(t) = \frac 1 L \frac{f'(t)}{f'(0)}.
\]
Thus \(g'\) and \(f'\) differ by a
multiplicative constant, and \(g\) and \(f\) both take \(0\) to
\(0\). By integrating, 
\(f\) and \(g\) differ by a multiplicative constant.  But
\(f(1)=g(1)=1\) and so the multiplicative constant is one.  It is
even easier to show that \(AB(x )=x \) for any \(x \in C_0(I)\).
Thus \(B\) and \(A\) are mutual inverses.

\subsubsection{Topologies}

We put the uniform norm on \(C_0(I)\) and the associated uniform
topology.  The topology on \(\mathrm{Diff}_+^1(I)\) is the \(C^1\)
tolopolgy given by the norm 
\mymargin{COneNorm}\begin{equation}\label{COneNorm}
\|f\|_1 = \|f\|_\infty + \|f'\|_\infty.
\end{equation}
Note that on \(\mathrm{Diff}_+^1(I)\) this is topologically
equivalent to the usual \(\mathrm{Diff}^1\) norm 
\mymargin{DiffOneNorm}\begin{equation}\label{DiffOneNorm}
\|f\|_\infty+\|f'\|_\infty+\|(f^{-1})'\|_\infty
\end{equation}
since the derivative of any function in \(\mathrm{Diff}_+^1(I)\) is
bounded away from zero and thus, using the chain rule, two functions
in \(\mathrm{Diff}_+^1(I)\) can be kept close in the norm
\tref{DiffOneNorm} by keeping them sufficiently close in the norm
\tref{COneNorm}.  Metric properties of the two norms (completeness,
say) are not the same.

Since derivatives in \(\mathrm{Diff}_+^1(I)\) are bounded away from
zero, it is clear that \(A\) is continuous.  That \(B\) is
continuous follows from 
\[
(B(x ))'(t) = 
\frac{e^{x (t)}} {\int_0^1 e^{x (\tau)}d\tau}.
\] 
(showing that close functions in \(C_0(I)\) go to functions in
\(\mathrm{Diff}_+^1(I)\) that have close derivatives),
and the fact that all functions in \(\mathrm{Diff}_+^1(I)\) take 0
to 0.

This shows that \(A\) and \(B\) are mutually inverse homeomorphisms.

\subsection{Defining the Wiener measure}\mylabel{WienerDef}

We now put the Wiener measure \(w\) on \(C_0(I)\).  This is defined
in \cite{MR0461643} by first defining \(w\) on a restricted class of
sets in \(C_0(I)\) and then extending to the Borel sets in
\(C_0(I)\).  (Note that what we call \(C_0(I)\) is denoted by
\(C[0,1]\) in \cite{MR0461643}.)

Let \(\mathbf t=(t_1, t_2, \ldots, t_n)\) with
\(0<t_1<t_2<\cdots <t_n \le 1\) be given along with a Borel
subset \(E\) of \(\R^n\).  The set 
\[
\mathrm{Cyl}(\mathbf t, E) = 
\{x\in C_0(I) \mid (x(t_1), x(t_2), \ldots, x(t_n)) \in E\}
\]
will be called a {\itshape cylinder set}.  Note that this is an
inverse image.  If \(\pi_{\mathbf t}:C_0(I)\rightarrow \R^n\) is
the evaluation map at \(\mathbf t\) in that 
\[
\pi_{\mathbf t}(x) = (x(t_1), x(t_2), \ldots, x(t_n)),
\]
then \(\mathrm{Cyl}(\mathbf t, E) = \pi^{-1}_{\mathbf t}(E)\).  Since
\(C_0(I)\) has the uniform metric, \(\pi_{\mathbf t}\) is continuous.
Thus \(\mathrm{Cyl}(\mathbf t, E)\) is a Borel set in \(C_0(I)\).

The collection of all cylinder sets is not a \(\sigma\)-algebra, but
is a algebra that we will denote by \(\scr{R}\).  The \(\sigma\)-algebra
that \(\scr{R}\) generates is the \(\sigma\)-algebra \(\scr{B}\) of
Borel sets in \(C_0(I)\).  See Theorem 3.3 and the exercises 12 and
13 that immediately preceed it in \cite{MR0461643}.

With \(\mathbf t\) and \(E\) as above, we set \(t_0=0\) and we define
\[
\begin{split}
&w(\mathrm{Cyl}(\mathbf t, E)) = \\
&\frac{1}
{\sqrt{\prod_{k=1}^n 2\pi(t_k-t_{k-1})}}
\int_E 
\exp\left({-\frac 1 2\sum_{k=1}^n
\frac{(u_k-u_{k-1})^2}{t_k-t_{k-1}}}
\right)
du_1\cdots du_n,
\end{split}
\]
where we take \(u_0=0\) in the integrand.

It is shown in Theorem 3.1 of \cite{MR0461643} that \(w\) extends to
a \(\sigma\)-additive measure on the \(\sigma\)-algebra generated by
\(\scr{R}\) which as mentioned above is \(\scr{B}\), the Borel sets
in \(C_0(I)\).

\subsection{Basic facts about Wiener measure}

\subsubsection{}\mylabel{FirstNormalFact}

Directly from the definition, for fixed \(0<t\le 1\) and \(a\le
b\), we get 
\mymargin{FirstNormForm}\begin{equation}\label{ForstNormForm}
w(\{x\in C_0(I)\mid a\le x(t) \le b\}) = 
\frac{1}{\sqrt{2\pi t}}
\int_a^b \exp\left(-\frac{u^2}{2t}\right) du.
\end{equation}
This is stated by saying that the functional \(f\) defined by
\(f(x)=x(t)\) is normally distributed with mean 0 and variance
\(t\).  If the elements of \(C_0(I)\) represent random walks that
start at 0, then \tref{ForstNormForm} gives the probability that
such a walk has reached a given interval \([a,b]\) at time \(t\).

It follows from \tref{ForstNormForm} that 
\mymargin{SecondNormForm}\begin{equation}\label{SecondNormForm}
w(C_0(I)) = 1.
\end{equation}

\subsubsection{}\mylabel{NormalInterval}

In \ref{FirstNormalFact}, the starting point of the walk is at
\(0\).  If we start at some other time \(s\), then we can ask about
the net change in the walk that starts there.  The distribution is
the same.  Specifically, for \(0\le s<t\le 1\) and \(a\le b\) in
\(\R\), let
\[
E(s,t;a,b) = \{x\in C_0(I)\mid a\le x(t)-x(s)\le b\}.
\]
We now have
\[
w(E(s,t;a,b))
=
\frac{1}{\sqrt{2\pi(t-s)}}
\int_a^b \exp\left(- \frac{u^2}{2(t-s)}\right) du.
\]
This straightforward calculation is carried out on Pages 37-38 of
\cite{MR0461643}.  We can say that the random variable \(x\mapsto
x(t)-x(s)\) is normally distributed with mean 0 and variance
\(t-s\).

\subsubsection{}\mylabel{IndepIntervals}

For fixed \(0\le s< t\le u< v\le 1\), the random variables \(x\mapsto
x(t)-x(s)\) and \(x\mapsto x(v)-x(u)\) are independent.  This is
given as Exercise 10 on Page 38 of \cite{MR0461643}, and is an
imitation of the calculations that go into the fact in \S
\ref{NormalInterval}.  It says that for intervals \([a,b]\) and
\([c,d]\) in \(\R\), we have \[w(E(s,t;a,b)\cap E(u,v;c,d)) =
w(E(s,t;a,b))\cdot w(E(u,v;c,d)).\]

\subsubsection{}\mylabel{PiSystems}

We will need a basic fact about measures that can be found as
Exercise 3.1.8 on Page 37 of \cite{MR1658777}.  This says that says
that if two measures are defined on a \(\sigma\)-algebra \(\scr{B}\)
(on a set \(E\)) that is generated (as a \(\sigma\)-algebra) by a
subset \(\scr{C}\) of \(\scr{B}\), if \(\scr{C}\) is closed under
pairwise intersection, and if the two measures agree on \(\scr{C}\)
and on the set \(E\) and are finite on \(E\), then the two measures
agree on all of \(\scr{B}\).  (This is also Theorem 3.5 on Page 13
of \cite{MR2125120}.)

\subsubsection{}\mylabel{DetermineWiener}

Note that the cylinder sets \(\mathrm{Cyl}(\mathbf t, E)\) of
\S\ref{WienerDef} are closed under pairwise intersection.  However,
we can do with fewer sets.  

The sets
\mymargin{NiceCorners}\begin{equation}\label{NiceCorners}
(-\infty, a_1] \times 
(-\infty, a_2] \times 
\cdots \times
(-\infty, a_n]
\end{equation}
generate the Borel sets in \(\R^n\).  Further the cylinder sets are
just inverse images (under the evaluation maps \(\pi_{\mathbf t}\))
of the Borel sets in \(\R^n\).  So cylinder sets
\(\mathrm{Cyl}(\mathbf t, E)\) with \(E\) in the form of
\tref{NiceCorners} generate the same \(\sigma\)-algebra as the full
collection of cylinder sets.  Let the cylinder sets
\(\mathrm{Cyl}(\mathbf t, E)\) with \(E\) in the form of
\tref{NiceCorners} be called the restricted cylinder sets.

The restricted cylinder sets are also closed under pairwise
intersection since the sets of the form in \tref{NiceCorners} are
closed under pairwise intersection.  Thus by \ref{PiSystems}, Wiener
measure is completely determined by its values on the restricted
cylinder sets and by the fact \tref{SecondNormForm} that the measure
of all of \(C_0(I)\) is 1.

If, for a given \(t\in [0,1]\), we let \(X_t\) be the random
variable on \(C_0(I)\) defined by \(X_t(x) = x(t)\), then the
previous paragraph can be restated by saying that Wiener measure is
determined by the probabilities of 
\mymargin{WienerTarget}\begin{equation}\label{WienerTarget}
(X_{t_1}\le a_1,\, X_{t_2}\le a_2,\, \ldots,\, X_{t_n}\le a_n).
\end{equation}

We let \(Y_1=X_{t_1}\), and \(Y_i = X_{t_i}-X_{t_{i-1}}\) for \(2\le
i\le n\).  The transformation from \((X_1, X_2, \ldots, X_n)\) to
\((Y_1, Y_2, \ldots, Y_n)\) has a simple invertible matrix with
determinant one, so probabilities for the \(X_i\) are determined by
those for the \(Y_i\) and vice versa.

Now we know that the \(Y_i\) are independent from
\ref{IndepIntervals} and that their distributions are given in
\ref{NormalInterval}.  Thus the probabilities of 
\[
(Y_{t_1}\le b_1,\, Y_{t_2}\le b_2,\, \ldots,\, Y_{t_n}\le b_n)
\]
are completely determined by \ref{NormalInterval} and
\ref{IndepIntervals}.  But by \tref{SecondNormForm} and
\ref{PiSystems}, this determines the probability of
\[
(Y_1,\, Y_2,\, \ldots, \, Y_n)\in E
\]
for any Borel set \(E\) in \(\R^n\).  In particular it determines
the probability of the image of \tref{WienerTarget}
under the transformation taking \((X_1, X_2, \ldots, X_n)\) to
\((Y_1, Y_2, \ldots, Y_n)\).  Thus the probability of
\tref{WienerTarget} is completely determined by \ref{NormalInterval},
\ref{IndepIntervals} and \tref{SecondNormForm}.  From this we can say
that the Wiener measure on the Borel sets in \(C_0(I)\) is
completely determined by \ref{NormalInterval},
\ref{IndepIntervals} and \tref{SecondNormForm}. 

\subsubsection{}\mylabel{TimeReverse}

For \(x\in C_0(I)\), let \(T(x)(t) = x(1-t)-x(1)\).  The
transformation \(T\) takes \(C_0(I)\) to itself, is self inverse, is
linear and no more than doubles distance and so is continuous.

Now \(0\le (1-t) < (1-s) \le 1\) and \(-b< -a\) and \[T(E(s,t;a,b))
= E(1-t, 1-s; -b, -a),\] so \[w(T(E(s,t;a,b))) = w(E(1-t, 1-s; -b,
-a)) = w(E(s,t;a,b)),\] because \((1-s)-(1-t)=t-s\) and
\(e^{-x^2/k}\) is even.  So \(T\) preserves \(w\) on the sets
\(E(s,t;a,b)\) and the measure defined by \(w'(S)=w(T(S))\)
satisfies \ref{NormalInterval}.  In addition
\(w'(C_0(I))=1\).  

Further, if \(0\le s< t\le u< v\le 1\), then \(0\le (1-v)< (1-u)\le
(1-t)< (1-s)\le 1\) and \(T\) carries \(E(s,t;a,b) \cap E(u,v;c,d)\)
to \(E(1-t, 1-s; -b, -a)\cap E(1-v, 1-u; -d, -c)\) at which point
applications of \ref{IndepIntervals} and the previous paragraph show
that \(w'\) also satisfies \ref{IndepIntervals}.  So by
\ref{DetermineWiener}, we have \(w'=w\).  This is often stated by
saying that the Wiener measure is preserved by time reversal.

\subsubsection{}\mylabel{HolderSupport}

Let \[C_\delta = \{x\in C_0(I)\mid \exists K \,\,\mathrm{s.t.}
\,\, \forall t,s, \,\, |x(t)-s(s)| \le K|t-s|^\delta \}.
\]

It follows from Lemma 3.1 in \cite{MR0461643} that \(C_\delta\) is a
Borel set in \(C_0(I)\).  This requires that one see that the sets
\(H_\alpha[a]\) in that lemma are unions of countably many
cylinder sets.
Theorem 3.2 in \cite{MR0461643} then states
\[
w(C_\delta) = 
\begin{cases}
1, &0 < \delta < \frac 1 2, \\
0, &\frac 1 2 < \delta.
\end{cases}
\]
Evidently nothing is known when \(\delta=\frac12\).

\subsection{The completion of Wiener measure}\mylabel{WienerComplete}

Recall that \(\scr{B}\) is the \(\sigma\)-algebra of Borel sets in
\(C_0(I)\) and is the \(\sigma\)-algebra on which the Wiener measure
is defined.  We let \(\overline{\scr{B}}\) denote the
\(w\)-completion of \(\scr{B}\).  That is \(S\) is in
\(\overline{\scr{B}}\) if and only if there are \(P\) and \(Q\) in
\(\scr{B}\) with \(P\subseteq S\subseteq Q\) and \(w(P)=w(Q)\).  We
extend \(w\) to such an \(S\) by setting \(w(S)=W(P)=w(Q)\) and say
that \(w\) is now defined on \(\overline{\scr{B}}\).

\subsection{Wiener measure on H\"older spaces}

We discuss the measurability of sets in H\"older spaces under \(w\).
We will make use of facts about H\"older spaces from Section 23 of
\cite{MR0257686}.

\subsubsection{Results from \cite{MR0257686}}

The discussion so far has been in \(C_0(I)\) and any \(x\in C_0(I)\)
has \(x(0)=0\).  In \cite{MR0257686}, the discussion is in \(C(I)\),
the space of all continuous real valued functions on \(I\) with no
restriction on the value at 0.  This will cause no great difficulty.

To parallel the definition of \(C_\delta\) in \ref{HolderSupport},
define 
\[C^\delta = \{x\in C(I)\mid \exists K \,\,\mathrm{s.t.}
\,\, \forall t,s, \,\, |x(t)-s(s)| \le K|t-s|^\delta \}.
\]
Note that the only difference is the shift from \(C_0(I)\) to
\(C(I)\).  

For \(x\in C^\delta\), define
\mymargin{HolderCDef}\begin{equation}\label{HolderCDef}
n_\delta = \sup_{0\le s<t\le1} \frac
{|x(t)-x(s)|}
{|t-s|^\delta},
\end{equation}
define
\[
\|x\|_\delta = \max(\|x\|_\infty, n_\delta(x)),
\]
and define
\[
\|x\|'_\delta = \max(|x(0)|, n_\delta(x)).
\]

Let
\[
\begin{split}
\Lambda^\delta
=
\{x\in C^\delta \mid
&\forall \epsilon>0, \, 
\exists \delta> 0, \,
 \,\,\mathrm{s.t.}
\\
&\forall t,s, (|t-s|< \delta 
\implies
|x(t)-x(s)| \le \epsilon|t-s|^\delta)
\}.
\end{split}
\]

For \(0< \delta < \gamma < 1\) the containments \[C^\gamma \subseteq
\Lambda^\delta \subseteq C^\delta \subseteq C(I)\] are
straightforward to verify.  It is stated in 23B of \cite{MR0257686}
that the metrics obtained from \(\|x\|_\delta\) and
\(\|x\|'_\delta\) are topologically equivalent,
that the metric defined by \(\|x\|_\delta\) is complete (a fact that
we probably do not need), and that \(\Lambda^\delta\) is a closed
subspace of \(C^\delta\).  That \(\Lambda^\delta\) is closed in
\(C^\delta\) is left in \cite{MR0257686} as an exercise for the
reader.

A map from \(C^\delta\) to \(l^\infty(\R)\) is defined in 23F of
\cite{MR0257686} as follows.  Let \({S_1}\) be the collection of
intervals in \([0,1]\) of the form 
\[
\left[
\frac j{2^n}, \frac {j+1}{2^n}
\right]
\]
where \(j\) and \(n\) are integers.  The set \({S_1}\) is
countable.  For an interval \(J\in S_1\) let \(aJ\) denote its left
endpoint, \(bJ\) denote its right endpoint and \(mJ\) denote its
midpoint.

We let \(S=\{0,1\} \cup S_1\), and for an \(x\in C^\delta\), we define
\(T_\delta x:S\rightarrow \R\) by 
\[
\begin{split}
(T_\delta x)(0) &= x(0), \\
(T_\delta x)(1) &= x(1), \\
(T_\delta x)(J) &= 
\frac
{2x(mJ) - x(aJ) - x(bJ)}
{2|mJ-aJ|^\delta}, 
\qquad\mathrm{for}\quad J\in S_1.
\end{split}
\]

The theorem in 23F of \cite{MR0257686} states that \(T_\delta\) is a
linear homeomorphism from \(C^\delta\) to \(l^\infty(\R)\) and its
restriction to \(\Lambda^\delta\) is a linear homeomorphism to the
subspace \(c_0(\R)\) consisting of the null sequences in
\(l^\infty(\R)\).

\subsubsection{Applying the results from
\cite{MR0257686}}\mylabel{HFactsApplied}

We note that \(C_\delta = C_0(I)\cap C^\delta\) and we define 
\[\Lambda_\delta = C_0(I)\cap \Lambda^\delta.\]
The elements of \(C_\delta\) are exactly those \(x\in C^\delta\) for
which \((T_\delta x)(0)=0\) and the elements of \(\Lambda_\delta\)
are exactly those \(x\in \Lambda^\delta\) for which \((T_\delta
x)(0)=0\).  From this we see that \(C_\delta\) is homeomorphic to
\(l^\infty(\R)\) and that \(\Lambda_\delta\) is homeomorphic to
\(c_0(\R)\).

For \(0< \delta < \gamma < 1\) we have the containments \[C_\gamma
\subseteq \Lambda_\delta \subseteq C_\delta \subseteq C_0(I).\]

For \(0< \delta < \gamma < \frac12\) we know that \(w(C_\gamma) =
w(C_\delta) = 1\), so \(w(\Lambda_\delta) = 1\).

The space \(c_0(\R)\) is separable, so \(\Lambda_\delta\) is
separable.  

\subsubsection{Balls in H\"older space}\mylabel{BallMeas}

We use the norm \(\|x\|'_\delta\) as it is more convenient for
this discussion.

The closed ball \(B_\delta(x;r) = \{y\in C_\delta\mid \|y-x\|'_\delta
\le r\}\) in \(C_\delta\) is simply the set of \(y\in C_\delta\)
with \(n_\delta(y-x)\le r\) since \(x(0)=y(0)=0\).  Note that from
the continuity of the elements of \(C_\delta\), the values of
\(n_\delta\) remain the same if the supremum in \tref{HolderCDef} is
taken over pairs rationals in \([0,1]\).  From this it follows that
the closed ball \(B_\delta(0;r)\) of radius \(r\) about 0 is the
intersection of sets of the form
\[
C(s,t;r) = \{y\in C_\delta \mid |y(t)-y(s)| \le r|t-s|^\delta\}
\]
where \(s\) and \(t\) are in \(\Q\cap [0,1]\).  For each
\(C(s,t;r)\), the \(s\) and \(t\) are fixed and \(C(s,t;r)\) is just
\(\mathrm{Cyl}((s,t),E)\cap C_\delta\) where \(E\) is the Borel
set \(\{(p,q)\in \R^2\mid |p-q| \le r|t-s|^\delta\}\).  From
\ref{HolderSupport}, we know that \(C_\delta\) is Borel in
\(C_0(I)\), so we have that \(B_\delta(0;r)\) and thus all closed
balls in \(C_\delta\) are Borel in \(C_0(I)\).

\subsubsection{Borel sets in H\"older space}

We will have that all Borel sets in \(C_\delta\) are
\(w\)-measurable if the open sets in \(C_\delta\) are.  That is, we
must show that any open \(U\) in \(C_\delta\) is in
\(\overline{\scr{B}}\) with \(\overline{\scr{B}}\) as defined in
\ref{WienerComplete}.

From \ref{HFactsApplied}, we know that \(\Lambda_\delta\) is in
\(\overline{\scr{B}}\), and so from \ref{BallMeas} the closed balls
in \(\Lambda_\delta\) are in \(\overline{\scr{B}}\).  However,
\(\Lambda_\delta\) is separable and so every open set in
\(\Lambda_\delta\) is a countable union of closed balls in
\(\Lambda_\delta\).  Thus every open set in \(\Lambda_\delta\) is in
\(\overline{\scr{B}}\).

Now if \(U\) is open in \(C_\delta\), then \(U\) is the disjoint
union of \(U\cap \Lambda_\delta\) and \(U-\Lambda_\delta\).  The
first is in \(\overline{\scr{B}}\) by the previous paragraph and the
second is in \(\overline{\scr{B}}\) since it is contained in the set
\(C_\delta - \Lambda_\delta\) which has measure zero.

\subsection{The measure on
\protect\(\mathrm{Diff}_+^1(I)\protect\)}\mylabel{MeasureOnDiff} 

We now use the homeomorphism \(A:\mathrm{Diff}_+^1(I) \rightarrow
C_0(I)\) from \tref{Forward} to define the measure \(\nu\) on
\(\mathrm{Diff}_+^1(I)\) by setting 
\[
\nu(X) = w(AX).
\]

We let \(E_\delta\) denote \(\mathrm{Diff}_+^{1, \delta}(I)\).  This
is to agree with the notation introduced in Page 8 of
\cite{shavgulidze}.  We next work to show that the restriction of
\(A\) to \(E_\delta\) is a homeomorphism onto \(C_\delta\).

\subsubsection{}\mylabel{DiffHldrSupp}

Recall that \(E_\delta=\mathrm{Diff}_+^1(I) \cap C_0^{1,
\delta}(I)\) where \(\mathrm{Diff}_+^1(I)\) consists of all
diffeomorphisms of class \(C^1\) of \(I=[0,1]\) that fix \(\{0,1\}\),
and \(C_0^{1, \delta}(I)\) consists of all continuously
differentiable, real valued functions on \(I\) that fix 0 and for
which \(\|f\|_{1, \delta} = |f'(0)| + n_\delta(f')\) is finite where
\(n_\delta\) is defined in \tref{HolderCDef}.

We use \(\|f\|_{1,\delta}\) as a norm on \(E_\delta\) to give the
topology we use on \(E_\delta\).

Now \((Af)(t) = \log(f'(t)) - \log(f'(0))\).  We know that \(f'\) is
continuous and never 0 since \(f\) is a diffeomorphism, so
\[
m=\min(f') = \min\left(\frac{1}{(f^{-1})'}\right)
\]
is strictly positive.  Thus 
\[
\max(\log(f'))=1/m = \|(f^{-1})'\|_\infty.
\]

Using the norm \(\|x\|'_\delta\) in the image of \(A\) gives
\(\|Af\|_\delta = n_\delta(Af)\).  We now have
\[
\begin{split}
|(Af)(t)-(Af)(s)| 
&=
|\log(f'(t)) - \log(f'(s))| \\
&\le
\|(f^{-1})'\|_\infty n_\delta(f') |t-s|^\delta \\
&\le
\|(f^{-1})'\|_\infty \|f\|_{1, \delta} |t-s|^\delta,
\end{split}
\]
which shows that \(\|Af\|_\delta \le \|(f^{-1})'\|_\infty \|f\|_{1,
\delta}\) and the restriction of \(A\) to \(E_\delta\) carries
\(E_\delta\) into \(C_\delta\) in a neighborhood of 0 and is
continuous at 0.
The map \(A\) is not linear, but \(f\mapsto f'\) is, and the above is
sufficient to show that \(A\) carries all of \(E_\delta\) into
\(C_\delta\) continuously.

To see that \(B\) carries \(C_\delta\) continuously to \(E_\delta\),
we first note that in an argument identical to that of Lemma
\ref{PreCompCont}, we have \(\|x\|_\infty\le n_\delta(x)\) for any
\(x\in C_\delta\).  Now writing \(Ee^x\) for \(\int_0^1e^{x(s)}ds\),
we have
\[
(Bx)(t) = \frac{1}{Ee^x}\int_0^te^{x(s)}ds,
\]
so 
\[(Bx)'(t) = \frac{1}{Ee^x} e^{x(t)}.
\]
Now the fact that \(\|x\|_\infty \le n_\delta(x)\) and an argument
similar to that for \(A\) shows that \(B\) carries \(C_\delta\) into
\(E_\delta\) continously.

\subsubsection{The measure on
\protect\(E_\delta\protect\)}\mylabel{MeasureOnEDelta} 

From now on we work with a \(\delta\) for which \(0 < \delta < \frac
1 2\).  This lets us conclude that 
\mymargin{EDeltaSup}\begin{equation}\label{EDeltaSup}
\nu(E_\delta)=1
\end{equation}
and that the
Borel sets in \(E_\delta\) are measurable with respect to \(\nu\).

\section{Quasi-invariance}

[XXXXXXXXXXXXXXXXXXXXXXXXXXXXx this space reserved for material from
\cite{MR1832477}.]

\section{Three lemmas}

This section covers Lemmas 7, 8 and 9 from \cite{shavgulidze}.

Recall that \(0 < \delta < \frac 1 2\).

\subsection{Definition}

For any natural \(l\) set
\[
M_l = \int\limits_{E_\delta}(q'(1))^l \nu(dq).
\]

\begin{lemma}[S-L7] For any natural \(l\), we have
\[
M_l = \int\limits_{E_\delta}(q'(0))^l \nu(dq).
\]
\end{lemma}

\begin{proof}  Let \(x=A(q)\) with \(A:\mathrm{Diff}_+^1(I)
\rightarrow C_0(I)\) as in \tref{Forward} so that \(q=B(x)\) as in
\tref{Backward}.  This gives
\[
q'(0)= \frac{1}{\int_0^1 e^{x(t)}\, dt}, \quad
q'(1)= \frac{e^{x(1)}}{\int_0^1 e^{x(t)}\, dt}.
\]
Now
\mymargin{LemSvnI-IV}\begin{align}
M_l 
&=
\int\limits_{E_\delta}(q'(1))^l \nu(dq)\notag \\
&=
\int\limits_{\mathrm{Diff}_+^1(I)}(q'(1))^l \nu(dq)\label{LemSvnI} \\
&=
\int\limits_{C_0(I)}\left(
\frac{e^{x(1)}}{\int_0^1 e^{x(t)}\, dt}
\right)^l w(dx) \label{LemSvnII}\\
&=
\int\limits_{C_0(I)}\left(
\frac{1}{\int_0^1 e^{x(t)-x(1)}\, dt}
\right)^l w(dx) \notag\\
&=
\int\limits_{C_0(I)}\left(
\frac{1}{\int_0^1 e^{x(1-t)-x(1)}\, dt}
\right)^l w(dx) \label{LemSvnIII}\\
&=
\int\limits_{C_0(I)}\left(
\frac{1}{\int_0^1 e^{x(t)}\, dt}
\right)^l w(dx) \label{LemSvnIV}\\
&=
\int\limits_{\mathrm{Diff}_+^1(I)}(q'(0))^l \nu(dq) \notag\\
&=
\int\limits_{E_\delta}(q'(0))^l \nu(dq),\notag
\end{align}
where \tref{LemSvnI} follows from \tref{EDeltaSup}, \tref{LemSvnII}
is the 
definition of \(\nu\), \tref{LemSvnIII} is ordinary change of
variables, and \tref{LemSvnIV} is \ref{TimeReverse}.  The integral
in  \tref{LemSvnII} exists because evaluation at 1 is continuous on
\(C_0(I)\) and is thus Borel measurable.
\end{proof}

\subsection{Some calculations}

The above shows equalities between values that might be infinite.
In fact, they are finite.  Since they are used later to define
constants, it is interesting to see how big they are.

\subsubsection{A preliminary calculation}\mylabel{MomentCalc}

Consider the random variable \(\pi_s:C_0(I)\rightarrow \R\) which is
evaluation at \(s\), together with the Wiener measure on \(C_0(I)\).
That is \(\pi_s(f)=f(s)\).  Since both \(\pi_s\) and \(-\pi_s\) have
normal distributions with mean 0 and variance \(s\), we can evaluate
\[
\begin{split}
\int\limits_{C_0(I)}
\exp(-l f(s))\,w(df)
&=
\frac{1}{\sqrt{2\pi s}}
\int\limits_{-\infty}^{\infty}
\exp(-lx)\exp(-x^2/(2s))\,dx \\
&=
\frac{1}{\sqrt{2\pi}}
\int\limits_{-\infty}^{\infty}
\exp(-lu\sqrt s)\exp(-u^2/(2))\,du \\
&=
\frac{1}{\sqrt{2\pi}}
\int\limits_{-\infty}^{\infty}
\exp(-\frac12(u^2 + 2lu\sqrt s))\,du \\
&=
\frac{1}{\sqrt{2\pi}}
\int\limits_{-\infty}^{\infty}
\exp(-\frac12(u + l\sqrt s)^2+\frac{l^2s}{2})\,du \\
&=
\exp(sl^2/2)
\frac{1}{\sqrt{2\pi}}
\int\limits_{-\infty}^{\infty}
\exp(-\frac12(u + l\sqrt s)^2)\,du \\
&=
\exp(sl^2/2)
\frac{1}{\sqrt{2\pi}}
\int\limits_{-\infty}^{\infty}
\exp(-\frac12(v)^2)\,dv \\
&=
\exp(sl^2/2)
\end{split}
\]

\subsubsection{Upper bounds}

Now
\[
\begin{split}
M_l 
&=
\int\limits_{E_\delta}(q'(0))^l \nu(dq)\\
&=
\int\limits_{\mathrm{Diff}_+^1(I)}(q'(0))^l \nu(dq) \\
&=
\int\limits_{C_0(I)}\left(
\frac{1}{\int_0^1 e^{x(t)}\, dt}
\right)^l w(dx) \\
&=
\int\limits_{C_0(I)}\left(
{\int_0^1 e^{x(t)}\, dt}
\right)^{-l} w(dx) \\
&\le
\int\limits_{C_0(I)}
{\int_0^1 e^{-lx(t)}\, dt}
\, w(dx)
\end{split}
\]
from Jensen's inequality and the convexity of \(x\mapsto x^{-l}\)
for positive \(x\).  Continuing, and making use of \ref{MomentCalc},
we get
\[
\begin{split}
M_l 
&\le 
\int_0^1 
\int\limits_{C_0(I)}
e^{-lx(t)}
\, w(dx) 
\, dt
\\
&=
\int_0^1 
e^{l^2t/2}
\, dt
=e^{l^2/2}-1 \\
&\le
e^{l^2/2}.
\end{split}
\]

A constant below also uses \(I=\int\limits_{E_\delta} \int_0^1
(q'(t))^2 \,dt\, \nu(dq)\).  With \(x=A(q)\) and
\[
q'(t)= \frac{e^{x(t)}}{\int_0^1 e^{x(s)}\, ds},
\]
we have
\[
\begin{split}
I 
&=
\int\limits_{C_0(I)} 
\int_0^1 
\left(
\frac{\exp(x(t))}{\int_0^1\exp(x(s))\,ds}
\right)^2
\,dt
\,w(dx) \\
&\le
\int\limits_{C_0(I)} 
\int\limits_0^1 
\int\limits_0^1
\exp(-2(x(s)-x(t)))\,ds
\,dt
\,w(dx) \\
&=
\int\limits_0^1 
\int\limits_0^1
\int\limits_{C_0(I)} 
\exp(-2(x(s)-x(t)))
\,w(dx)
\,ds
\,dt.
\end{split}
\]
Now we know that \(x(s)-x(t)\) is normally distributed with mean 0
and variance \(|t-s|\) so a calculation as in \ref{MomentCalc} tells
us that the inner integral equals \(\exp(2|t-s|)\).  Then
straightfoward integration gives
\[
\begin{split}
I
&\le
\int\limits_0^1
\int\limits_0^t
\exp(2(t-s))\,ds\,dt
+
\int\limits_0^1
\int\limits_t^1
\exp(2(s-t))\,ds\,dt \\
&= 
\frac{e^2-1}{2}.
\end{split}
\]

\subsubsection{Lower bounds}\mylabel{LowerMBound}

Now we can estimate from below.  We have
\[
\begin{split}
M_1 
&=
\int\limits_{E_\delta}
(q'(0))\,\nu(dq) \\
&=
\int\limits_{C_0(I)}
\frac{1}
{\int_0^1 \exp(f(t))\,dt}
\,w(df) \\
&\ge \frac{1}
{\int \int_0^1 \exp(f(t))\, dt\, w(df)}
\end{split}
\]
where the last inequality follows from Jensen's theorem on
probability spaces (Theorem 5.1 on Page 132 of \cite{MR2125120}) and
the convexity of \(x\mapsto x^{-1}\) when \(x>0\).  Now from
\ref{MomentCalc}, we have
\[
\begin{split}
M_1
&\ge
\frac{1}{\int_0^1\exp(t/2)\,dt} \\
&=
\frac{1}{2(\exp(1/2)-1)} \\
&\ge \frac23.
\end{split}
\]

Now from the Lyapounov inequality (see \ref{SomeMorePrep}), we have
\(M_2\ge M_1^2\) so \(M_2>1/3\).  Now \(I\ge0\), so the definition
below of \(c_4\) in \ref{LemEightDefs} has
\[
c_4 = 1 + M_1 + M_2 + I \ge 2.
\]

\subsection{Some preparation}\mylabel{SomeMorePrep}

The next lemma will use material that is found in standard texts
on probability.  We will use \cite{MR2125120}.

We give some basics.

A probability measure \(\mu\) on a set \(S\) is a countably
additive, positive measure on a \(\sigma\)-algebra on \(S\) so that
\(\mu(S)=1\).  A random variable \(X\) is a measurable function \(X
: S\rightarrow \R\).

Using notation
\[
P(X\in E) = \int\limits_{X^{-1}(E)} 1\,d\mu
\]
for a Borel set \(E\in \R\),
we will write \(P(X>K)\) for \(P(X\in (K,\infty))\).  Obviously
\(K\le L\) implies \(P(X>K) \ge P(X>L)\), and \(X\le Y\)
a.e. implies \(P(X>K) \le P(Y>K)\).

A random variable \(X\) has finite \(m\)-th moment if \(\int
|X|^m\,dP\) is finite and we call \(\int X^m\,dP\) the \(m\)-th
moment of \(X\).  The first moment is the {\itshape expected value}
of \(X\) and is denoted \(E[X]\) or \(EX\).

From Pages 128--129 of \cite{MR2125120}: For a random variable
\(X\), we set \(\|X\|_r = (E|X|^r)^{1/r}\).  The Lyapounov
inequality is Theorem 2.5 on Page 129 of \cite{MR2125120} which says
that for \(0<r\le p\), we have \(\|X\|_r\le \|X\|_p\).

Theorem 1.1 on Page 120 of \cite{MR2125120} is Markov's inequality
which says that if
\(X\) has finite \(r\)-th moment for some \(r>0\), then 
\[
P(|X|>x) \le \frac{E|X|^r}{x^r}.
\]

If \(X\) has finite second moment then 
\[
\int (X^2-E[X])^2\,dP
\]
is finite, is called the {\itshape variance} of \(X\), and is denoted
\(\mathrm{Var}(X)\).  One can calculate 
\[
\mathrm{Var}(X) = E[X^2] - (E[X])^2 \le E[X^2].
\]

Theorem 1.4(i) on Page 121 of \cite{MR2125120} is Chebyshev's
inequality which says that if \(X\) has finite second moment and
\(x>0\), then 
\[
P(|X-EX|>x) \le \frac{\mathrm{Var}(X)}{x^2}.
\]

Random variables \(X_1\), \(X_2\), \ldots, \(X_n\) are independent
if 
\[
P\left(\bigcap_{k=1}^n \{X_k\in A_k\}\right)
=
\prod_{k=1}^n P(X_k\in A_k)
\]
for arbitrary Borel sets \(A_1, A_2, \ldots, A_k\).

Theorem 10.3 on Page 70 of \cite{MR2125120} says that for
independent random variables \(X\), \(Y\) of finite second moment, we
have 
\[
\mathrm{Var}(X+Y) =  \mathrm{Var}(X) + \mathrm{Var}(Y).
\]

We also use the special case of the Cauchy-Schwarz inequaltiy
\((a+b)^2\le 2(a^2+b^2)\) for non-negative \(a\) and \(b\).  (We
have \(0\le (a-b)^2=a^2+b^2-2ab\) implies \(2ab\le a^2+b^2\) and we
plug this into \((a+b)^2=a^2+b^2+2ab\).)

\subsection{Some definitions}\mylabel{LemEightDefs}

We use the product measure \(\nu_n = \nu \otimes \cdots \otimes
\nu\) on \(E_\delta^n = E_\delta \times \cdots \times E_\delta\).

For \(g\in \mathrm{Diff}_+^3(I)\), we let \(S_g\) be the {\itshape
Schwarzian derivative} \cite{schwarz:deriv} of \(g\) defined by 
\[
S_g(t)
=
\left(
\frac{g''(t)}{g'(t)}
\right)'
-
\frac12 
\left(
\frac{g''(t)}{g'(t)}
\right)^2
=
\frac{g'''(t)}{g'(t)}
-
\frac32 
\left(
\frac{g''(t)}{g'(t)}
\right)^2.
\]

We let \(c_4= 1 + M_1 + M_2 + \int\limits_{E_\delta} \int_0^1
(q'(t))^2 \,dt\, \nu(dq)\).

For any \(r>0\), and any \(g\in \mathrm{Diff}_+^3(I)\),
we write
\[
C_g = 1+ \max_{0\le t\le 1} 
\left(
\left| \frac{g''(t)}{g''(t)} \right| 
+
\left| \left( \frac{g''(t)}{g'(t)} \right)^2 \right|
+ 
\left| \frac{g'''(t)}{g'(t)} \right|
\right).
\]
Note that \(|S_g(t)| < (3/2)C_g\) for all \(t\).

For any \(g\in \mathrm{Diff}_+^3(I)\), any
\(\mathbf x = (x_1, \ldots, x_{n-1}) \in D_n\) (with \(x_0=0\) and
\(x_n=1\) as defined previously), any
\(\mathbf q = (q_1, \ldots, q_n)\in E_\delta^n\), 
and integer \(k\) with \(1\le k\le n\), let
\[
\begin{split}
X_k(\mathbf q) 
&= 
(x_k-x_{k-1})
\left(
\frac{g''(x_{k-1})}{g'(x_{k-1})}q'_k(0)
-
\frac{g''(x_{k})}{g'(x_{k})}q'_k(1)
\right), \\
Y_k(\mathbf q) 
&= 
(x_k-x_{k-1})^2
\int_0^1
S_g(x_{k-1} + (x_k-x_{k-1})q_k(t))(q'_k(t))^2\,dt,
\end{split}
\]
let \(f_1 = \sum_1^nX_i\), and let \(f_2 = \sum_1^n Y_i\).

For any \(\mathbf x = (x_1, \ldots, x_{n-1}) \in D_n\) (with \(x_0=0\) and
\(x_n=1\)), let 
\[
\|\mathbf x\| = \max_{1\le k\le n}(x_k-x_{k-1}).
\]

\subsection{Statement}

In the following, probabilities \(P_n\) are with respect to the
measures \(\nu_n\).

\begin{lemma}[S-L8]  Given \(\epsilon\in(0,1)\), \(g\in
\mathrm{Diff}_+^3(I)\), and \(\mathbf x\in D_n\) with \(\|\mathbf
x\|<\epsilon\), we have 
\[
P_n\Big( |f_1+f_2| > 4 c_4 C_g \sqrt[3]{\epsilon} \Big) \le
2\sqrt[3]{\epsilon}.
\]  
\end{lemma}

\begin{proof} The random variables \(X_k\) are independent since,
for each \(k\), \(X_k\) depends only on \(q_k\) and the product
measure is being used on \(E_\delta^n\).

We have 
\[
E[X_k] 
= 
(x_k-x_{k-1})
\left(
\frac{g''(x_{k-1})}{g'(x_{k-1})}
-
\frac{g''(x_{k})}{g'(x_{k})}
\right)M_1.
\]
From
\[
\left(
\frac{g''(t)}{g'(t)}
\right)'
=
\frac
{g'(t)g'''(t) -(g''(t))^2}
{(g'(t))^2}
=
\left(
\frac{g'''(t)}{g'(t)}
\right)
-
\left(
\frac{g''(t)}{g'(t)}
\right)^2
\]
and the definition of \(C_g\), we have
\[
|E[X_k]| \le (x_k-x_{k-1})^2C_gM_1,
\]
so that
\[
|E[f_1]| 
\le
M_1C_g\sum_{k=1}^n(x_k-x_{k-1})^2
\le
M_1C_g\epsilon \sum_{k=1}^n(x_k-x_{k-1})
\le
M_1C_g\epsilon 
\le
c_4C_g\epsilon.
\]

We also have
\[
\begin{split}
E[X_k^2]
&\le
(x_k-x_{k-1})^2\,
2\left(
\left(\frac{g''(x_{k-1})}{g'(x_{k-1})}\right)^2M_2
+
\left(\frac{g''(x_{k})}{g'(x_{k})}\right)^2M_2
\right) \\
&\le
(x_k-x_{k-1})^2
4C_gM_2,
\end{split}
\]
where the first line uses the special case of the Cauchy-Schwarz
inequality.  This leads to
\[
\mathrm{Var}(f_1) 
\le 
\sum_{k=1}^n
(x_k-x_{k-1})^2
4C_gM_2
\le 
\epsilon
4C_gM_2.
\]

Since \(\epsilon\in(0,1)\), we have \(\sqrt[3]\epsilon
> \epsilon\), so that the Chebyshev inequality gives
\mymargin{FOneEst}\begin{equation}\label{FOneEst}
\begin{split}
P(|f_1| > 3c_4 C_g\epsilon^{1/3})
&\le
P(|f_1-E[f_1]| + |E[f_1]| > 3c_4 C_g \epsilon^{1/3}) \\
&\le
P(|f_1-E[f_1]| + c_4 C_g\epsilon^{1/3} > 3c_4 C_g \epsilon^{1/3}) \\
&\le
P(|f_1-E[f_1]| > 2c_4 C_g \epsilon^{1/3}) \\
&\le
\frac{4c_4C_g\epsilon}{4c_4^2C_g^2\epsilon^{2/3}} \\
&\le
\frac{\epsilon^{1/3}}{c_4C_g} \le\frac 1 2 \sqrt[3]\epsilon
\end{split}
\end{equation}
where we have used our estimate \(c_4\ge 2\) from \ref{LowerMBound}.

We next work with \(f_2\).  We have
\[
S_g(x_{k-1} + (x_k-x_{k-1})q_k(t)) = S_g(q_k(t))
\]
from direct checking or the fact that
the Schwarzian derivative  is invariant
under composition with linear fractional transformations (Property 1
on Page 35 of \cite{schwarz:deriv}).

This gives
\[
|f_2(\mathbf q)| \le (3/2)C_g\sum_{k=1}^n (x_k-x_{k-1}) \int
(q'_k(t))^2\, dt.
\]
From this
\[
\begin{split}
E[|f_2|]
&\le
(3/2)C_g\epsilon
\int
\int_0^1
(q'(t))^2 \, dt \,\nu(dq) \\
&\le
(3/2)c_4 C_g \epsilon.
\end{split}
\]
By the Markov inequality with \(r=1\), we have
\mymargin{FTwoEst}\begin{equation}\label{FTwoEst}
P(|f_2| > c_4C_g)
\le 
\frac
{(3/2)c_4 C_g \epsilon}
{c_f C_g\epsilon^{1/3}}
\le \frac 3 2 \epsilon^{2/3}
\le \frac 3 2 \epsilon^{1/3}.
\end{equation}

Now if \(|f_1+f_2|> 4t\), then at least one of \({|f_1|}>3t\) or
\(|f_2|>t\) holds, so 
\[
P(|f_1+f_2| > 4t) \le P(|f_1| > 3t) + P(|f_2|>t)
\]
for any \(t\).  Using \(t=c_4C_g \epsilon^{1/3}\) 
with
\tref{FOneEst} and \tref{FTwoEst} gives the conclusion of the
lemma.
\end{proof}

\subsection{Some notation}\mylabel{LemNinePrep}

The next lemma compares the ``equality'' of a partition \(\mathbf x\in
D_n\) to the ``equality'' of its image \(g(\mathbf x)\) under a
\(g\in \mathrm{Diff}_0^3(I)\).  The equality will be measured by
the function \(v\) that appears in \tref{WhatVIs}.  A typical
argument for \(v\) is of the form 
\[
R_{\mathbf x,k}
=
\frac
{x_k-x_{k-2}}
{2\sqrt{(x_k-x_{k-1}) (x_{k-1}-x_{k-2})}},
\]
and we will be comparing this to the expression 
\[
R_{\mathbf x,k}^g
=
\frac
{g(x_k)-g(x_{k-2})}
{2\sqrt{(g(x_k)-g(x_{k-1})) (g(x_{k-1})-g(x_{k-2}))}}.
\]
We will use the notation \(R_{\mathbf x, k}\) and \(R_{\mathbf x,
k}^g\) often without the subscripts since they will be clear from
the context.

The comparison will be by way of the quadratic Taylor polynomial.
To keep notation under control, we will write out what is needed and
introduce notation for various parts.

For \(\mathbf x\in D_n\), we have the usual conventions that
\(x_0=0\), \(x_n=1\) and \(x_{-1}=x_{n-1}-1\).  We let \(g(\mathbf
x)\) be the image of \(\mathbf x\) in \(D_n\) and write \(g(x_{-1})
= g(x_{n-1})-1\).  Of course \(g(x_0)=0\) and \(g(x_n)=1\).

For each \(k\) with \(2\le k\le n\), there are
\(x_{k-1}^*\in (x_{k-1}, x_k)\) and 
\(x_{k-1}^{**}\in (x_{k-2}, x_{k-1})\), and also 
\(x_0^*\in (0,x_1)\) and \(x_0^{**}\in (x_{n-1}, 1)\) so that
\[
\begin{split}
g(x_k)-g(x_{k-1}) 
&=
g'(x_{k-1})(x_k-x_{k-1}) 
+ 
\frac 1 2 g''(x_{k-1}^*)(x_k-x_{k-1})^2
\\
&=
g'(x_{k-1})(x_k-x_{k-1})
(1 + \frac{g''(x_{k-1}^*)}{2g'(x_{k-1})} (x_k-x_{k-1})), \\
g(x_{k-1})-g(x_{k-2}) 
&=
g'(x_{k-1})(x_{k-1}-x_{k-2}) 
+ 
\frac 1 2 g''(x_{k-1}^{**})(x_{k-1}-x_{k-2})^2
\\
&=
g'(x_{k-1})(x_{k-1}-x_{k-2})
(1 + \frac{g''(x_{k-1}^{**})}{2g'(x_{k-1})} (x_{k-1}-x_{k-2})).
\end{split}
\]
The different treatment of \(x_0^{**}\) is because \(g\) is not
defined below 0, and the interval \([x_{-1}, 0]\) really just
mirrors the interval \([x_{n-1}, 1]\).

For a given \(k\), we let
\[
\begin{split}
a_k
&=x_k-x_{k-1}, \\
b_k
&=x_{k-1}-x_{k-2}, \\
A_k
&=
\frac{g''(x_{k-1}^*)}{2g'(x_{k-1})}, \\
B_k
&=
\frac{g''(x_{k-1}^{**})}{2g'(x_{k-1})}.
\end{split}
\]
Note that we already dispense with the dependence on \(\mathbf x\).

We thus have
\[
\begin{split}
g(x_k)-g(x_{k-1})
&=
g'(x_{k-1})a_k(1+A_ka_k), \\
g(x_{k-1})-g(x_{k-2})
&=
g'(x_{k-1})b_k(1+B_kb_k).
\end{split}
\]

We have 
\[
R_{\mathbf x, k} 
=
\frac
{a_k + b_k}
{2\sqrt{a_kb_k}}
.
\]
Since everything has subscript \(k\), we supress that subscript.
Since a fixed \(\mathbf x\) is involved we supress that as well.
This gives
\[R=\frac{a+b}{2\sqrt{ab}}.\]

Now write \(g'\) for \(g'(x_{k-1})\), and we have
\[
\begin{split}
R^g
&=
\frac
{g' a(1+Aa) + g' b(1+Bb)}
{2\sqrt{g' a(1+Aa)}\sqrt{g' b(1+Bb)}}
\\
&=
\frac
{a+Aa^2 + b+Bb^2}
{2\sqrt{ab}\sqrt{(1+Aa)(1+Bb)}},
\\
\end{split}
\]
so that 
\mymargin{RRatio}\begin{equation}\label{RRatio}
\begin{split}
\frac{R^g}{R}
&=
\frac{1}{a+b} 
\left(
\frac
{a+ b+Aa^2 +Bb^2}
{\sqrt{(1+Aa)(1+Bb)}}
\right) \\
&=
\frac
{1 + \frac{Aa^2 +Bb^2}{a+b}}
{\sqrt{(1+Aa)(1+Bb)}}.
\end{split}
\end{equation}

\subsection{Log and exp estimates}\mylabel{LogEst}

The ratio \tref{RRatio} will be further modified to take the form
\(1+y\) and a product of such ratios will be studied by looking at
the sum of the logs.  Thus expressions of the form \(\log(1+y)\)
will occur.

The function \(\log(1+y)\) is concave down and has slope 1 at 0, so
for all \(y\), we have \(\log(1+y)<y\).  However, we need more to
estimate \(|\log(1+y)|\).  We have the chord from \((K,\log(1+K))\)
to \((0,0)\) has slope
\[
\frac{\log(1+K)}{K}
\]
whether \(K\) is negative or positive.  The graph of \(\log(1+y)\)
will lie above this chord on the interval with endpoints \(0\) and
\(K\).  For \(K=-\frac12\), this slope is \(1.3\dots\), and for
\(K=\frac12\), the slope is \(0.8\dots\).  Thus we are safe in
estimating 
\[
|\log(1+y)|\le 2|y|
\]
for \(|y|\le \frac12\).

Below, we will encounter \(\log(1+3y\).  The slopes of the chords
are \(3.5\dots\) for \(K=-\frac1{10}\) and \(2.6\dots\) for
\(K=\frac1{10}\).  Thus we can use 
\[
|\log(1+3y)| \le 4|y|
\]
for \(|y|\le 0.1\).

We will also want to known \(|e^x-1|\).  On \([-1,1]\), this is
bounded by \(K|x|\) with \(K=e^1-1\).  Thus we can say \(|e^x-1|\le
2|x|\) for \(x\in [-1,1]\).  

\subsection{The lemma}

To keep things simple, we use the notation from \ref{LemNinePrep} as
well as \(\|\mathbf x\|\) from \ref{LemEightDefs}.

\begin{lemma}[S-L9] For any \(g\in \mathrm{Diff}_0^3(I)\) and any
\(\epsilon>0\), there are \(r>4\) and \(\delta_1\in(0,1)\) so
that for any \(\mathbf x\in D_n\) satisfying \(\|\mathbf x\| <
\delta_1\) and \(\displaystyle{\min_{1\le k\le n}} (R_{\mathbf x,
k}) > r\), we have
\[
\left| 
\left(
\prod_{k=1}^n \frac{R_{\mathbf x,k}^g}{R_{\mathbf x,k}} 
\right)
-1
\right|
\le \epsilon.
\]
\end{lemma}

\begin{proof}  We continue to use notation from \ref{LemNinePrep}.

To make estimates at the end easier, we take \(\epsilon\in (0,10)\),
and we let  
\[
\begin{split}
C
&=
\max_{t_1, t_2\in [0,1]} \left|\frac{g''(t_1)}{g'(t_2)} \right|, \\
\delta_1
&=
\frac{1}{400(C+1)}, \qquad\mathrm{and} \\
r
&=
\exp(8000(C+1)/\epsilon).
\end{split}
\]

Take an \(\mathbf x\in D_n\) satisfying the hypotheses of the
lemma.  We temporarily fix a \(k\) with \(1\le k\le n\) and
temporarily drop subscripts.  Until we reintroduce \(k\), what
follows applies to any \(k\).

We have 
\[
\frac{R^g}{R}
=
\frac
{1 + \frac{Aa^2 +Bb^2}{a+b}}
{\sqrt{(1+Aa)(1+Bb)}}.
\]

We rewrite this as 
\[
R^g = R(1+\lambda(a+b))
\]
where
\[
\lambda=\frac{1}{a+b} \left(
\frac
{1 + \frac{Aa^2 +Bb^2}{a+b}}
{\sqrt{(1+Aa)(1+Bb)}} -1
\right)
\]

We have both \(|A|\) and \(|B|\) are bounded by \(C/2\) and both \(|a|\)
and \(|b|\) are bounded by \(\delta_1\).  We also have
\(C\delta_1<1/400\).  Let \(d=a+b\).

Now 
\[
\left| \frac{Aa^2 +Bb^2}{a+b}  \right|
< 
\left|A\frac{a^2}{a}\right| + 
\left|B\frac{b^2}{b}\right|
=
|Aa|+|Bb|
<
Cd.
\]
We have that both the numerator and denominator of 
\[
\frac
{1 + \frac{Aa^2 +Bb^2}{a+b}}
{\sqrt{(1+Aa)(1+Bb)}}
\]
lie between \(1-Cd\) and \(1+Cd\).  Thus 
\[
\frac{-2C}{1+Cd} =
\frac 1 d \left(
\frac{1-Cd}{1+Cd} -1\right) 
< 
\lambda
<
\frac 1 d \left(
\frac{1+Cd}{1-Cd} -1\right) 
=\frac{2C}{1-Cd}
\]
and since \(Cd< 1/200\) implies \(1-Cd > 4/5\), we get
\(|\lambda|<(5/2)C\). 

To agree with the notation of \cite{shavgulidze}, we let
\(t=R_{\mathbf x,k}\) and we let \(\alpha=\lambda (x_k-x_{k-2}) =
\lambda d\).  This gives us 
\mymargin{SetUpForLog}\begin{equation}\label{SetUpForLog}
R_{\mathbf x,k}(1+\alpha) = t(1+\alpha) = R^g_{\mathbf x,k}
\end{equation}
and
we get the estimate 
\mymargin{EstForLog}\begin{equation}\label{EstForLog}
|\alpha| < (5/2)Cd < 1/80.
\end{equation}
Note that \(t\ge
r>4\).

We let 
\(\tau = \log(t+\sqrt{t^2-1}) > \log(t) > 1\), and we let
\[
\beta = \log
\left(
1 + 
\alpha
\left(
\frac{t}{t+\sqrt{t^2-1}}
\right)
\left(
1 + 
\frac
{(2+\alpha)t}
{\sqrt{t^2-1} + \sqrt{t^2(1+\alpha)^2 -1}}
\right)
\right).
\]

Note that 
\[
\begin{split}
&(t+\sqrt{t^2-1})
\left(
1 + 
\alpha
\left(
\frac{t}{t+\sqrt{t^2-1}}
\right)
\left(
1 + 
\frac
{(2+\alpha)t}
{\sqrt{t^2-1} + \sqrt{t^2(1+\alpha)^2 -1}}
\right)
\right) \\
=&
t+\sqrt{t^2-1} + 
\alpha
{t}
\left(
1 + 
\frac
{(2+\alpha)t}
{\sqrt{t^2-1} + \sqrt{t^2(1+\alpha)^2 -1}}
\right) \\
=&
t+\alpha t + \sqrt{t^2-1} + 
\frac
{(2\alpha t^2+\alpha^2 t^2)
(\sqrt{t^2(1+\alpha)^2 -1} - \sqrt{t^2-1})}
{({t^2(1+\alpha)^2 -1}) - ({t^2-1})}
 \\
=&
t+\alpha t + \sqrt{t^2-1} + 
\frac
{(2\alpha t^2+\alpha^2 t^2)
(\sqrt{t^2(1+\alpha)^2 -1} - \sqrt{t^2-1})}
{(2\alpha t^2+\alpha^2 t^2)}
 \\
=&
t+\alpha t +
\sqrt{t^2(1+\alpha)^2 -1}.
\end{split}
\]

From this we see that 
\[
\tau+ \beta = \log(t(1+\alpha) + \sqrt{(t(1+\alpha))^2-1}),
\]
and that 
\[
v(t(1+\alpha)) = v_1(\tau+\beta).
\]

We can estimate \(\beta\).  Using \(t>4\), we get
\[
\frac 1 2 < \frac t {t+\sqrt{t^2-1}} < \frac 3 4.
\]
Taking into accont \(|\alpha|< \frac{1}{80}< 0.1\), we get 
\[
0.9 < \frac
{(2+\alpha)t}
{\sqrt{t^2-1} + \sqrt{t^2(1+\alpha)^2 -1}}
< 2.
\]
Thus \[
\beta=\log(1+K\alpha)
\]
for a \(K\) that lies between \(0.9\) and 3. From \ref{LogEst} and
the fact \(|\alpha|\le \frac1{80}\) from \tref{EstForLog}, we have
\[
\frac 1 4|\alpha| < |\beta| < 4|\alpha|.
\]
Thus \(|\beta|<1/20\).

We can say more about \(\beta\).  We have
\[
\frac 1 4 |\lambda(x_k-x_{k-2})| = \frac 1 4|\alpha|\le |\beta| \le
4|\alpha| = 4|\lambda(x_k-x_{k-2})|
\]
so
\[
\left|\frac{\beta}{x_k-x_{k-2}}\right| \le 4|\lambda|<16C.
\]

Since \(v(t(1+\alpha)) = v_1(\tau+\beta)\), there is a \(\theta\in
(0,1)\) so that 
\[
v(t(1+\alpha)) = v_1(\tau+\beta) = v_1(\tau) +
v'_1(\tau + \theta\beta) \beta.
\]

From Lemma S-L2, we have
\[
|v'_1(\tau + \theta\beta)| \le
\frac{4}{\tau+\theta\beta}v_1(\tau+\theta\beta).
\]
Since \(v_1\) takes only positive values, since \(\tau>1\), and since
\(|\theta\beta|+\frac{1}{20}\), we get from the mean value theorem
that for some \(\theta_1\in (0,\theta)\)
\[
\begin{split}
|v_1(\tau+\theta\beta)| 
&\le 
v_1(\tau) + |v'_1(\tau+\theta_1\beta) (\theta\beta)| \\
&\le
v_1(\tau) + 
\left| 
\frac {4(\theta\beta)}{\tau+\theta_1\beta}
\right|
v_1(\tau+\theta_1\beta) \\
&\le
v_1(\tau) + 
\left| 
\frac {4}{20(.95)}
\right|
v_1(\tau+\theta_1\beta) \\
&\le
v_1(\tau) + 
\left| 
\frac {1}{4}
\right|
v_1(\tau+\theta_1\beta).
\end{split}
\]
Applying this analysis to \(v_1(\tau+\theta_1\beta)\) and
continuing, we get
\[
|v_1(\tau+\theta\beta)|
\le 
v_1(\tau) \sum_{i=0}^\infty \frac{1}{4^i}
=
\frac 4 3 v_1(\tau).
\]
Now
\[
|v'_1(\tau+\theta\beta)|
\le
\frac{4}
{\tau(1+(\theta\beta)/\tau)}
\frac 4 3 v_1(\tau)
\le \frac{10}{\tau} v_1(\tau).
\]

From 
\[
v_1(\tau+\beta) = v_1(\tau) + v'_1(\tau+\theta\beta)\beta,
\]
we have
\[
\frac{v_1(\tau+\beta)}{v_1(\tau)}
=
1+\left(\frac{\tau v'_1(\tau+\theta\beta)}{v_1(\tau)}\right)
\left(\frac \beta \tau \right)
\]
where
\[
\left|
\frac{\tau v'_1(\tau+\theta\beta)}{v_1(\tau)}
\right|
\le 10.
\]

Let \(\omega(\alpha, t)\) be such that 
\[
\frac{v_1(\tau+\beta)}{v_1(\tau)}
=
1+\omega(\alpha, t) \frac \beta \tau.
\]

We reintroduce \(k\).  Thus everything gets a subscript of \(k\) to
remind us that the values depend on \(k\)  Recall that 
\[
v_1(\tau_k+\beta_k) = v(t_k(1+\alpha_k)) = v(R_{\mathbf x, k}^g)
\]
and
\[
v_1(\tau_k) = v(t_k) = v(R_{\mathbf x, k}).
\]
We wish to express 
\[
\frac{v_1(\tau_k+\beta_k)}{v_1(\tau_k)}
\]
as \(1+\omega_kE_k\) where \(E_k\) depends on \(k\) in a way we can
control and \(\omega_k\) has an estimate that does not depend on
\(k\).

We write
\[
\frac{v_1(\tau_k+\beta_k)}{v_1(\tau_k)}
=
1+\omega(\alpha_k,t_k)
\frac{\beta_k\log(t_k)}{\tau(x_k-x_{k-2})}
\frac{x_k-x_{k-2}}{\log(t_k)}
\]
and we define
\[
\omega_k = \omega(\alpha_k,t_k)
\frac{\beta_k\log(t_k)}{\tau(x_k-x_{k-2})}
\]
We know \(|\omega(\alpha_k, t_k)|<10\), \(\tau_k>\log(t_k)>1\), and
\(|\beta_k/(x_k-x_{k-2})|<16C\) from above.  So we have
\(|\omega_k|\le 200C\).

Recall also that \(t_k=R_{\mathbf x, k}>r\), so \(\log(t_k)>
\log(r)\).  We can set
\[
\begin{split}
\sigma=\log
\left(
\prod_{k=1}^n \frac{R_{\mathbf x,k}^g}{R_{\mathbf x,k}} 
\right)
&=
\sum_{k=1}^n
\log\left(
\frac{R_{\mathbf x,k}^g}{R_{\mathbf x,k}} 
\right) \\
&=
\sum_{k=1}^n
\log\left(
\frac{v_1(\tau_k+\beta_k}{v_1(\tau_k)} 
\right) \\
&=
\sum_{k=1}^n
\log\left(
1+\omega_k \frac{x_k-x_{k-2}}{\log(t_k)}
\right).
\end{split}
\]

Now we are using
\[
\frac{R_{\mathbf x,k}^g}{R_{\mathbf x,k}} 
=
1+\omega_k \frac{x_k-x_{k-2}}{\log(t_k)}.
\]
From \tref{SetUpForLog} and \tref{EstForLog}, we have
\[
\left|
\omega_k \frac{x_k-x_{k-2}}{\log(t_k)}
\right|
\le \frac1{80}.
\]
Thus by \ref{LogEst}, we have
\[
\left|
\log\left(
1+\omega_k \frac{x_k-x_{k-2}}{\log(t_k)}
\right)
\right|
\le
2
\left|
\omega_k \frac{x_k-x_{k-2}}{\log(t_k)}
\right|
.
\]

We now get
\[
\begin{split}
|\sigma|
&\le
2\sum_{k=1}^n |\omega_k|\frac{x_k-x_{k-2}}{\log(t_k)} \\
&\le
\frac{400C}{\log(r)}
\sum_{k=1}^n {x_k-x_{k-2}} \\
&\le
\frac{800C}{\log(r)} \\
&\le 
\frac{\epsilon}{10}
\end{split}
\]
where we know \(\epsilon/10\in (0,1)\).

Now 
\[
\exp(-\epsilon/10)-1 
\le 
\left(
\prod_{k=1}^n \frac{R_{\mathbf x,k}^g}{R_{\mathbf x,k}} 
\right)-1 
\le 
\exp(\epsilon/10)-1,
\]
so from \ref{LogEst} we get
\[
\left|
\left(
\prod_{k=1}^n \frac{R_{\mathbf x,k}^g}{R_{\mathbf x,k}} 
\right)
-1
\right|
\le
2\frac{\epsilon}{10}< \epsilon.
\]
This completes the proof.
\end{proof}

\section{The third theorem}

The last theorem combines all the lemmas that have come before.  We
need some preliminary discussion.

\subsection{Preliminary discussion}

Let \(\varphi\) be in \(\mathrm{Diff}_+^{1, \delta}\) and let \(J\)
and \(K\) be two closed intervals in \(\R\) that have non-empty
interior.  Then \((\varphi; J, K)\) will represent the
``affine distortion'' of \(\varphi\) that takes \(J\) to \(K\).
Specifically, if \(J\) has left endpoint \(x\) and length \(j\) and
\(K\) has left endpoint \(y\) and length \(k\), then
\[
(\varphi; J, K)(t) 
= 
y+k\left(\varphi\left(\frac{t-x}{j}\right)\right).
\]

We have
\[
\begin{split}
(\varphi;J,K)'_+(x) &= \frac k j \varphi'_+(0) \\
(\varphi;J,K)'_-(x+j) &= \frac k j \varphi'_-(1)
\end{split}
\]
where \(\varphi'_+\) and \(\varphi'_-\) are right and left hand
derivatives, respectively.

Given a pair \((x,y)\), a \(\varphi\), and an \(m>0\) which is a
desired slope at \(x\), then we can accomplish either of two tasks.
If we are also given the length \(j\) of \(J\), we can find a length
\(k\) for \(K\) so that \((\varphi; J,K)'_+(x)=m\), or if we are
given the length \(k\) of \(K\), we can find a length \(j\) for
\(J\) with the same result.  It is clear that each problem has a
unique solution for a given set of data.  We will be concerned with
the second problem (the length of the range interval is given and we
figure out the length of the domain interval).

Recall the convention that if \(\mathbf x= (x_1, \ldots, x_{n-1})\)
is in \(D_n\), then \(x_0=0\) and \(x_n=1\) are assumed.

Now let \(\mathbf y= (y_1, \ldots, y_{n-1})\) be from \(D_n\) and
let \(\mathbf \varphi = (\varphi_1, \ldots, \varphi_n)\) be from
\(E_\delta^n\).

\begin{lemma}\label{HolderStitch}
Given \(\mathbf y\) and \(\varphi\) as above, there is
a unique \(\mathbf x\in D_n\) so that there is a function
\(\overline{\varphi}\) in \(E_\delta\) so that for each \(i\) with
\(1\le i \le n\), the restriction of \(\overline{\varphi}\)  to
\([x_{i-1}, x_i]\) is 
\[
(\varphi_i; [x_{i-1}, x_i],\, [y_{i-1}, y_i]).
\]
\end{lemma}

\begin{proof} The important point is that the derivatives exist at
the \(x_i\).  We start with the wrong \(x_1\) and then fix it later.

Let \(x_1\) be arbitrary.  

Now there is a unique \(x_2\) so that
\((\varphi_2; [x_1, x_2],\,[y_1, y_2])\) has the same derivative at
\(x_1\) as \((\varphi_1; [x_0, x_1],\,[y_0, y_1])\).  We repeat and
continue in this way increasing the subscripts by one at each
repetition.  In the end, we get an \(x_n\) that is probably not
equal to 1.  However, it is clear that \(x_n\) is a linear function
of our initial choice of \(x_1\).  If we divide all the \(x_i\) with
\(i<n\) by \(x_n\), we get the desired \(\mathbf x\).  It is clear
that \(\overline{\varphi}\) will be in \(E_\delta\).
\end{proof}

\subsection{Definitions}

Let \(Q_n:D_n\times E_\delta^n\rightarrow E_\delta\) be the function
that takes \((\mathbf y, \varphi)\) to the function
\(\overline{\varphi}\) as given by Lemma \ref{HolderStitch}.

For \(F\in C_b(E_\delta) = C_b(\mathrm{Diff}_+^{1, \delta}(I))\) we
let 
\[
L_{\delta, n}(F)
=
\int\limits_{D_n} \int\limits_{E_\delta^n} F(Q_n(\mathbf x, \varphi)
u(\mathbf x)\, d\mathbf x\, \nu_n(d\varphi).
\]

\subsection{Statement}

\begin{thm}[S-3]\mylabel{ShavThree}  Let \(F\) be from
\(C_b(E_\delta)\), and let \(g\) be from \(\mathrm{Diff}_0^3(I)\).
Then
\[
\lim_{n\rightarrow \infty}
|L_{\delta, n}(F_g) - L_{\delta, n}(F)| = 0.
\]
\end{thm}

\newpage

\providecommand{\bysame}{\leavevmode\hbox to3em{\hrulefill}\thinspace}
\providecommand{\MR}{\relax\ifhmode\unskip\space\fi MR }
\providecommand{\MRhref}[2]{%
  \href{http://www.ams.org/mathscinet-getitem?mr=#1}{#2}
}
\providecommand{\href}[2]{#2}

\end{document}